\documentclass[11pt, a4paper, reqno]{amsart}

\usepackage[margin=2.5cm,top=4cm,bottom=3cm,footskip=1cm,headsep=2cm]{geometry}

\usepackage[english]{babel}
\usepackage[utf8]{inputenc}
\usepackage{amsmath, amssymb, amsthm, amsfonts}
\usepackage{graphicx}
\usepackage{hyperref}
\usepackage{xcolor}
\usepackage{amsaddr}

\theoremstyle{plain}
\newtheorem{thm}{Theorem}[section]
\newtheorem{lem}[thm]{Lemma}
\theoremstyle{definition}
\newtheorem{definition}[thm]{Definition}
\newtheorem{remark}[thm]{Remark}

\newcommand{\R}{\mathbb{R}} 
 
\newcommand{\Lnorm}[1]{\| #1 \|_{L^2}^2}
\newcommand{\Cp}{\ubar C_{P''}}

\newcommand{\vp}{ \nu_{\partial}}
\def\dx{\partial_x}
\def\dt{\partial_t}
\def\H{\mathcal{H}}
\def\G{\mathcal{G}}
\def\F{\mathcal{F}}
\def\E{\mathcal{E}}
\def\V{\mathcal{V}}
\def\e{\text{e}}
\usepackage{accents}
\newcommand{\ubar}[1]{\underaccent{\bar}{#1}}

\newcommand{\xip}[1]{(#1, \xi_+^{\mathcal{S}^i}(#1,x,t))}

\newcommand{\xipe}[1]{(#1, \xi_+^{\mathcal{S}^i_e}(#1,x,t))}
\newcommand{\xipte}[1]{(#1, \xi_+^{\mathcal{\tilde S}^i_e}(#1,x,t))}
\newcommand{\rhor}{\rho_{\text{ref}}}

\newcommand{\xipoe}{\xi_+^{\mathcal{S}^i_e}}
\newcommand{\xiptoe}{\xi_+^{\mathcal{\tilde S}^i_e}}
\newcommand{\tx}[1]{t_+^{\mathcal{S}^i_e}(#1,t)}

\title[Observer-based data assimilation for barotropic gas transport]{Observer-based data assimilation for barotropic gas transport using distributed measurements}
\author[J. Giesselmann, M. Gugat, T. Kunkel]{Jan Giesselmann$^1$, Martin Gugat$^2$, Teresa Kunkel$^{1,\star}$}
\address{$^1$Technische Universität Darmstadt\\ $^2$Friedrich-Alexander-Universität Erlangen-Nürnberg}

\date{April 9, 2024}  
\email{giesselmann@mathematik.tu-darmstadt.de}
\email{martin.gugat@fau.de}
\email{tkunkel@mathematik.tu-darmstadt.de}
\thanks{$^\star$ Corresponding author}

\begin{document}
\maketitle

\begin{abstract}
	We consider a state estimation problem for gas pipeline flow modeled by the one-dimensional barotropic Euler equations. In order to reconstruct the system state,
	we construct an observer system of Luenberger type based on distributed measurements of one state variable. 
	First, we show the existence of Lipschitz-continuous semi-global solutions of the observer system and of the original system for initial and boundary data satisfying smallness and compatibility conditions for a single pipe and for general networks. 
	Second, based on an extension of the relative energy method we prove that the state of the observer system converges exponentially in the long time limit towards the original system state.
	We show this for a single pipe and for star-shaped networks.
\end{abstract}

\begin{quote}	
\end{quote}

\begin{quote}
\end{quote}

\section{Introduction}
The goal of data assimilation is to estimate the current state of some system
by combining available measurement data with a physical model of the system \cite{LawStuartZygalakis_Dataassimilation2015, AschBocquetNodet_DataAssimilation_2016}.
This is important, for example, if control problems are considered that need the current system state as input data.
Here, we will consider systems that are described by nonlinear hyperbolic balance laws.
In many applications such systems are considered on networks, for a survey on balance laws on networks see \cite{BressanHerty2014}.
Examples are the Saint-Venant system modeling water flow through open channels,
the Aw-Rascle-Zhang model for traffic flow
and
the barotropic Euler equations modeling gas transport through pipes,
where our long-term goal is to study data assimilation for these problems on networks.

There are different approaches for data assimilation, which can mainly be grouped into variational approaches, statistical approaches such as Kalman filtering and observer-based approaches.
In this paper we will use the observer-based approach,
which means that an observer system is constructed that contains, in comparison to the original system, some 
additional source terms or boundary conditions accounting for the measurement data.
Then, the main goal is to show the synchronization, i.e., the convergence of the observer system towards the original system state.
Proving synchronization of observer systems can be viewed as a control problem, where the difference system has to be controlled to zero.
A particular difficulty arises from the nonlinearity of the balance laws,
since for nonlinear systems the difference system does not have the same structure as the original system.
To deal with this challenge we construct Lyapunov functions based on relative energy.

Similar to other control problems, we can distinguish between boundary observers based on boundary measurements (cf. \cite{Gugat2021} for an example for a semilinear model for gas transport),  and distributed observers, where distributed measurements are inserted into the observer system through suitable source terms.
Synchronization of a boundary observer for linear hyperbolic systems is shown in \cite{LiLu2022}.
For an overview of boundary stabilization of one-dimensional hyperbolic systems see \cite{Coron2007, Hayat2021} and references therein.
An important class of observers for distributed measurements are Luenberger-type observers \cite{Luenberger1971}, where the inserted source terms are proportional to the difference between the state of the observer system and the measurements, see e.g. \cite{BoulangerMoireau2015} for an application to the Saint-Venant system using a Luenberger-type observer for a kinetic representation.
Other works using Luenberger-type observers include 
\cite{FarhatJohnstonJollyTiti2018}, where synchronization of an observer for the 2D Bénard convection problem is investigated numerically for coarse-grained measurements of one variable,
\cite{JollyTiti2019} and \cite{ChapelleCindea2012}, where exponential synchronization of an observer for the surface quasi-geostrophic
equation and the wave equation, respectively, is shown, 
and \cite{Titi2020}, where error estimates for a mixed finite element discretization of an observer system for the Navier-Stokes equation are provided.
Luenberger observers for linear hyperbolic systems of second order are investigated in \cite{CindeaImperialeMoireau2015} and \cite{ChapelleCindeaMoireau2012}, where in the latter work measurement data are incorporated into the discretization of wave-like equations in order to guarantee bounded approximation errors.
Joint state-parameter estimation for a front-tracking problem  using a Luenberger observer for state estimation and a Kalman filter for parameter estimation is studied in \cite{RochouxCollin2018}
and in \cite{BENABDELHADI_Krstic_2021} simultaneous state and parameter estimation is investigated for wave equations and exponential synchronization is shown.
Luenberger observers can also be designed for boundary measurements, see e.g. \cite{Aamo06}, where a boundary Luenberger observer is introduced for a linearized model of pipeline flow.

In the following, we will consider observer-based data assimilation for systems modeled by the one-dimensional barotropic Euler equations
\begin{align}
	\dt \rho + \dx m &= 0 \label{eq:Euler1}\\
	\dt m + \dx\left(\frac{m^2}{\rho} + p(\rho) \right) &= -\gamma \frac{|m| m}{\rho}\label{eq:Euler2}
\end{align}
for $ 0<x<\ell,\ t>0$ 
allowing for a nonlinear friction term on the right-hand side.
Here, the density $\rho$ and the mass flow rate $m=\rho v$ with the velocity $v$ are the unknowns and the system is complemented by a strictly monotone pressure law  $p(\rho)$ 
and a friction coefficient $\gamma\ge 0$.
This system describes gas transport through gas pipes, where the friction term results from the friction at the pipe walls (cf. \cite{BrouwerGasserHerty2011}). 
Let us note that for a pressure law given by $p(\rho)=\tfrac{g}{2} \rho^2$ with $g$ the gravity of earth the system (without friction) is equivalent to the Saint-Venant system  for flat bathymetry  and the system is also similar to the  Aw-Rascle-Zhang model (see e.g.  \cite{YuBayenKrstic2019}) for traffic flow.

In order to investigate the convergence of the state of the observer system towards the original solution, we have to bound the difference between two solutions. There are different frameworks for this in the case of systems of nonlinear hyperbolic balance laws: The $L^1$-stability framework by Bressan \cite{bressan2005hyperbolic}, which needs a small total variation condition, and the relative energy framework, which we will use in the following.
The relative energy defined with respect to conservative variables was introduced in \cite{Dafermos1979} 
and was applied to a variety of thermo-mechanical theories of fluids in \cite{GiesselmannLattanzioTzavaras2017}. If systems on networks are studied, it turns out, see \cite{EggerGiesselmann}, that it is more suitable to use the relative energy with respect to the
non-conservative variables $(\rho, v)$. Therefore, we rewrite
system \eqref{eq:Euler1}-\eqref{eq:Euler2} in the Hamiltonian form 
\begin{align}
	\dt \rho + \dx m &= 0 \label{eq:system_1}\\
	\dt v + \dx h &= -\gamma|v|v \label{eq:system_2} 
\end{align}
with the specific enthalpy $h(\rho, v):=\tfrac{1}{2} v^2 + P'(\rho)$, a smooth and strictly convex pressure potential $P=P(\rho)$, connected to the pressure law by $p'(\rho) = \rho P''(\rho)$,
and associated energy functional
\begin{align} \label{eq:H}
	\mathcal{H}(\rho,v):=\int_{0}^{\ell} \left(\tfrac{1}{2}  \rho v^2 +P(\rho)\right)\ dx.
\end{align}
In the case of subsonic velocities and smooth solutions, which is the relevant case for gas transport, this formulation allows to use relative energy techniques in order to investigate
the stability of the equations \eqref{eq:system_1}-\eqref{eq:system_2} with respect to parameters and initial data, see \cite{EggerGiesselmann}.
The essential idea of our approach is to
use a modification of this relative energy method to
prove the exponential synchronization of the observer to the original system, i.e.,
we show the exponential convergence in the long time limit of the state of the observer system to the original system state. 
In this paper we show the exponential convergence of an exact solution of the observer system towards the original system state, but it is also interesting to investigate the convergence of a discretized observer system towards the exact solution.
We expect that for suitable discretization schemes, e.g. for a mixed finite element method as in \cite{Egger2023_RelativeEnergysiDiscrete}, similar convergence results can be shown.

We study the case where distributed partial measurements are available, i.e.,  measurements of only one state variable on the whole computational domain. Our aim is to estimate the complete system state. 
A motivation for this scenario is that e.g. for the Saint-Venant system the water height in a network of open channels is accessible more easily than the velocity. For the barotropic Euler system, we assume 
that we have distributed measurements of one of the fields velocity, density or mass flow.
Then, we introduce an observer system for the original system \eqref{eq:system_1}-\eqref{eq:system_2}, which reads for the case of velocity measurements
\begin{align}
	\dt \hat\rho + \dx\hat m &= 0 
	\label{eq:observer_1v}\\
	\dt \hat v + \dx \hat h &= -\gamma|\hat v| \hat v +\mathcal{L}_v\label{eq:observer_2v} 
\end{align}
with source term $\mathcal{L}_v = \mu (v-\hat v)$ of Luenberger type  with `nudging parameter' $\mu>0$.
We complement both systems by identical Dirichlet boundary conditions at the left and right end of the
pipe. 
When we consider the equations on networks, we additionally need coupling conditions at inner nodes. Here, we will use the conservation of mass and the continuity of the specific enthalpy, which yields energy conservation as shown in \cite{Reigstad2015}.

The observer system \eqref{eq:observer_1v}-\eqref{eq:observer_2v} may seem similar to relaxation systems, see e.g. \cite{GoatinLaurent-Brouty2019,GugatHertyYu2018,JinXin1995}.
However, note that the convergence of relaxation systems is different, since
the relaxation system has more equations than the original system and convergence is considered in the limit where the relaxation parameter converges to zero.
In our case, the observer system has the same size and structure as the  original system and the nudging parameter $\mu$ is fixed, while convergence is shown in the limit $t\rightarrow\infty$.

The two main contributions of this paper are to prove exponential decay of the difference between the state of the observer system and the original system state for partial measurements and to
extend the existence result from \cite{GugatUlbrich2018} by showing existence of semi-global Lipschitz-continuous of the observer system as well as of the original system for the case of more general coupling and boundary conditions than in \cite{GugatUlbrich2018}. 
Using only measurements of one of the state variables the other state variable can be reconstructed exponentially fast.
Our synchronization result is valid for single pipes as well as for star-shaped networks, while the existence result is valid for general networks.
For the synchronization, 
it turns out that increasing the nudging parameter does not necessarily improve the speed of synchronization.
To the best of our knowledge, this is the first time that the synchronization of an observer system is proven for $2\times2$--systems of  quasilinear  hyperbolic balance laws.  

This paper is structured as follows. In section \ref{sec:GeneralSetup} we present the general setup, relevant notations and the Luenberger terms for the three cases of measurement of velocity, density or mass flow.
In section \ref{sec:existence} we show the existence of semi-global Lipschitz-continuous solutions of the observer system \eqref{eq:observer_1v}-\eqref{eq:observer_2v} for measurement of velocity or density provided that
the original system \eqref{eq:system_1}-\eqref{eq:system_2} admits a unique solution
and the solution of the original system as well as
the initial and boundary data are sufficiently small when expressed in Riemann invariants
and have sufficiently small Lipschitz-constants.
Let us note that our existence theorem also yields existence of a solution of the original system, if the initial and boundary data satisfy suitable smallness and regularity assumptions,
i.e., the existence theorem can first be used to show existence of a solution of the original system,
and then in a second step the theorem can be applied to show existence of a solution of the observer system, which depends on this original solution.
For the original system \eqref{eq:system_1}-\eqref{eq:system_2}, existence of semi-global solutions was shown in \cite{GugatUlbrich2018} for coupling conditions requiring conservation of mass and continuity of the pressure. This was done by writing the system in terms of Riemann invariants, which yields 
an integral equation along the characteristic curves.
The proof cannot be transferred directly to the observer system due to the form of the observer-terms, but 
by using a modified fixed-point iteration, we can show the existence of solutions of the observer system on a single pipe and on general networks with coupling conditions given by the conservation of mass and the continuity of the specific enthalpy.
In addition, our proof yields existence of solutions of the original system for these coupling conditions and for more general boundary conditions,
i.e., for coupling conditions requiring continuity of the specific enthalpy instead of continuity of the pressure as it was used in \cite{GugatUlbrich2018} and for boundary conditions for the mass flow or specific enthalpy instead of boundary conditions in terms of Riemann invariants.

In section \ref{sec:exp_synchronization} we show exponential convergence of solutions of the observer system towards the original system state
under the assumption that both systems admit solutions that satisfy suitable regularity assumptions, 
are bounded away from vacuum, have sufficiently small velocity and the solution of the original system has small time derivatives.
We can show this result for a single pipe and for star-shaped networks for all three cases, i.e., measurement of velocity, density or mass flow.
In the convergence proof we use techniques that are similar to those 
used in \cite{EggerGiesselmann} to show stability of the barotropic Euler equations, i.e., we will measure the distance between the solution of the original system \eqref{eq:system_1}-\eqref{eq:system_2} and the solution of the observer system \eqref{eq:observer_1v}-\eqref{eq:observer_2v} in terms of the relative energy and then estimate the time derivative of the relative energy.
This leads to decrease with respect to the variable that is measured, but not with respect to the other state variable. Therefore one of the main ideas of the proof is to use an extension of the relative energy framework. More precisely, 
we introduce an additional functional whose time derivative yields decrease with respect to the variable that is not measured. This modification of the relative energy is inspired by the extension of the energy that was used in \cite{EggerKugler2018} to study convergence of wave equations to steady states,
which in turn was based on \cite{Zuazua1988}, among others.

\section{General Setup and Luenberger Observer} \label{sec:GeneralSetup}
In the following we assume that the pressure law $p:\R_+\to\R$ is smooth and strictly monotone so that for given density bounds $0<\ubar\rho <\bar \rho$ there exist constants $\ubar C_{p'}, \bar C_{p'}>0$ such that
\begin{align*}
0<\ubar C_{p'}\le  p'(\rho)\le \bar C_{p'}   \qquad \forall \, 0<\ubar \rho\le \rho \le \bar \rho.
\end{align*}
Since the pressure potential $P:\R_+\to\R$ is connected to the pressure law by $p'(\rho) = \rho P''(\rho)$, this implies that the pressure potential is smooth and strongly convex with 
\begin{align*}
\  0<\ubar{C}_{P''}\le P''(\rho) \le \bar{C}_{P''},
\qquad |P'''(\rho)|\le C_{P'''}  \qquad \forall \, 0<\ubar \rho\le \rho \le \bar \rho
\end{align*}
for some constants $\ubar{C}_{P''}, \bar{C}_{P''}, C_{P'''}>0$.

We consider three different cases of available measurements, i.e., we assume that one of the three fields $v$, $\rho$ or  $m$ is measured in the whole interval $[0,\ell]$, while for the other state variable we have no information. Then, we consider the observer system
\begin{align}
	\dt \hat\rho + \dx\hat m &= \mathcal{L}_\rho \label{eq:observer_1}\\
	\dt \hat v + \dx \hat h &= -\gamma|\hat v| \hat v +\mathcal{L}_v\label{eq:observer_2} 
\end{align}
for $ 0<x<\ell,\ t>0$
with Luenberger observer terms $\mathcal{L}_\rho$, $\mathcal{L}_v$ depending on the measurements.
For velocity measurements we set
\begin{align} \label{eq:observer-term_v}
	\mathcal{L}_\rho=0, \qquad \mathcal{L}_v=\mu  (v-\hat v), \qquad \mu>0
\end{align}
and for density measurements 
\begin{align} \label{eq:observer-term_rho}
	\mathcal{L}_\rho=\mu\frac{c}{\sqrt{p'(\hat{\rho})}}\hat\rho (\tilde P(\rho)-\tilde P(\hat\rho)), \qquad \mathcal{L}_v=0, \qquad \mu>0
\end{align}
with  $c:= \sqrt{p'(\rho_{\text{ref}})}$ for a reference density $\rho_{\text{ref}}>0$ and $\tilde P(\rho):= \int_{\rho_{\text{ref}}}^{\rho}\tfrac{\sqrt{p'(s)}}{c s}ds$.
For measurements of the mass flow we will use
\begin{align}  \label{eq:observer-term_m}
	\mathcal{L}_\rho=0, \qquad \mathcal{L}_v=\mu (m-\hat m), \qquad \mu>0.
\end{align}
Note that the observer term for density measurements is not exactly in Luenberger form, but
for $0 < \underline{\rho}\le \rho , \hat{\rho}\le \bar \rho$  we have 
$\mathcal{L}_\rho=\mu\tfrac{\hat\rho}{\tilde \rho}\tfrac{\sqrt{p'(\tilde \rho)}}{\sqrt{p'(\hat{\rho})}} (\rho - \hat \rho)$
with some $\tilde \rho \in[ \underline{\rho},\bar \rho ]$
and due to the monotonicity of $p$ we have
\begin{align*} 
	0<\tfrac{\ubar\rho}{\bar \rho}\tfrac{\sqrt{\ubar C_{p'}}}{\sqrt{\bar C_{p'}}}
	\le\tfrac{\hat\rho}{\tilde \rho}\tfrac{\sqrt{p'(\tilde \rho)}}{\sqrt{p'(\hat{\rho})}}
	\le \tfrac{\bar\rho}{\ubar \rho}\tfrac{\sqrt{\bar C_{p'}}}{\sqrt{\ubar C_{p'}}}.
\end{align*}
The reason why we choose this form of the observer term is that this 
observer term is linear when expressed in Riemann invariants, which will be relevant for the existence proof.

\begin{remark}
	If exact measurements of $\rho$ are available one might compute $\partial_t \rho= -\partial_x m$ and, thus, reconstruct the whole system state without recourse to an observer. However, in practical applications, measurements are subject to measurement errors and while these are usually small for the measured field they will be significant for its derivatives. Thus, we provide a framework that uses measurement data of one of the fields $v$, $\rho$ or $m$, but does not make use of derivatives of measured data.
\end{remark}

Since we will also consider the observer system on networks, we have to introduce some notation.
We describe a network by a directed, connected graph with edges $e \in\mathcal{E}$ and vertex set $\V$ consisting of boundary nodes $\nu\in \V_\partial$ 
 and inner nodes $\nu\in \V\setminus \V_{\partial}$.
The edge $e \in\mathcal{E}$ is identified  with an interval $(0,\ell^e)$ and the set of edges that are incident to some node $\nu$ is denoted by $\E(\nu)$.
We assume that all pipes have the same diameter.
For an edge $e=(\nu_1, \nu_2)$ that starts in node $\nu_1$ and ends at node $\nu_2$ we denote the pipe direction by $s^e(\nu_1)=-1$ and $s^e(\nu_2)=1$.
In the following, we will abbreviate
\begin{align*}
	\|  u- v\|_{L^2(\E)}^2:= \sum_{e\in\E} \| u^e-v^e\|_{L^2(0, \ell^e)}^2.
\end{align*}
Then, we assume that on each edge $e\in\E$ the equations \eqref{eq:system_1}-\eqref{eq:system_2} and
\eqref{eq:observer_1}-\eqref{eq:observer_2}, respectively, are satisfied.
At the inner nodes we have to prescribe coupling conditions, where we will use the conservation of mass and the continuity of the specific enthalpy, i.e.,
\begin{align}\label{eq:couplingc}
	\sum_{e\in\E(\nu)} m^e(\nu) s^e(\nu) = 0, \qquad 
	h^e(\nu)=h^f(\nu) \quad \forall e,f\in\E(\nu) 
\end{align}
for every inner node $\nu\in\V\setminus\V_\partial$. These coupling conditions yield energy conservation, see \cite{Reigstad2015}.

\section{Existence of solutions} \label{sec:existence}
A prerequisite of our proof of convergence of solutions of the observer system towards the original solution in section \ref{sec:exp_synchronization} is the existence of Lipschitz-continuous solutions of the observer system that have small, subsonic velocities and densities that are bounded away from vacuum. Therefore, in this section we will show existence of solutions of the observer system with suitable bounds on velocity and density for observer terms given by \eqref{eq:observer-term_v} or \eqref{eq:observer-term_rho}, i.e., for measurement of velocity or density.

For the original system \eqref{eq:system_1}-\eqref{eq:system_2} existence of solutions was shown in \cite{GugatUlbrich2018} for a single pipe and for networks, where in contrast to our coupling conditions conservation of mass and continuity of the pressure is used as coupling condition. If the initial and boundary conditions are sufficiently small when written in terms of Riemann invariants and have sufficiently small Lipschitz constant, existence of semi-global solutions, i.e., existence of solutions on a given time interval, with a priori bounds on velocity and density was established (see Theorem 6.1 in \cite{GugatUlbrich2018}). The main idea in the proof is rewriting the barotropic Euler equations \eqref{eq:system_1}-\eqref{eq:system_2} in terms of Riemann invariants.
Then, by integrating along the characteristic curves, the problem is formulated as an integral equation, for which existence of solutions is shown by a fixed-point argument.
These techniques have also been used to study semi-global classical solutions, see \cite{LiRao2003}.

This result is not directly transferable to the observer system \eqref{eq:observer_1}-\eqref{eq:observer_2} (since the conditions (5.7) and (6.1) of \cite{GugatUlbrich2018} are not satisfied for the right hand side of the observer term for fixed $\mu>0$), but with a similar strategy as in \cite{GugatUlbrich2018} we can show existence of semi-global solutions of the observer system. The key idea here is to modify the fixed-point map by using the specific form of the observer terms.
Note that by setting the observer terms to zero, i.e., setting $\mu=0$, our proof also yields  existence of solutions of the original system for the coupling conditions \eqref{eq:couplingc}, i.e., for conservation of mass and continuity of the specific enthalpy.

\subsection{Observer system in Riemann invariants} 

Before we formulate the fixed-point map, we first rewrite the observer system for a single pipe in terms of Riemann invariants in order to diagonalize the advective part of the system.
We can write \eqref{eq:observer_1}-\eqref{eq:observer_2} as
\begin{align} \label{eq:observer_fluxmatrix}
	\begin{pmatrix} \dt \hat\rho \\ \dt \hat v\end{pmatrix}
	+\begin{pmatrix}\hat  v& \hat \rho \\ P''(\hat \rho) & \hat v\end{pmatrix}
	\begin{pmatrix} \dx \hat \rho \\ \dx \hat v\end{pmatrix}
	=\begin{pmatrix} \mathcal{L}_\rho \\ -\gamma|\hat v| \hat v +\mathcal{L}_v	\end{pmatrix},
\end{align}
where the flux Jacobian matrix has the eigenvalues
\begin{align*}
	\hat \lambda_\pm = \hat v \pm \sqrt{p'(\hat{\rho})},
\end{align*}
where $\sqrt{p'(\hat{\rho})}$ is the speed of sound,
and left eigenvectors
\begin{align*}
	l_\pm = \left( \frac{\sqrt{p'(\hat{\rho})}}{c \hat\rho} , \pm\frac{1}{c} \right)
\end{align*}
with $c= \sqrt{p'(\rho_{\text{ref}})}$ for a reference density $\rho_{\text{ref}}>0$.
Now, we will write \eqref{eq:observer_fluxmatrix} in terms of Riemann invariants $S_\pm$, which have the property 
$(\partial_\rho S_\pm, \partial_v S_\pm)=l_\pm$ (for an introduction to Riemann invariants see \cite{dafermos2016hyperbolic}, Chapter 7).
Here, the observer system has the Riemann invariants
\begin{align}\label{eq:def_RI}
S_\pm = \int_{\rho_{\text{ref}}}^{\hat{\rho}}\tfrac{\sqrt{p'(s)}}{c s}ds\pm \tfrac{\hat v}{c}=:\tilde P(\hat{\rho})\pm \tfrac{\hat v}{c}.
\end{align}
By multiplying \eqref{eq:observer_fluxmatrix} from the left by the eigenvectors $l_\pm$, we see that the observer system can be written in terms of Riemann invariants as the diagonalized system
\begin{equation} \label{eq:obs_RI}
\begin{split}
&\begin{pmatrix}
\dt S_+ \\ \dt S_-
\end{pmatrix}
+\begin{pmatrix}
\hat\lambda_+ & 0\\ 0 &\hat\lambda_-
\end{pmatrix}
\begin{pmatrix}
\dx S_+ \\ \dx S_-
\end{pmatrix}  \\
&= (\sigma (S_+, S_-) -\tfrac{1}{c} \mathcal{L}_v (S_+, S_-)) \begin{pmatrix} -1\\ 1 \end{pmatrix}
+\tfrac{\sqrt{p'(\hat{\rho})}}{c \hat{\rho}}  \mathcal{L}_\rho(S_+, S_-) \begin{pmatrix} 1\\ 1 \end{pmatrix},
\end{split}
\end{equation}
where $\sigma(S_+, S_-) = \tfrac{\gamma}{c} |\hat v| \hat v=\gamma \tfrac{c}{4} |S_+-S_-| (S_+-S_-)$ is the friction term and $\mathcal{L}_v (S_+, S_-)$, $\mathcal{L}_\rho(S_+, S_-)$ are the Luenberger terms written in Riemann invariants.

Let us mention some properties of the Riemann invariants and of the eigenvalues of the system that are needed to get existence of characteristic curves.
First, the equation for the Riemann invariants implies
\begin{align} \label{eq:RI2rhov}
	\hat v= \tfrac{c}{2} (S_+-S_-), \quad
	\hat \rho = \tilde P^{-1}(\tfrac{1}{2}(S_++S_-)),
\end{align}
where $\tilde P^{-1}$ is the inverse of the function $\tilde P$ introduced in \eqref{eq:def_RI}, which exists since $\tilde P'>0$ due to the assumption $p'>0$.
Thus, if the Riemann invariants satisfy $|S_\pm|\le S_{\max}$ for a constant $S_{\max}>0$, then the corresponding velocity and density satisfy the bounds
\begin{align*}
	|\hat v|\le c S_{\max}=:\bar v,\quad
	\ubar\rho:=\tilde P^{-1}(-S_{\max})\le \hat\rho\le \tilde P^{-1}(S_{\max})=:\bar{\rho}.
\end{align*}
Since $\tilde P^{-1}(0)=\rhor>0$ and $\tilde P^{-1}$ is continuous, there exists a constant $S_{\max}>0$ such that
$0<\ubar{\rho}\le \hat\rho\le \bar{\rho}$.
If additionally $S_{\max}$ is sufficiently small such that $c S_{\max}\le \tfrac{1}{2} \sqrt{\ubar C_{p'}}$, 
then from the bounds on velocity and density we can deduce the following bounds for the eigenvalues
\begin{align} \label{eq:bound_lambda}
	0<\ubar{\Lambda}(S_{\max})\le \hat\lambda_+\le \Lambda(S_{\max}),\quad
	0>-\ubar{\Lambda}(S_{\max})\ge \hat\lambda_-\ge -\Lambda(S_{\max})
\end{align}
with $\ubar{\Lambda}(S_{\max})=\tfrac{1}{2} \sqrt{\ubar C_{p'}}$ and $\Lambda(S_{\max})=\tfrac{3}{2} \sqrt{\bar{C}_{p'}}$.
Furthermore, we can show
\begin{align*} 
	\big|\hat{\lambda}_\pm(S_+, S_-)- \hat{\lambda}_\pm(\tilde S_+, \tilde S_-)\big|
	\le L_{\lambda} (|S_+-\tilde S_+| + |S_--\tilde S_-| ),
\end{align*}
i.e., the eigenvalues are Lipschitz-continuous with Lipschitz constant
$L_{\lambda}\le \tfrac{c}{2} +\tfrac{c \bar\rho}{4 \ubar C _{p'}} \bar C_{p''}$,
where $\bar C_{p''}:=\max_{\ubar{\rho}\le\rho\le\bar{\rho}}| p''(\rho)|$.

Using these prerequisites we can show analogously to Lemma 5.1 in \cite{GugatUlbrich2018}
that, if $\mathcal{S}=(S_+, S_-)\in C([0,T]\times[0,\ell])^2$ is Lipschitz-continuous with respect to $x$ and satisfies $|S_{\pm}|\le S_{\max}$ and if $T\in(0,\min_{e\in\E}\tfrac{\ell^e}{\Lambda(S_{\max})})$,
the characteristic curves $\xi_\pm^\mathcal{S}(s,x,t)$ defined by
\begin{align*}
\xi_\pm^\mathcal{S}(t,x,t) = x, \quad 
\partial_s \xi_\pm^\mathcal{S}(s,x,t) =\lambda_\pm (\mathcal{S}(s, \xi_\pm^\mathcal{S}(s,x,t)))
\end{align*}
exist locally. They can be extended up to the boundary of $[0,T]\times[0,\ell]$
and are Lipschitz-continuous with respect to $x$.
Now, using the characteristic curves system \eqref{eq:obs_RI} can be written as
\begin{equation*}
\begin{split}
\partial_s S_\pm(s, \xi_\pm^\mathcal{S}(s,x,t))
&=  \mp\sigma (S_+, S_-)(s, \xi_\pm^\mathcal{S}(s,x,t)) \pm\tfrac{1}{c} \mathcal{L}_v (S_+, S_-)(s, \xi_\pm^\mathcal{S}(s,x,t))\\ 
&\quad+ \tfrac{\sqrt{p'(\hat{\rho})}}{c}\tfrac{1}{\hat{\rho}} \mathcal{L}_\rho(S_+, S_-)(s, \xi_\pm^\mathcal{S}(s,x,t)). 
\end{split}
\end{equation*}

In the following we will study the existence of solutions of the observer system for the case of velocity measurement or density measurement, i.e., with observer terms given by \eqref{eq:observer-term_v} and \eqref{eq:observer-term_rho}, respectively.
Therefore, we rewrite the observer terms $\mathcal{L}_\rho$, $\mathcal{L}_v$ in terms of the Riemann invariants ${R_\pm = \tilde P(\rho) \pm \tfrac{v}{c}}$ of the original system and  ${S_\pm = \tilde P(\hat \rho) \pm \tfrac{\hat v}{c}}$  of the observer system.
Using \eqref{eq:RI2rhov} we see that the system for velocity measurement, i.e., for $\mathcal{L}_\rho=0$, $\mathcal{L}_v=\mu  (v-\hat v)$, can be written as
\begin{equation}\label{eq:obs_RI_v-measurement}
\begin{split}
&\partial_s S_\pm(s, \xi_\pm^\mathcal{S}(s,x,t))\\
&=  \mp\sigma (S_+, S_-)(s, \xi_\pm^\mathcal{S}(s,x,t)) \pm \tfrac{\mu}{2} \left( R_+-R_- -(S_+-S_-)\right)(s, \xi_\pm^\mathcal{S}(s,x,t)).
\end{split}
\end{equation}
For measurement of the density we have $\mathcal{L}_v=0$, 
$\mathcal{L}_\rho=\mu\tfrac{c}{\sqrt{p'(\hat{\rho})}}\hat\rho (\tilde P(\rho)-\tilde P(\hat\rho))$
and therefore
\begin{equation}\label{eq:obs_RI_rho-measurement}
\begin{split}
&\partial_s S_\pm(s, \xi_\pm^\mathcal{S}(s,x,t))\\
&=  \mp\sigma (S_+, S_-)(s, \xi_\pm^\mathcal{S}(s,x,t)) 
+\tfrac{\mu}{2} \left(R_++R_--S_+-S_-\right)(s, \xi_\pm^\mathcal{S}(s,x,t)) .
\end{split}
\end{equation}

\begin{remark}
	In the following, we will show existence of semi-global Lipschitz-continuous
	solutions of the observer system for velocity measurement with observer terms \eqref{eq:observer-term_v}. 
	Since the system \eqref{eq:obs_RI_rho-measurement} for measurement of $\rho$ has the same structure as the system \eqref{eq:obs_RI_v-measurement} for measurement of $v$, i.e., both systems can be written as
	\begin{align*}
		\partial_s S_+ +\tfrac{\mu}{2} S_+ = -\sigma +b_+(R_+, R_-, S_-),\qquad
		\partial_s S_- +\tfrac{\mu}{2} S_- = +\sigma +b_-(R_+, R_-, S_+)
	\end{align*}
	for suitable functions $b_+, b_-$,
	existence of solutions of \eqref{eq:obs_RI_rho-measurement} can be shown analogously.
	The crucial point in the proof is that for measurements of velocity or density, the observer terms are linear when written in terms of Riemann invariants. However, this is not the case for measurement of the mass flow such that the same strategy cannot be used to show existence of solutions of the observer system for mass flow measurements.
\end{remark}

Before we present the existence proof, we first consider the coupling conditions, since we also show existence of solutions on networks.
One important step is to express the coupling conditions as well as the boundary conditions in Riemann invariants.

\subsection{Coupling and boundary conditions in Riemann invariants}
In order to show existence of solutions of the observer system for networks and for the case of boundary conditions for the mass flow or the enthalpy, we have to rewrite the coupling and boundary conditions in terms of Riemann invariants.
First, we consider the coupling conditions \eqref{eq:couplingc}.
Using that the mass flow can be written as
\begin{align*}
	m=\rho v = \tilde P^{-1}(\tfrac{1}{2}(R_++R_-)) \tfrac{c}{2} (R_+-R_-)
\end{align*}
and the specific enthalpy as
\begin{align*}
	h = \tfrac{1}{2} v^2 +P'(\rho)
	= \tfrac{1}{2} \left( \tfrac{c}{2} (R_+-R_-) \right)^2 + P'(\tilde P^{-1}(\tfrac{1}{2}(R_++R_-))),
\end{align*}
the coupling conditions in Riemann invariants are given by
\begin{align}
	\sum_{e\in\E(\nu)} \tilde P^{-1}(\tfrac{1}{2}(R_+^e+R_-^e))  (R_+^e-R_-^e)
	= 0, \label{eq:cc_RI_m}
\end{align}
\vspace{-0.4cm}
\begin{multline}
	\tfrac{c^2}{8}(R_+^e-R_-^e)^2+P'(\tilde P^{-1}(\tfrac{1}{2}(R_+^e+R_-^e)))= \\
	\qquad \tfrac{c^2}{8}(R_+^f-R_-^f)^2+P'(\tilde P^{-1}(\tfrac{1}{2}(R_+^f+R_-^f))),\quad \forall e,f\in\E(\nu) \label{eq:cc_RI_h}
\end{multline}
for any inner node $\nu\in\V\setminus \V_\partial$.
For ease of notation we have assumed that the direction of the incident pipes is such that all pipes start in the node $\nu$.
In order to show existence of solutions also on networks, 
we need the following lemma:
\begin{lem}  \label{lem:coupling_conditions}
	Consider an inner node $\nu$ of a pipe network, where $\nu$ has $n$ incident edges enumerated
	by $i=1,\ldots,n$.
	Then there exists some constant $S_{\max}>0$ 
	such that, if the incoming Riemann invariants $R_-=(R_-^1,\ldots, R_-^n)$ satisfy $|R_-^i|\le S_{\max}$, $i=1,\ldots,n$, 
	then there exist unique outgoing Riemann invariants $R_+=(R_+^1,\ldots,R_+^n)$
	satisfying the coupling conditions \eqref{eq:cc_RI_m}-\eqref{eq:cc_RI_h}.
	In this case there exists a constant $C(n)>0$ independent of $R_-$ such that the outgoing Riemann invariants are bounded by
	\begin{align*}
	|R_+|_\infty \le C(n) |R_-|_\infty.
	\end{align*}
\end{lem}

\begin{proof}
	We show the assertion by using the implicit function theorem.
	As a first step, we write \eqref{eq:cc_RI_m}-\eqref{eq:cc_RI_h} as
	\begin{align*}
	F(R_-,R_+)=\begin{pmatrix}
	0\\ \vdots \\0
	\end{pmatrix}
	\end{align*}
	with
	\begin{align*} 
		F(R_-, R_+):= \begin{pmatrix}
		\sum_{i=1}^n \tilde P^{-1}(\tfrac{1}{2}(R_+^i+R_-^i))  (R_+^i-R_-^i)\\
		\tfrac{c^2}{8}(R_+^2-R_-^2)^2+P'(\rho_2)-\tfrac{c^2}{8}(R_+^{1}-R_-^{1})^2-P'(\rho_1)\\
		\vdots\\
		\tfrac{c^2}{8}(R_+^n-R_-^n)^2+P'(\rho_n)-
		\tfrac{c^2}{8}(R_+^{n-1}-R_-^{n-1})^2-P'(\rho_{n-1})
		\end{pmatrix}
	\end{align*}
	with $\rho_i =\tilde P^{-1}(\tfrac{1}{2}(R_+^i+R_-^i))$.
	The entries of the derivative of $F$ with respect to $R_+$ are given by
	\begin{align*}
		&\left(\frac{\partial F}{\partial R_+}\right)_{1,i}=   
		(\tilde P^{-1})'(\tfrac{1}{2}(R_+^i+R_-^i)) \tfrac{1}{2}  (R_+^i-R_-^i) + \rho_i 
		,\quad&& i\in\{1, \ldots, n\},\\
		&\left(\frac{\partial F}{\partial R_+}\right)_{i,i}= \tfrac{c^2}{4} (R_+^i-R_-^i) + P''(\rho_i) (\tilde P^{-1})'(\tfrac{1}{2}(R_+^i+R_-^i)) \tfrac{1}{2}
		,\quad&& i\in\{2, \ldots, n\},\\
		&\left(\frac{\partial F}{\partial R_+}\right)_{i,i-1}= -\tfrac{c^2}{4} (R_+^{i-1}-R_-^{i-1})
		 - P''(\rho_{i-1}) (\tilde P^{-1})'(\tfrac{1}{2}(R_+^{i-1}+R_-^{i-1})) \tfrac{1}{2}
		,\quad&& i\in\{2, \ldots, n\}
	\end{align*}
	and all other entries of $\frac{\partial F}{\partial R_+}$ are zero.
	For $R_-=0$, $ R_+=0$, we have $F(0, 0)=0$ and
	\begin{align*}
		\frac{\partial F}{\partial R_+}(0,0)
		= \begin{pmatrix}
		\rhor&\rhor&\rhor& \hdots& \rhor\\
		-a & a&0&  \ldots& 0\\
		\vdots & \ddots& \ddots & &\vdots \\
		0 &\ldots & -a& a& 0\\
		0 &\ldots &  0 & -a& a
		\end{pmatrix},
	\end{align*}
	where $a:=\tfrac{1}{2}P''(\rhor)\rhor$ and $a>0$ since $P$ is strictly convex.
	We can show by induction that $\det\big(\tfrac{\partial F}{\partial R_+}(0,0)\big)>0$, i.e., $\tfrac{\partial F}{\partial R_+}(0,0)$ is invertible.
	Therefore the implicit function theorem provides the existence of some open neighborhood $W\subset\mathbb{R}^n$ of $0$ and a unique, continuous function $f:W\to \mathbb{R}^n$, $R_-\mapsto R_+=f(R_-)$
	such that $F(R_-, f(R_-))=0$. 
	This means that for incoming Riemann invariants $R_-$ with $|R_-^i|\le S_{\max}$, $i=1,\ldots,n$, for some $S_{\max}>0$ with $B_{S_{\max}}(0)\subset W$,
	the outgoing Riemann invariants $R_+$ are given by $R_+=f(R_-)$.
	
	It remains to show that we have a bound of the form $|R_+|_{\infty}\le C |R_-|_{\infty}$ for some constant $C>0$.
	In order to show this note that the implicit function theorem implies
	\begin{align*}
		\frac{\partial f}{\partial R_-}(R_-) = - \left(\frac{\partial F}{\partial R_+}(R_-,f(R_-))\right)^{-1} \frac{\partial F}{\partial R_-}(R_-,f(R_-)).
	\end{align*}
	Since all components of $F$ are continuously differentiable, $\big(\tfrac{\partial F}{\partial R_+}(R_-,f(R_-))\big)^{-1} $ and \linebreak$\frac{\partial F}{\partial R_-}(R_-,f(R_-))$ depend continuously on $R_-$, i.e., for $(R_-,f(R_-))$ in a neighbourhood of $(0,0)$ also the derivatives $\big(\tfrac{\partial F}{\partial R_+}(R_-,f(R_-))\big)^{-1} $ and $\frac{\partial F}{\partial R_-}(R_-,f(R_-))$ are close to $\big(\tfrac{\partial F}{\partial R_+}(0,0)\big)^{-1} $ and $\frac{\partial F}{\partial R_-}(0,0)$, respectively.
	Therefore, there exists a constant $C(n)$ such that, if we choose $S_{\max}$ sufficiently small, then
	\begin{align*}
		\big|\frac{\partial f}{\partial R_-}(R_-)\big|_{\infty} \le  \big| \left(\frac{\partial F}{\partial R_+}(R_-,f(R_-))\right)^{-1}\big|_{\infty}  \big|\frac{\partial F}{\partial R_-}(R_-,f(R_-))\big|_{\infty} \le C(n) 
	\end{align*}
	for all incoming Riemann invariants $R_-$ that satisfy $|R_-|_{\infty}\le S_{\max}$.
\end{proof}

Now, we consider the boundary conditions for the original system and the observer system. For the synchronization proof in Section \ref{sec:exp_synchronization}, the original and observer system are complemented by boundary conditions of the form
\begin{align}\label{eq:BC_existence}
	m(t,\nu)=\hat m(t,\nu)=m_b(t) \quad \text{or} \quad
	h(t,\nu)=\hat h(t,\nu)=h_b(t), \qquad 0\le t\le T,\ \nu\in\V_{\partial}
\end{align}
for the mass flow $m$ or the total specific enthalpy $h$. In order to deal with such boundary conditions in the proof of existence of solutions, we have to transform the boundary conditions \eqref{eq:BC_existence} to Riemann invariants.

\begin{lem} \label{lem:boundary_cond}
	Let $\nu\in\V_{\partial}$ be a boundary node and denote by $R_-$ the Riemann invariant that is directed into the node $\nu$ and by $R_+$ the Riemann invariant that is directed out of the node $\nu$. Consider the boundary conditions \eqref{eq:BC_existence}.
	Then there exists a constant $S_{\max}>0$ such that, if
	\begin{align}
		\tilde P^{-1}(-\tfrac{1}{8}S_{\max})(-\tfrac{c}{8}S_{\max})
		&\le m_b \le 
		\tilde P^{-1}(\tfrac{1}{8}S_{\max}) \tfrac{c}{8}S_{\max},\label{eq:BC_bound_mb}\\
		\tfrac{c^2}{8} (-\tfrac{1}{4}S_{\max})^2 + P'\left(\tilde  P^{-1}(-\tfrac{1}{8}S_{\max})\right)
		&\le h_b\le
		\tfrac{c^2}{8} (\tfrac{1}{4}S_{\max})^2 + P'\left(\tilde  P^{-1}(\tfrac{1}{8}S_{\max})\right)  \label{eq:BC_bound_hb}
	\end{align}
	and $|R_-|\le S_{\max}$, then there exists a unique outgoing Riemann invariant $R_+$ such that $(R_+, R_-)$ satisfies the boundary conditions \eqref{eq:BC_existence} and $R_+$ is bounded by
	\begin{align*}
		|R_+| \le \tfrac{1}{4} S_{\max} + 3 |R_-|.
	\end{align*}	

	For fixed $R_-$ with $|R_-|\le S_{\max}$, but different boundary values $m_b,\,\tilde m_b$ (or $h_b,\,\tilde h_b$) that satisfy \eqref{eq:BC_bound_mb}-\eqref{eq:BC_bound_hb},
	$R_+$ depends on the boundary data Lipschitz continuously, i.e.,
	we can estimate
	\begin{align} \label{eq:bound_BC_wrt_mb}
		|R_+(m_b)-R_+(\tilde m_b)|\le C_b |m_b-\tilde m_b|,\qquad 
		|R_+(h_b)-R_+(\tilde h_b)|\le C_b |h_b-\tilde h_b|
	\end{align}
	with a constant $0<C_b\le \max\{\tfrac{4}{c\ubar\rho}, \tfrac{4}{c\sqrt{\ubar C_{p'}}}\}$.
\end{lem}

\begin{proof}
	We only give the proof for boundary conditions for the mass flow. In this case, we can write \eqref{eq:BC_existence} in terms of Riemann invariants as
	\begin{align*}
		0=G(R_+, R_-)
		= \tilde P^{-1}(\tfrac{1}{2} (R_++R_-)) \tfrac{c}{2}(R_+-R_-) - m_b
	\end{align*}
	with 
	\begin{align*}
		\frac{\partial G}{\partial R_+} (R_+, R_-) &= \tfrac{c}{2} \rho \big(1+\tfrac{c}{2}(R_+-R_-) \tfrac{1}{\sqrt{p'(\rho)}} \big),\\
		\frac{\partial G}{\partial R_-} (R_+, R_-) &= - \tfrac{c}{2} \rho \big(1-\tfrac{c}{2}(R_+-R_-) \tfrac{1}{\sqrt{p'(\rho)}} \big),
	\end{align*}
	where $\rho=P^{-1}(\tfrac{1}{2} (R_++R_-))$.
	If $|R_\pm|\le \bar S_{\max}$ for some $0<\bar S_{\max}\le\tfrac{1}{2c} \sqrt{\ubar C_{p'}}$, then $\tfrac{\partial G}{\partial R_+}>0$ and $\tfrac{\partial G}{\partial R_-}<0$, i.e., the contour line of $G(R_+, R_-)=0$ is monotonously increasing. This means that there exists a function $g:[-S_{\max}, S_{\max}]\rightarrow \R$ such that $g(R_-)=R_+$ if $G(R_+, R_-)=0$, where $S_{\max}>0$ has to be chosen sufficiently small such that $|R_\pm|\le \bar S_{\max}$.
	
	Note that
	\begin{align*}
		\left|\frac{\partial R_+}{\partial R_-}\right| = \left|   \frac{\tfrac{\partial G}{\partial R_-}}{\tfrac{\partial G}{\partial R_+}}\right|
		\le \frac{1+c \bar S_{\max}\tfrac{1}{\sqrt{\ubar C_{p'}}}}{1-c \bar S_{\max}\tfrac{1}{\sqrt{\ubar C_{p'}}}}
		\le 3
	\end{align*}
	for $|R_\pm|\le \bar S_{\max}\le \tfrac{1}{2c} \sqrt{\ubar C_{p'}}$.
	Now, let $R_-^\star=0$ and let $R_+^\star$ be defined by $G(R_+^\star, R_-^\star)=0$.
	If \eqref{eq:BC_bound_mb} is satisfied, then there exists a unique $R_+^\star$ with $|R_+^\star|\le \tfrac{1}{4} S_{\max} $.
	Thus, for $|R_-|\le S_{\max}$ and $S_{\max}<\tfrac{1}{4} \bar S_{\max}$ the Riemann invariant $R_+=g(R_-)$ satisfies
	\begin{align*}
		|R_+|\le |R_+^\star| + |R_+- R_+^\star|
		\le \tfrac{1}{4} S_{\max} + \sup\{|\tfrac{\partial R_+}{\partial R_-}|\} |R_-- R_-^\star|
		\le \tfrac{1}{4} S_{\max} + 3 |R_-|
		\le \bar S_{\max},
	\end{align*}
	i.e., the boundary condition is invertible for $|R_-|\le S_{\max}$ if the bound \eqref{eq:BC_bound_mb} on $m_b$ is satisfied.
	
	In order to derive the bound \eqref{eq:bound_BC_wrt_mb}, we estimate
	\begin{align*}
		\left|\frac{\partial R_+}{\partial m_b}\right| = \left|   \frac{\tfrac{\partial G}{\partial m_b}}{\tfrac{\partial G}{\partial R_+}}\right|
		\le \frac{1}{\tfrac{c}{2} \rho \big(1+\tfrac{c}{2}(R_+-R_-) \tfrac{1}{\sqrt{p'(\rho)}} \big) }
		\le \frac{4}{c \ubar\rho},
	\end{align*}
	where the last inequality holds for $|R_\pm|\le \bar S_{\max}$, which is satisfied for $S_{\max}$ sufficiently small.
\end{proof}

\subsection{Existence of solutions of the observer system for measurement of $v$}
By formulating the observer system \eqref{eq:obs_RI_v-measurement} for velocity measurements as a fixed-point iteration along the characteristic curves we can show the following existence theorem:

\begin{thm} \label{thm:localExistence}
	Let some $\mu\ge0$ be given.
	Consider a pipe network, where any inner node has at most $n$ incident edges. 
	Let some 
	$T\in(0, \min_{e\in\E}\tfrac{\ell^e}{\Lambda(S_{\max})})$ such that
	\linebreak  $1-\e^{-\tfrac{\mu}{2} T}\le \tfrac{1}{12}\min\{1,\tfrac{1}{4+C(n)}
	\tfrac{4}{9}\tfrac{\ubar\Lambda(S_{\max})}{\Lambda(S_{\max})}\}$ 
	with the constant $C(n)$ from Lemma \ref{lem:coupling_conditions} be given
	and consider the system \eqref{eq:obs_RI_v-measurement}, i.e., the observer system for measurement of $v$, together with the initial conditions 
		\begin{align}
			S_{\pm,e}(0,x)=y_\pm^e(x), \quad  e\in\E,\, x\in(0,\ell^e) \label{eq:obs_RI_IC}. 		
		\end{align}
	For the boundary conditions we use one of the following two cases.
	
	\textbf{Case (i):} We complement the system by the boundary conditions
	\begin{align}
		S_{\text{out},e}(t,\nu)=u_{\text{out}}^e(t) 
		, \quad  \nu\in\V_\partial,\, e\in\E(\nu),\, t\in(0,T), \label{eq:obs_RI_BC}
	\end{align}
	where $S_{\text{out},e}(t,\nu)$ denotes the Riemann invariant that is directed out of the node $\nu$, and assume that the boundary and initial data are $C^0$-compatible, i.e., ${y_\text{out}^e(\nu)=u_{\text{out}}^e(0)}$ for $\nu\in\V_\partial$, $e\in\E(\nu)$.
	In addition, we assume that the boundary data $u_\text{out}^e$ is bounded by $B_{\max}$, i.e.,
	\begin{align*}
		\max_{e\in\E} \sup_{t\in[0,T]}  |u_\text{out}^e(t)|
		\le B_{\max} ,
	\end{align*} and $u_\text{out}^e$, $e\in\E$, is Lipschitz-continuous with the Lipschitz-constant $L_I$.
	
	\textbf{Case (ii):}
	We complement the system by boundary conditions of the form \eqref{eq:BC_existence} for the mass flow $m$ or the specific enthalpy $h$ and assume that the boundary and initial data are $C^0$-compatible, i.e., $y_\pm^e(\nu)$ satisfies \eqref{eq:BC_existence} for $\nu\in\V_\partial$, $e\in\E(\nu)$.
	In addition, we assume that the bounds \eqref{eq:BC_bound_mb}-\eqref{eq:BC_bound_hb} are satisfied and the boundary data $ m_b$ and $h_b$ are Lipschitz-continuous with Lipschitz-constant $L_I$.
	
	Moreover, in both cases we assume that the initial data is compatible with the coupling conditions \eqref{eq:cc_RI_m}-\eqref{eq:cc_RI_h},
	the initial data $y_\pm^e$ is bounded by $B_{\max}$, i.e.,
		\begin{align*}
		\max_{e\in\E} \sup_{x\in[0,\ell^e]}  |y_\pm^e(x)|
		\le B_{\max} ,
	\end{align*}
	and $y_\pm^e$, $e\in\E$, is Lipschitz-continuous with the Lipschitz-constant $L_I$.

	For $\mu>0$, assume in addition that the original system \eqref{eq:system_1}-\eqref{eq:system_2} (written in Riemann invariants, i.e., as \eqref{eq:obs_RI} with $\mathcal{L}_v=0, \, \mathcal{L}_\rho=0$) together 
	with the boundary data
	\begin{align}
		R_{\text{out},e}(t,\nu)=u_{\text{out}}^e(t), \quad & \nu\in\V_\partial,\, e\in\E(\nu),\, t\in(0,T) \label{eq:sys_RI_BC}
	\end{align}
	in Case (i) or with the boundary conditions \eqref{eq:BC_existence} in Case (ii) and with
	the coupling conditions \eqref{eq:cc_RI_m}-\eqref{eq:cc_RI_h}
	has a unique solution $\mathcal R =((R_{+,e}, R_{-,e}))_{e\in\E}$, $(R_{+,e}, R_{-,e})\in C^0([0,T]\times[0,\ell^e])$ that satisfies
	$\|R_{\pm,e}\|_{L^\infty([0,T]\times[0,\ell])}\le S_{\max}$ and is Lipschitz-continuous with respect to $x$ with Lipschitz-constant bounded by $L_R$.
	
	Then, if $B_{\max}$, $L_I$, $S_{\max}$ and  $L_R$ are sufficiently small, the system \eqref{eq:obs_RI_v-measurement} 
	together with the initial conditions \eqref{eq:obs_RI_IC}, the boundary conditions of Case (i) or (ii) 
	and the coupling conditions \eqref{eq:cc_RI_m}-\eqref{eq:cc_RI_h}
	admits a unique solution in 
	\begin{align*}
		M(S_{\max}, L_R)
		:=&\{ \mathcal S =((S_{+,e}, S_{-,e})) _{e\in\E}, \, (S_{+,e}, S_{-,e})\in C^0([0,T]\times [0,\ell^e]): \\
		&|S_{\pm,e}|\le S_{\max}\ \forall e\in\E 
		\text{ and all $S_{\pm,e}$ are  Lipschitz-continous }\\
		&\text{ with respect to $x$	with Lipschitz-constant } L_R    \}.
	\end{align*}	
\end{thm}

\begin{remark} \label{rem:bounds_rho-v}
	Theorem \ref{thm:localExistence} provides the existence of Lipschitz-continuous solutions satisfying the bounds $|S_\pm|\le S_{\max}$. This implies the bounds
	\begin{align*}
		|\hat v|\le c S_{\max}=:\bar v,\quad
		\ubar\rho:=\tilde P^{-1}(-S_{\max})\le \hat\rho\le \tilde P^{-1}(S_{\max})=:\bar{\rho}
	\end{align*}
	for the velocity and density with $\ubar\rho>0$ for $S_{\max}$ sufficiently small.
	Thus, the bound $|S_\pm|\le S_{\max}$ on the Riemann invariants corresponds to bounds for $\hat\rho,\, \hat v$ that are of the same form as the bounds that will be used in the convergence proof, cf.  assumptions (A1) and (A3) in section \ref{sec:exp_synchronization}.
\end{remark}

\begin{remark}
		Note that Theorem \ref{thm:localExistence} has two main statements.
		First, in the case $\mu=0$, the theorem asserts the existence of 
		Lipschitz-continuous solutions of the original system for the coupling conditions \eqref{eq:couplingc} and for boundary conditions of the form \eqref{eq:BC_existence}. This was not known before, since \cite{GugatUlbrich2018} uses other coupling and boundary conditions.
		Second, given a solution of the original system, Theorem \ref{thm:localExistence} applied to the case $\mu>0$ states the existence of a solution of the observer system, which depends on the given original solution.
\end{remark}

\begin{proof}[Proof of Theorem \ref{thm:localExistence}]
	The main idea of the proof is to write \eqref{eq:obs_RI_v-measurement} as a suitable fixed-point mapping and then apply the Banach fixed-point theorem in order to show existence of a unique fixed-point in $M(S_{\max}, L_R)$, where we use on $M(S_{\max}, L_R)$ the norm
	\begin{align*}
		\|\mathcal{S}_e\|_M
		:=\max_{(t,x)\in[0,T]\times[0,\ell^e]} |S_{+,e}(t,x) | + |S_{-,e}(t,x) |.
	\end{align*}
	The strategy of the proof is similar to the strategy of the proof of Theorem 5.1 in \cite{GugatUlbrich2018}, but we use a modification of the fixed-point iteration that is used there in order to deal with the observer terms.
	Note that for $\mu>0$ we assume that the original system has a unique solution $\mathcal R =((R_{+,e}, R_{-,e})) _{e\in\E}$
	and we want to find a solution $\mathcal S$ of the observer system in dependence of this given solution $\mathcal R$,
	i.e., for $\mu>0$ the fixed-point mapping for $\mathcal S$ depends on the original solution $\mathcal R$.
	In order to define the fixed-point iteration, we denote $\mathcal{S}^i=
	((S^i_{+,e}, S^i_{-,e})) _{e\in\E}$ and define the mapping 
	\begin{align*}
		\Phi: \, M(S_{\max}, L_R)\to  M(S_{\max}, L_R),
		\, \mathcal{S}^i\mapsto \mathcal{S}^{i+1}, 
	\end{align*}
	where $\mathcal{S}^{i+1}=((S_{+,e}^{i+1}, S_{-,e}^{i-1})) _{e\in\E}$ is the solution of the differential equations 
	\begin{equation} \label{eq:FP1_S+}
	\begin{split}
	&\partial_s S_{+,e}^{i+1} (s, \xi_{+,e}^{\mathcal{S}^i}(s,x,t))
	+\tfrac{\mu}{2} S_{+,e}^{i+1}(s, \xi_{+,e}^{\mathcal{S}^i}(s,x,t))\\
	&=\tfrac{\mu}{2} \left( R_{+,e}-R_{-,e} +S_{-,e}^i\right)(s, \xi_{+,e}^{\mathcal{S}^i}(s,x,t))
	- \sigma  (\mathcal{S}^i) (s, \xi_{+,e}^{\mathcal{S}^i}(s,x,t)),\quad e\in\E,
	\end{split}
	\end{equation}
	along the $\xi_+$-characteristics and
	\begin{equation} \label{eq:FP1_S-}
	\begin{split}
	&\partial_s S_{-,e}^{i+1} (s, \xi_{-,e}^{\mathcal{S}^i}(s,x,t))
	+\tfrac{\mu}{2} S_{-,e}^{i+1}(s, \xi_{-,e}^{\mathcal{S}^i}(s,x,t))\\
	&=- \tfrac{\mu}{2} \left( R_{+,e}-R_{-,e} -S_{+,e}^i\right)(s, \xi_{-,e}^{\mathcal{S}^i}(s,x,t))
	+ \sigma  (\mathcal{S}^i) (s, \xi_{-,e}^{\mathcal{S}^i}(s,x,t)),\quad e\in\E, 
	\end{split}
	\end{equation}
	along the $\xi_-$-characteristics complemented with the initial and boundary conditions \eqref{eq:obs_RI_IC}-\eqref{eq:obs_RI_BC} and the coupling conditions \eqref{eq:cc_RI_m}-\eqref{eq:cc_RI_h}.
	Note that on the left hand side there is not only the derivative term $\partial_s S_+^{i+1} $, but also some part of the observer term. 
	This has the advantage that we get some extra damping from this term, which will be relevant in the estimates that are needed to show that $\Phi$ is a contraction and a self-mapping.
	
	For fixed $x$ and $t$, equations \eqref{eq:FP1_S+} and \eqref{eq:FP1_S-} are ordinary differential equations in the variable $s$ that can be solved by the "variation of constants"- formula. For the equation \eqref{eq:FP1_S+} for $S_+$ this yields
	\begin{equation} 
	\begin{split}\label{eq:FP2_S+}
	&S_{+,e}^{i+1} (s, \xi_+^{\mathcal{S}^i_e}(s,x,t))
	= 		S_{+,e}^{i+1}\xipe{t_+^{\mathcal{S}^i_e}(x,t)} \e^{-\tfrac{\mu}{2}(s-t_+^{\mathcal{ S}^i_e}(x,t))}	\\
	&\quad +\int_{t_+^{\mathcal{S}^i_e}(x,t)}^s \e^{-\tfrac{\mu}{2}(s-r)}
	\Big( \tfrac{\mu}{2} ( R_{+,e}-R_{-,e} +S_{-,e}^{i})(r, \xi_+^{\mathcal{S}^i_e}(r,x,t)) 
	- \sigma  (\mathcal{S}^i_e) (r, \xi_+^{\mathcal{S}^i_e}(r,x,t))  \Big)dr
	\end{split}
	\end{equation}
	on every pipe $e\in\E$
	for $s\ge t_+^{\mathcal{S}^i_e}(x,t)$, where $t_+^{\mathcal{S}^i_e}(x,t)$ is the time when the characteristic $\xi_+^{\mathcal{S}^i_e}(s,x,t)$ hits the boundary of the interval $[0,\ell^e]$ or the initial time, i.e.,  $t_+^{\mathcal{S}^i_e}(x,t)=0$, if
	$\xi_+^{\mathcal{S}^i_e}(s,x,t)>0$ for all $s\in [0,T]$, and $t_+^{\mathcal{S}^i_e}(x,t)$ is defined by $\xi_+^{\mathcal{S}^i_e}(t_+^{\mathcal{S}^i_e}(x,t),x,t)=0$ otherwise.
	In particular, $S_{+,e}^{i+1}\xip{t_+^{\mathcal{S}^i_e}(x,t)}$ is known from the initial data,  the boundary conditions
	or from the coupling conditions.
	This means that
	the fixed-point iteration can be written as
	\begin{equation*} 
	\begin{split}
	&S_{+,e}^{i+1} (s, \xi_+^{\mathcal{S}^i_e}(s,x,t))
	\\
	&=\int_{t_+^{\mathcal{S}^i_e}(x,t)}^s \e^{-\tfrac{\mu}{2}(s-r)}
	\left( \tfrac{\mu}{2} ( R_{+,e}-R_{-,e} +S_{-,e}^i)(r, \xi_+^{\mathcal{S}^i_e}(r,x,t)) 
	- \sigma  (\mathcal{S}^i_e) (r, \xi_+^{\mathcal{S}^i_e}(r,x,t))  \right)dr\\
	& \quad	+
	\begin{cases}
	y_+^e(\xipoe(0,x,t)) \e^{-\tfrac{\mu}{2}s}& \text{ if } t_+^{\mathcal{S}^i_e}(x,t)=0,\\
	S_{+,e}^{i+1}(t_+^{\mathcal{S}^i_e}(x,t),\nu) \e^{-\tfrac{\mu}{2}(s-t_+^{\mathcal{ S}^i_e}(x,t))}
	& \text{ if } \xipoe(t_+^{\mathcal{S}^i_e}(x,t),x,t)=\nu \text{ is an inner node}\\
	 & \text{ or a boundary node}
	\end{cases}	
	\end{split}
	\end{equation*}
	on each edge $e\in\E$.
	Now, if we choose $S_{\max}$ sufficiently small such that \eqref{eq:bound_lambda} holds, from the bound $T\le\min_{e\in\E}\tfrac{\ell^e}{\Lambda(S_{\max})}$ we deduce that 
	the value of $S_{+,e}^{i+1}(t_+^{\mathcal{S}^i_e}(x,t),\nu)$ in the last case can be followed back to the initial state.
	If $\nu$ is a boundary node and we use the boundary conditions of Case (i), 
	then $S_{+,e}^{i+1}(t_+^{\mathcal{S}^i_e}(x,t)=u_{+}^e(t_+^{\mathcal{S}^i_e}(x,t))$.
	If $\nu$ is an inner node or $\nu$ is a boundary node and we use the boundary conditions of Case (ii),
	then for $|S_-|\le S_{\max}$ with $S_{\max}>0$ sufficiently small by Lemma~\ref{lem:coupling_conditions} and Lemma~\ref{lem:boundary_cond} there exist functions $f_e$ and $g$, respectively, such that
	\begin{align*}
		S_{+,e}^{i+1}(t_+^{\mathcal{S}^i_e}(x,t),\nu) 
		=f_e(S_{-}^{i+1}(t_+^{\mathcal{S}^i_e}(x,t),\nu) )
	\end{align*}
	in the case that $\nu$ is an inner node and
	\begin{align*}
		S_{+,e}^{i+1}(t_+^{\mathcal{S}^i_e}(x,t),\nu) 
		=g(S_{-, e}^{i+1}(t_+^{\mathcal{S}^i_e}(x,t),\nu) )
	\end{align*}
	in the case that $\nu$ is a boundary node.
	The components $S_{-,f}^{i+1}$, $f\in\E(\nu)$, of $S_{-}^{i+1}$ are given by
	\begin{equation} 
	\begin{split} \label{eq:S-}
	&S_{-,f}^{i+1}(t_+^{\mathcal{S}^i_e}(x,t),\nu) \\
	&=\int_{0}^{t_+^{\mathcal{S}^i_e}(x,t)}
	\e^{-\tfrac{\mu}{2}(t_+^{\mathcal{S}^i_e}(x,t)-r)}
	\Big( -\tfrac{\mu}{2} ( R_{+,f}-R_{-,f} -S_{+,f}^i)(r, \xi_-^{\mathcal{S}^i_f}(r,\nu,t_+^{\mathcal{S}^i_e}(x,t))) \\
	&\qquad +\sigma  (\mathcal{S}^i_f) (r, \xi_-^{\mathcal{S}^i_f}(r,\nu,t_+^{\mathcal{S}^i_e}(x,t)))  \Big)dr	
	 + S_{-,f}^{i+1}(0,\xi_-^{\mathcal{S}^i_f}(0,\nu,t_+^{\mathcal{S}^i_e}(x,t)))
	\e^{-\tfrac{\mu}{2}t_+^{\mathcal{ S}^i_e}(x,t)} 
	\end{split}
	\end{equation}
	with $S_{-,f}^{i+1}(0,\xi_-^{\mathcal{S}^i_f}(0,\nu,t_+^{\mathcal{S}^i_e}(x,t)))$ given by the initial data.
	For ease of presentation we assume here that all pipes that are incident to the node $\nu$ are oriented such that they start in $\nu$.
		
	In the first part of the proof we show that the fixed-point mapping $\Phi$
	is well-defined, i.e., $\Phi(\mathcal{S}) \in M(S_{\max}, L_R)$ for $\mathcal{S} \in M(S_{\max}, L_R)$.
	Before we start, let us note that 
	\begin{align*}
	\sigma(S_+, S_-)=\tfrac{\gamma}{c}|\hat v| \hat v = \gamma \tfrac{c}{4} |S_+-S_-| (S_+-S_-)
	\le \gamma c(S_{\max})^2 =: \sigma_{\max}
	\end{align*}
	and $\sigma(S_+, S_-)$ is Lipschitz-continuous with Lipschitz-constant $L_\sigma$, i.e.,
	\begin{align} \label{eq:def_Lsigma}
	|\sigma(R_+, R_-)-\sigma(\tilde R_+, \tilde R_-)|
	\le L_\sigma \left( |R_+-\tilde R_+| +   |R_--\tilde R_-|  \right)
	\end{align}
	with $L_\sigma \le \gamma\, c\, S_{\max}$ for $\mathcal{S}\in M(S_{\max}, L_R)$.
	In the following, let us choose $S_{\max}$ sufficiently small such that $\hat v (S_+, S_-)=\tfrac{c}{2}(S_+-S_-) \le \tfrac{1}{2} \sqrt{\ubar C_{p'}}$, the bounds \eqref{eq:bound_lambda} on $\lambda_\pm$ are satisfied 
	and the bounds on $S_{\max}$ that are required in Lemma \ref{lem:coupling_conditions} and Lemma \ref{lem:boundary_cond} are satisfied.
	
	Let $\mathcal{S}^i \in M(S_{\max}, L_R)$. Then 
	\begin{align*}
		&|S_{+,e}^{i+1} (s, \xi_+^{\mathcal{S}^i_e}(s,x,t))|
		\le |		S_{+,e}^{i+1}\xipe{t_+^{\mathcal{S}^i_e}(x,t)} \e^{-\tfrac{\mu}{2}(s-t_+^{\mathcal{ S}^i}(x,t))}|	\\
		&\quad +\int_{t_+^{\mathcal{S}^i_e}(x,t)}^s \e^{-\tfrac{\mu}{2}(s-r)}
		\Big( \tfrac{\mu}{2} | R_{+,e}-R_{-,e} +S_{-,e}^i|(r, \xi_+^{\mathcal{S}^i_e}(r,x,t)) 
		+| \sigma  (\mathcal{S}^i_e)| (r, \xi_+^{\mathcal{S}^i_e}(r,x,t))  \Big)dr\\
		&\le |S_{+,e}^{i+1}\xipe{t_+^{\mathcal{S}^i_e}(x,t)}| + \int_{t_+^{\mathcal{S}^i_e}(x,t)}^s \e^{-\tfrac{\mu}{2}(s-r)}
		\left( \tfrac{\mu}{2} 3 S_{\max} 
		+\sigma_{\max}  \right)dr\\
		& =  |S_{+,e}^{i+1}\xipe{t_+^{\mathcal{S}^i_e}(x,t)}| + 
		\left(1-\e^{-\tfrac{\mu}{2} (s-t_+^{\mathcal{S}^i_e}(x,t))  }\right)
		3 S_{\max} + T \sigma_{\max}  \\
		&\le |S_{+,e}^{i+1}\xipe{t_+^{\mathcal{S}^i_e}(x,t)}|+ \tfrac{1}{12} \tfrac{1}{4+C(n)}
		  3 S_{\max} + T\sigma_{\max} 
	\end{align*}
	for all $(x,t)\in[0,\ell]\times[0,T]$, $t_+^{\mathcal{S}^i_e}(x,t)\le s\le T$,
	since $0\le s-t_+^{\mathcal{S}^i_e}(x,t) \le T$ and $1-\e^{-\tfrac{\mu}{2} T}\le \tfrac{1}{12} \tfrac{1}{4+C(n)}$.
	If $S_{+,e}^{i+1}\xipe{t_+^{\mathcal{S}^i_e}(x,t)}$ is given by the initial data or the boundary conditions of Case~(i), 
	then
	$|S_{+,e}^{i+1}\xipe{t_+^{\mathcal{S}^i_e}(x,t)}| \le B_{\max}$.
	If $\xi_+^{\mathcal{ S}^i_e}(t_+^{\mathcal{S}^i_e}(x,t),x,t)=\nu$ is an inner node or $\nu$ is a boundary node and we use the boundary conditions of Case (ii), 
	then $	S_{+,e}^{i+1}(t_+^{\mathcal{S}^i_e}(x,t),\nu) 
	=f_e(S_{-}^{i+1}(t_+^{\mathcal{S}^i_e}(x,t),\nu) )$ (or $	S_{+,e}^{i+1}(t_+^{\mathcal{S}^i_e}(x,t),\nu) 
	=g(S_{-,e}^{i+1}(t_+^{\mathcal{S}^i_e}(x,t),\nu) )$) with 
	\begin{align}
		|S_{-,f}^{i+1}(t_+^{\mathcal{S}^i_e}(x,t),\nu) |
		&\le |S_{-,f}^{i+1}(0,\xi_+^{\mathcal{S}^i_f}(0,\nu,t_+^{\mathcal{S}^i_e}(x,t)))  \e^{-\tfrac{\mu}{2}t_+^{\mathcal{ S}^i_e}(x,t)}|\notag\\
		&\quad	+\tfrac{3}{2} \mu S_{\max}
		\int_0^{t_+^{\mathcal{S}^i_e}(x,t)}
		\e^{-\tfrac{\mu}{2}(t_+^{\mathcal{S}^i_e}(x,t)-r)}dr +T \sigma_{\max} \label{eq:S-_self-mapping}\\
		&\le B_{\max} + 3 S_{\max}
		\tfrac{1}{12} \tfrac{1}{4+C(n)}+T \sigma_{\max} \qquad \forall f\in\E(\nu). \notag
	\end{align}
	For the case of coupling conditions, Lemma~\ref{lem:coupling_conditions} implies		
	\begin{align*}
		 |S_{+,e}^{i+1}\xipe{t_+^{\mathcal{S}^i_e}(x,t)}|
		 &\le C(n)\max_{f\in\E(\nu)} |S_{-,f}^{i+1}(t_+^{\mathcal{S}^i_e}(x,t),\nu) |\\
		 &\le  C(n) \left(B_{\max} + \tfrac{1}{4} S_{\max} \tfrac{1}{4+C(n)}+T \sigma_{\max} \right).
	\end{align*}
	For the case of boundary conditions of Case (ii), 
	Lemma ~\ref{lem:boundary_cond} implies
		\begin{align*}
		|S_{+,e}^{i+1}\xipe{t_+^{\mathcal{S}^i_e}(x,t)}|
		&\le \tfrac{1}{4} S_{\max} + 3 |S_{-,f}^{i+1}(t_+^{\mathcal{S}^i_e}(x,t),\nu) |\\
		&\le \tfrac{1}{4} S_{\max} + 3 \left(B_{\max} + \tfrac{1}{4} S_{\max} \tfrac{1}{4+C(n)}+T \sigma_{\max} \right).
	\end{align*}	
	Therefore
	\begin{align*}
		&|S_{+,e}^{i+1} (s, \xi_+^{\mathcal{S}^i_e}(s,x,t))|\\
		&\le \tfrac{1}{4} S_{\max} \tfrac{1}{4+C(n)}+T \sigma_{\max}
			+\max\big \{B_{\max}, C(n)\left( B_{\max} + \tfrac{1}{4} S_{\max}  \tfrac{1}{4+C(n)} + T \sigma_{\max} \right),\\
			&\qquad  \tfrac{1}{4} S_{\max} + 3 \left(B_{\max} + \tfrac{1}{4} S_{\max} \tfrac{1}{4+C(n)}+T \sigma_{\max} \right) \big\}\\
		&\le (4+C(n))  B_{\max}  + \tfrac{1}{2} S_{\max}+  (4+C(n)) T\sigma_{\max}.
	\end{align*}
	
	Now, if $B_{\max}\le \tfrac{1}{4} \tfrac{1}{4+C(n)} S_{\max}$ and
	$S_{\max}$ is sufficiently small such that
	$(4+C(n)) T\sigma_{\max}\le \tfrac{1}{4} S_{\max}  $,
	then
	\begin{align*}
		|S_{+,e}^{i+1} (s, \xi_+^{\mathcal{S}^i_e}(s,x,t))|\le S_{\max}.
	\end{align*}
	Analogously, we can show $|S_{-,e}^{i+1} | \le S_{\max}$. Thus,
	$|\Phi(\mathcal{S}^ i)|\le S_{\max}$ for $\mathcal{S}^ i \in M(S_{\max}, L_R)$, if $S_{\max}$ and $B_{\max}$ are sufficiently small.

	Next, we want to show that the Lipschitz-constant of $\mathcal{S}^i$ with respect to $x$ is uniformly bounded in $i$,
	i.e, if $\mathcal{S}^i$ has the Lipschitz-constant $L_R$ with respect to $x$, then we have to show that $\mathcal{S}^{i+1}$
	has the same Lipschitz-constant with respect to $x$.
	From \eqref{eq:FP2_S+} we know that
	\begin{equation*} 
		\begin{split}
		&S_{+,e}^{i+1} (t, x) =S_{+,e}^{i+1}\xip{t_+^{\mathcal{S}^i_e}(x,t)}  \e^{-\tfrac{\mu}{2}(t-t_+^{\mathcal{ S}^i_e}(x,t))}\\
		&\quad+\int_{t_+^{\mathcal{S}^i_e}(x,t)}^t \e^{-\tfrac{\mu}{2}(t-r)}
		\left( \tfrac{\mu}{2} ( R_{+,e}-R_{-,e} +S_{-,e}^i)(r, \xi_+^{\mathcal{S}^i_e}(r,x,t)) 
		- \sigma  (\mathcal{S}^i_e) (r, \xi_+^{\mathcal{S}^i_e}(r,x,t))  \right)dr\\
		&=:f_{1,e}(t,x)+f_{2,e}(t,x)
		\end{split}
	\end{equation*}
	for all $e\in\E$.
	Let $x_1, x_2\in [0,\ell^ e]$ with
	$t_+^{\mathcal{S}^i_e}(x_1,t)\le t_+^{\mathcal{S}^i_e}(x_2,t)$.
	Then 
	\begin{align*}
		|S_{+,e}^{i+1} (t, x_1) - S_{+,e}^{i+1} (t,x_2)|
		\le |f_{1,e}(t, x_1) - f_{1,e}(t,x_2) | + |f_{2,e}(t, x_1) - f_{2,e}(t,x_2) |
	\end{align*}
	with
	\begin{align*}
		&|f_{2,e}(t, x_1) - f_{2,e}(t,x_2) |\\
		&\le \int_{t_+^{\mathcal{S}^i_e}(x_2,t)}^t \e^{-\tfrac{\mu}{2}(t-r)}
		\Big| \tfrac{\mu}{2} ( R_{+,e}-R_{-,e} +S_{-,e}^i)(r, \xi_+^{\mathcal{S}^i_e}(r,x_1,t)) \\
		&\qquad-\tfrac{\mu}{2} ( R_{+,e}-R_{-,e} +S_{-,e}^i)(r, \xi_+^{\mathcal{S}^i_e}(r,x_2,t)) \\
		&\qquad- (\sigma  (\mathcal{S}^i_e) (r, \xi_+^{\mathcal{S}^i_e}(r,x_1,t)) 
		- \sigma  (\mathcal{S}^i_e) (r, \xi_+^{\mathcal{S}^i_e}(r,x_2,t))  )  \Big|dr\\
		&\quad + \int_{t_+^{\mathcal{S}^i_e}(x_1,t)}^{t_+^{\mathcal{S}^i_e}(x_2,t)} \e^{-\tfrac{\mu}{2}(t-r)}
		\Big| \tfrac{\mu}{2} ( R_{+,e}-R_{-,e} +S_{-,e}^i)(r, \xi_+^{\mathcal{S}^i_e}(r,x_1,t)) 
			- \sigma  (\mathcal{S}^i_e) (r, \xi_+^{\mathcal{S}^i_e}(r,x_1,t))  \Big|dr\\
		&\le \int_{t_+^{\mathcal{S}^i_e}(x_2,t)}^t \e^{-\tfrac{\mu}{2}(t-r)}
		\big(\tfrac{\mu}{2} 3 L_R 
		 + L_\sigma 2 L_R\big) |\xi_+^{\mathcal{S}^i_e}(r,x_1,t) -\xi_+^{\mathcal{S}^i_e}(r,x_2,t)|  dr\\
		&\quad + \int_{t_+^{\mathcal{S}^i_e}(x_1,t)}^{t_+^{\mathcal{S}^i_e}(x_2,t)} \e^{-\tfrac{\mu}{2}(t-r)}
		\big( \tfrac{\mu}{2} 3 S_{\max} +\sigma_{\max} \big ) dr,
	\end{align*}
	where $L_\sigma$ is the Lipschitz-constant of the friction term $\sigma$ with respect to the Riemann invariants, see \eqref{eq:def_Lsigma}. 
	Now, we want to estimate $|\xi_+^{\mathcal{S}^i_e}(r,x_1,t) -\xi_+^{\mathcal{S}^i_e}(r,x_2,t)|$.
	Analogously to Lemma 5.1 in \cite{GugatUlbrich2018} we can show that that the characteristics $\xi_\pm$ are Lipschitz-continuous with respect to $x$ with Lipschitz-constant $L_x= \exp(2 L_R T L_\lambda)$, i.e.,
	\begin{align*}
		|\xi_+^{\mathcal{S}^i_e}(r,x_1,t) -\xi_+^{\mathcal{S}^i_e}(r,x_2,t)|
		\le L_x |x_1-x_2|.
	\end{align*}
	In addition, as in \cite[Lemma 5.1]{GugatUlbrich2018} we can show  that $t_+^{\mathcal{S}}(x,t)$ is Lipschitz-continuous with respect to $x$ with the Lipschitz-constant 
	\begin{align*}
		L_t= \tfrac{1}{\ubar{\Lambda}(S_{\max})}\exp(2 L_R T L_\lambda)=\tfrac{1}{\ubar{\Lambda}(S_{\max})} L_x.
	\end{align*}
	This implies
	\begin{align*}
		|f_{2,e}(t, x_1) - f_{2,e}(t,x_2) |
		&\le \left( \tfrac{1}{12}   3 L_R + 2 T L_\sigma L_R    \right) L_x |x_1-x_2|
		+ L_t |x_1-x_2| \big( \tfrac{\mu}{2} 3 S_{\max} +\sigma_{\max} \big )\\
		&\le \left( \tfrac{1}{4}L_R  L_x + 2 T L_\sigma L_R  L_x  +L_t  \tfrac{\mu}{2} 3 S_{\max}  +L_t \sigma_{\max}\right)  |x_1-x_2|\\
		& =: L_2  |x_1-x_2|.
	\end{align*}
	Thus, if we choose $S_{\max}$ and $L_R$ sufficiently small such that $L_x\le \tfrac{3}{2}$, $3 T L_\sigma \le \tfrac{1}{24}  $, 
	$L_t  \tfrac{\mu}{2} 3 S_{\max}\le \tfrac{1}{24} L_R$ and 
	$L_t \sigma_{\max} \le \tfrac{1}{24} L_R$, 
	then
	\begin{align}\label{eq:Lipschitz-const_f1}
	L_2\le \tfrac{1}{2} L_R.
	\end{align}

	Now, we will estimate the Lipschitz-constant of  
	\begin{align*}
		f_{1,e}(t,x) = S_{+,e}^{i+1}\xipe{t_+^{\mathcal{S}^i_e}(x,t)}	\e^{-\tfrac{\mu}{2}(t-t_+^{\mathcal{ S}^i_e}(x,t))}.
	\end{align*}
	Here, we distinguish between three cases.
	
	\textbf{Case 1:} $t_+^{\mathcal{S}^i_e}(x_1,t)= 0 = t_+^{\mathcal{S}^i_e}(x_2,t)$.\newline
	Then $f_{1,e}(t,x_1)$ and $f_{1,e}(t,x_2)$ are given by the initial condition, i.e.,
	\begin{align*}
		|f_{1,e}(t,x_1)-f_{1,e}(t,x_2)|
		&= |S_{+,e}^{i+1}(0, \xi_+^{\mathcal{S}^i_e}(0,x_1,t))\e^{-\tfrac{\mu}{2}t}
		-S_{+,e}^{i+1}(0, \xi_+^{\mathcal{S}^i_e}(0,x_2,t))\e^{-\tfrac{\mu}{2}t} |\\
		&\le |\e^{-\tfrac{\mu}{2}t} |\, |y_+^e(\xi_+^{\mathcal{S}^i_e}(0,x_1,t)) - y_+^e(\xi_+^{\mathcal{S}^i_e}(0,x_2,t)) |\\
		&\le L_I |\xi_+^{\mathcal{S}^i_e}(0,x_1,t) - \xi_+^{\mathcal{S}^i_e}(0,x_2,t)|\\
		&\le L_I L_x |x_1-x_2|
	\end{align*}
	with $L_I L_x\le \tfrac{1}{2} L_R$ for $L_I$ sufficiently small.

	\textbf{Case 2:} $t_+^{\mathcal{S}^i_e}(x_1,t)>0$ and $  t_+^{\mathcal{S}^i_e}(x_2,t)>0$. \newline
	Then $\xi_+^{\mathcal{S}^i_e}(t_+^{\mathcal{S}^i_e}(x_1,t),x_1,t)=\xi_+^{\mathcal{S}^i_e}(t_+^{\mathcal{S}^i_e}(x_2,t),x_2,t)$ is either a boundary node or an inner node of the network, i.e.,
	$f_{1,e}(t,x_1)$ and $f_{1,e}(t,x_2)$ are given by the boundary conditions or the coupling conditions.
	
	\textbf{Case 2a:}
	If $f_{1,e}(t,x_1)$ and $f_{1,e}(t,x_2)$ are given by the boundary conditions of Case (i), 
	we have
	\begin{align*}
		&|f_{1,e}(t,x_1)-f_{1,e}(t,x_2)|\\
		&= |S_{+,e}^{i+1}(t_+^{\mathcal{S}^i_e}(x_1,t), 0)\e^{-\tfrac{\mu}{2}(t-t_+^{\mathcal{ S}^i_e}(x_1,t))} -S_{+,e}^{i+1}(t_+^{\mathcal{S}^i_e}(x_2,t), 0) \e^{-\tfrac{\mu}{2}(t-t_+^{\mathcal{ S}^i_e}(x_2,t))}|\\
		&\le |u_+^e(t_+^{\mathcal{S}^i_e}(x_1,t)) - u_+^e(t_+^{\mathcal{S}^i_e}(x_2,t)) | \,
		|\e^{-\tfrac{\mu}{2}(t-t_+^{\mathcal{ S}^i_e}(x_1,t))} |\\
		&\quad + |S_{+,e}^{i+1}(t_+^{\mathcal{S}^i_e}(x_2,t), 0)|\, |\e^{-\tfrac{\mu}{2}(t-t_+^{\mathcal{ S}^i_e}(x_1,t))}-\e^{-\tfrac{\mu}{2}(t-t_+^{\mathcal{ S}^i_e}(x_2,t))}|\\
		&\le L_I |t_+^{\mathcal{S}^i_e}(x_1,t)  - t_+^{\mathcal{S}^i_e}(x_2,t)|
		+B_{\max} \tfrac{\mu}{2} |t_+^{\mathcal{S}^i_e}(x_1,t)  - t_+^{\mathcal{S}^i_e}(x_2,t)| \\
		&\le(L_I  L_t +B_{\max} \tfrac{\mu}{2} L_t)|x_1-x_2|.
	\end{align*}
	
	\textbf{Case 2b:}
	If $\xi_+^{\mathcal{S}^i_e}(t_+^{\mathcal{S}^i_e}(x_1,t),x_1,t)=\xi_+^{\mathcal{S}^i_e}(t_+^{\mathcal{S}^i_e}(x_2,t),x_2,t)$ is an inner node $\nu\in\V\setminus\V_\partial$, then we can estimate
	\begin{align*}
		&|f_{1,e}(t,x_1)-f_{1,e}(t,x_2)|\\
		&=|S_{+,e}^{i+1}(t_+^{\mathcal{S}^i_e}(x_1,t), \nu)  \e^{-\tfrac{\mu}{2}(t-t_+^{\mathcal{ S}^i_e}(x_1,t))}
		-S_{+,e}^{i+1}(t_+^{\mathcal{S}^i_e}(x_2,t), \nu) \e^{-\tfrac{\mu}{2}(t-t_+^{\mathcal{ S}^i_e}(x_2,t))}|\\
		&\le |S_{+,e}^{i+1}(t_+^{\mathcal{S}^i_e}(x_1,t), \nu) |\, |\e^{-\tfrac{\mu}{2}(t-t_+^{\mathcal{ S}^i_e}(x_1,t))} -\e^{-\tfrac{\mu}{2}(t-t_+^{\mathcal{ S}^i_e}(x_2,t))}|\\
		&\quad + |S_{+,e}^{i+1}(t_+^{\mathcal{S}^i_e}(x_1,t), \nu) -S_{+,e}^{i+1}(t_+^{\mathcal{S}^i_e}(x_2,t), \nu)|\, |\e^{-\tfrac{\mu}{2}(t-t_+^{\mathcal{ S}^i_e}(x_2,t))}|\\
		&\le |S_{+,e}^{i+1}(t_+^{\mathcal{S}^i_e}(x_1,t), \nu) | \tfrac{\mu}{2}L_t |x_1-x_2|
		 + C(n) |S_-^{i+1}(\tx{x_1}, \nu) -S_-^{i+1}(\tx{x_2}, \nu)|_\infty
	\end{align*}
	with
	\begin{align*}
		|S_{+,e}^{i+1}(t_+^{\mathcal{S}^i_e}(x_1,t), \nu) |
		&\le C(n) |S_{-}^{i+1}(t_+^{\mathcal{S}^i_e}(x_1,t), \nu) |_\infty\\
		&\le C(n) \left( B_{\max}+3 S_{\max}\tfrac{1}{12}\tfrac{1}{4+C(n)}+T\sigma_{\max} \right)\\
		&\le C(n) B_{\max} +\tfrac{1}{4} S_{\max} + C(n) T \sigma_{\max} 
	\end{align*}
	and
	\begin{align*}
		&|S_{-,f}^{i+1}(\tx{x_1}, \nu) -S_{-,f}^{i+1}(\tx{x_2}, \nu)|\\
		&\le \big| S_{-,f}^{i+1}(0, \xi_-^{\mathcal{ S}_f^i}(0, \nu, \tx{x_1}))\e^{-\tfrac{\mu}{2}t_+^{\mathcal{ S}^i_e}(x_1,t)}  -
		S_{-,f}^{i+1}(0, \xi_-^{\mathcal{ S}_f^i}(0, \nu, \tx{x_2})) \e^{-\tfrac{\mu}{2}t_+^{\mathcal{ S}^i_e}(x_2,t)} \big|\\
		&\quad + \Big |\int_0^{\tx{x_1} }\e^{-\tfrac{\mu}{2}(\tx{x_1}-r)}  
		\Big( -\tfrac{\mu}{2} ( R_{+,f}-R_{-,f} -S_{+,f}^i)(r, \xi_-^{\mathcal{S}^i_f}(r,\nu,\tx{x_1})) \\
		&\hspace{5cm} +\sigma  (\mathcal{S}^i_e) (r, \xi_-^{\mathcal{S}^i_f}(r,\nu,\tx{x_1}))  \Big)dr\\
		&\qquad -\int_0^{\tx{x_2} }\e^{-\tfrac{\mu}{2}(\tx{x_2}-r)}  
		\Big( -\tfrac{\mu}{2} ( R_{+,f}-R_{-,f} -S_{+,f}^i)(r, \xi_-^{\mathcal{S}^i_f}(r,\nu,\tx{x_2})) \\
		&\hspace{5cm} +\sigma  (\mathcal{S}^i_e) (r, \xi_-^{\mathcal{S}^i_f}(r,\nu,\tx{x_2}))  \Big)dr  \Big |\\
		&\le B_{\max}\tfrac{\mu}{2} L_t |x_1-x_2| +
		L_I |\xi_-^{\mathcal{ S}_f^i}(0, \nu, \tx{x_1})- \xi_-^{\mathcal{ S}_f^i}(0, \nu, \tx{x_2})|\\
		& \quad + \big( \tfrac{1}{12}\tfrac{1}{4+C(n)}\tfrac{4}{9}\tfrac{\ubar\Lambda(S_{\max})}{\Lambda(S_{\max})}  3 L_R +2 T L_\sigma L_R  \big)\\
		&\qquad\max_{r\in[0, \tx{x_1}]}  | \xi_-^{\mathcal{S}^i_f}(r,\nu,\tx{x_1}) -\xi_-^{\mathcal{S}^i_f}(r,\nu,\tx{x_2})|\\
		&\quad +(\tfrac{3}{2} \mu S_{\max}+\sigma_{\max}) T \tfrac{\mu}{2} L_t |x_1-x_2|   
		 + L_t |x_1-x_2|  (\tfrac{3}{2} \mu S_{\max}+\sigma_{\max}).
	\end{align*}
	Using that for all $r\in[0,t_+^{\mathcal{ S}^i_e}(x_1,t)]$
	\begin{align*}
		& | \xi_-^{\mathcal{S}^i_f}(r,\nu,\tx{x_1}) -\xi_-^{\mathcal{S}^i_f}(r,\nu,\tx{x_2})|\\
		& = | \xi_-^{\mathcal{S}^i_f}(r,\nu,\tx{x_1}) -\xi_-^{\mathcal{S}^i_f}(r,  \xi_-^{\mathcal{S}^i_f}(\tx{x_1},\nu,\tx{x_2})
		,\tx{x_1})|\\
		&\le L_x |\nu -\xi_-^{\mathcal{S}^i_f}(\tx{x_1},\nu,\tx{x_2})|
		\le L_x L_s L_t |x_1-x_2|,
	\end{align*}
	where $L_s\le \Lambda(S_{\max})$ is the Lipschitz-constant of $\xi_\pm$ with respect to $s$,
	this implies 
	\begin{align*}
		&|f_{1,e}(t,x_1)-f_{1,e}(t,x_2)|\\
		&\le \Big( (C(n)B_{\max} +\tfrac{1}{4} S_{\max}+C(n) T\sigma_{\max}) \tfrac{\mu}{2}\\
		&\quad+C(n)\big(B_{\max}\tfrac{\mu}{2}
		+L_I L_x L_s
		+\big (\tfrac{1}{12}\tfrac{1}{4+C(n)}\tfrac{4}{9}\tfrac{\ubar\Lambda(S_{\max})}{\Lambda(S_{\max})}
		  3 L_R +2 T L_\sigma L_R  \big)L_x L_s \\
		&\hspace{2cm} + (\tfrac{3}{2} \mu S_{\max}+\sigma_{\max}) (1+\tfrac{\mu}{2}T) \big) \Big) L_t  |x_1-x_2|\\
		&\le \big( C(n)B_{\max} \mu +\tfrac{\mu}{8} S_{\max}+C(n) T \tfrac{\mu}{2}\sigma_{\max} 
		+C(n)L_x L_s L_I \\
		&\quad+ C(n)  (\tfrac{3}{2} \mu S_{\max}+\sigma_{\max}) (1+\tfrac{\mu}{2}T)  \big) L_t |x_1-x_2|\\
		&\quad	+ L_t L_x L_s \big( \tfrac{4}{9}\tfrac{\ubar\Lambda(S_{\max})}{\Lambda(S_{\max})} \tfrac{1}{4}+C(n) 2 T L_\sigma\big) L_R |x_1-x_2|.
	\end{align*}
	Since the bounds from above imply that $L_t L_x L_s
	\le \tfrac{9}{4}\tfrac{\Lambda(S_{\max})}{\ubar\Lambda(S_{\max})}$ for $L_R$ sufficiently small, we can choose $B_{\max}$, $S_{\max}$, $L_I$ and  $L_R$ sufficiently small such that 
	\begin{align*}
		|f_{1,e}(t,x_1)-f_{1,e}(t,x_2)|\le \tfrac{1}{2} L_R |x_1-x_2|.
	\end{align*}

	\textbf{Case 2c:} 	If $\xi_+^{\mathcal{S}^i_e}(t_+^{\mathcal{S}^i_e}(x_1,t),x_1,t)=\xi_+^{\mathcal{S}^i_e}(t_+^{\mathcal{S}^i_e}(x_2,t),x_2,t)$ is a boundary node $\nu\in \V_\partial$
	and we use the boundary conditions of Case (ii), 
	then by Lemma \ref{lem:boundary_cond} we can estimate
	\begin{align*}
		&|f_{1,e}(t,x_1)-f_{1,e}(t,x_2)|\\
		&\le |S_{+,e}^{i+1}(t_+^{\mathcal{S}^i_e}(x_1,t), \nu) |\, |\e^{-\tfrac{\mu}{2}(t-t_+^{\mathcal{ S}^i_e}(x_1,t))} -\e^{-\tfrac{\mu}{2}(t-t_+^{\mathcal{ S}^i_e}(x_2,t))}|\\
		&\quad + |S_{+,e}^{i+1}(t_+^{\mathcal{S}^i_e}(x_1,t), \nu) -S_{+,e}^{i+1}(t_+^{\mathcal{S}^i_e}(x_2,t), \nu)|\, |\e^{-\tfrac{\mu}{2}(t-t_+^{\mathcal{ S}^i_e}(x_2,t))}|\\
		&\le \left(\tfrac{1}{4} S_{\max} + 3 |S_{-,e}^{i+1}(t_+^{\mathcal{S}^i_e}(x_1,t), \nu)  | \right)
		 \tfrac{\mu}{2} L_t |x_1-x_2| 
		 + |S_{+,e}^{i+1}(t_+^{\mathcal{S}^i_e}(x_1,t), \nu) -S_{+,e}^{i+1}(t_+^{\mathcal{S}^i_e}(x_2,t), \nu)|
	\end{align*}
	with 
	\begin{align*}
		|S_{-,e}^{i+1}(t_+^{\mathcal{S}^i_e}(x_1,t), \nu)  |
		\le B_{\max}+\tfrac{1}{4} S_{\max} \tfrac{1}{4+C(n)}+T\sigma_{\max}.
	\end{align*}
	The last term can be estimated by
	\begin{align*}
		&|S_{+,e}^{i+1}(t_+^{\mathcal{S}^i_e}(x_1,t), \nu) -S_{+,e}^{i+1}(t_+^{\mathcal{S}^i_e}(x_2,t), \nu)|\\
		&\le 3 |S_{-,e}^{i+1}(t_+^{\mathcal{S}^i_e}(x_1,t), \nu) -S_{-,e}^{i+1}(t_+^{\mathcal{S}^i_e}(x_2,t), \nu)|
		+ C_b L_I |t_+^{\mathcal{S}^i_e}(x_1,t) - t_+^{\mathcal{S}^i_e}(x_2,t)|,
	\end{align*}
	where the second term appears since $S_{+,e}^{i+1}(t_+^{\mathcal{S}^i_e}(x_1,t)$ and $S_{+,e}^{i+1}(t_+^{\mathcal{S}^i_e}(x_2,t)$
	depend on boundary values for $m_b$ or $h_b$ at the different times $t_+^{\mathcal{S}^i_e}(x_1,t)$ and $t_+^{\mathcal{S}^i_e}(x_2,t)$.
	The constant $C_b$ is the constant from estimate \eqref{eq:bound_BC_wrt_mb} in Lemma \ref{lem:boundary_cond}.
	Similar to Case 2b we have
	\begin{align*}
		&|S_{-,e}^{i+1}(t_+^{\mathcal{S}^i_e}(x_1,t), \nu) -S_{-,e}^{i+1}(t_+^{\mathcal{S}^i_e}(x_2,t), \nu)|\\
		&\le B_{\max}\tfrac{\mu}{2} L_t |x_1-x_2| +
		L_I L_x L_s L_t |x_1-x_2|\\
		& \quad + \big( \tfrac{1}{12}\tfrac{1}{4+C(n)}\tfrac{4}{9}\tfrac{\ubar\Lambda(S_{\max})}{\Lambda(S_{\max})}  3 L_R +2 T L_\sigma L_R  \big) L_x L_s L_t |x_1-x_2|\\
		&\quad +(\tfrac{3}{2} \mu S_{\max}+\sigma_{\max}) T \tfrac{\mu}{2} L_t |x_1-x_2|   
		+ L_t |x_1-x_2|  (\tfrac{3}{2} \mu S_{\max}+\sigma_{\max})\\
		&\le \Big( B_{\max}\tfrac{\mu}{2}  +
		L_I L_x L_s 
		+  \tfrac{1}{4} \tfrac{L_R}{L_t}
		+2 T L_\sigma L_R   L_x L_s \\
		&\quad +(\tfrac{3}{2} \mu S_{\max}+\sigma_{\max}) ( T \tfrac{\mu}{2}    +1) \Big )
		L_t |x_1-x_2|.
	\end{align*}
	In the last step we have again used that $L_t L_x L_s
	\le \tfrac{9}{4}\tfrac{\Lambda(S_{\max})}{\ubar\Lambda(S_{\max})}$ for $L_R$ sufficiently small.
	Summarizing yields
	\begin{align*}
		&|f_{1,e}(t,x_1)-f_{1,e}(t,x_2)|\\
		&\le \left(\tfrac{1}{4} S_{\max} + 3 \big(B_{\max}+\tfrac{1}{4} S_{\max} \tfrac{1}{4+C(n)}+T\sigma_{\max}  \big)\right)
		\tfrac{\mu}{2} L_t |x_1-x_2| \\
		&\quad 	+ 3 \Big( B_{\max}\tfrac{\mu}{2}  +
		L_I L_x L_s 
		+  \tfrac{1}{4} \tfrac{L_R}{L_t}
		+2 T L_\sigma L_R   L_x L_s \\
		&\qquad +(\tfrac{3}{2} \mu S_{\max}+\sigma_{\max}) ( T \tfrac{\mu}{2}    +1) \Big )
		L_t |x_1-x_2|
		+ C_b L_I L_t |x_1-x_2|.
	\end{align*}
	Again, if we choose $B_{\max}$, $S_{\max}$, $L_I$ and  $L_R$ sufficiently small, then
	\begin{align*}
		|f_{1,e}(t,x_1)-f_{1,e}(t,x_2)|\le \tfrac{1}{2} L_R |x_1-x_2|.
	\end{align*}

	\textbf{Case 3:} $t_+^{\mathcal{S}^i_e}(x_1,t)= 0 < t_+^{\mathcal{S}^i_e}(x_2,t)$.\newline
	Then $\xi_+^{\mathcal{S}^i_e}(t_+^{\mathcal{S}^i_e}(x_2,t),x_2,t)=0$ and $f_{1,e}(t,x_1)$ is given by the initial condition, while $f_{1,e}(t,x_2)$ is given by the boundary condition or the coupling conditions. This case can be treated similarly to Case 2.
	Therefore we do not present the details here. Note that in the case that $f_{1,e}(t,x_2)$ is given by the coupling conditions the compatibility of the initial data with the coupling conditions has to be used.

	Summing up the Cases 1--3 shows that the Lipschitz-constant of $f_{1,e}$ is bounded by $\tfrac{1}{2} L_R$, if $B_{\max}$, $S_{\max}$, $L_I$ and  $L_R$ are sufficiently small.
	Together with \eqref{eq:Lipschitz-const_f1} this shows that the Lipschitz-constant of $S_{+,e}^{i+1}$ is bounded by $L_R$ for every $e\in\E$.
	Analogously it can be shown that the Lipschitz-constant of $S_-^{i+1}$ is bounded by $L_R$. 
	In other words, we have shown that the Lipschitz-constant of $\mathcal{ S}$ with respect to $x$ remains bounded under the fixed-point mapping.
	Therefore $\Phi(\mathcal{S}) \in M(S_{\max}, L_R)$ for $\mathcal{S} \in M(S_{\max}, L_R)$. In particular, the fixed-point mapping is well-defined.
		
	In order to apply the Banach fixed-point theorem, it remains to show the contraction property.
	We will only show the estimation for $S_+$, the estimation for $S_-$ can be done similarly.
	Let $\mathcal{S}^i, \tilde{\mathcal{S}}^i\in M(S_{\max}, L_R)$ and denote  $\mathcal{S}^{i+1} :=\Phi(\mathcal{S}^i)$, $\tilde{\mathcal{S}}^{i+1} :=\Phi(\tilde{\mathcal{S}}^i)$. Let $e\in\E$ and assume without loss of generality $t_+^{\mathcal{\tilde S}^i_e}(x,t)\le t_+^{\mathcal{ S}^i_e}(x,t)$.  Then
	\begin{align*} 
		&S_{+,e}^{i+1} (s, \xi_+^{\mathcal{S}^i_e}(s,x,t))
		-\tilde S_{+,e}^{i+1} (s, \xi_+^{\mathcal{\tilde S}^i_e}(s,x,t))\\
		&= \int_{t_+^{\mathcal{S}^i_e}(x,t)}^s \e^{-\tfrac{\mu}{2}(s-r)}
		\Big( \tfrac{\mu}{2} \big(( R_{+,e}-R_{-,e})(r, \xi_+^{\mathcal{S}^i_e}(r,x,t)) -( R_{+,e}-R_{-,e})(r, \xi_+^{\mathcal{\tilde S}^i_e}(r,x,t)) \big) \\
		&\qquad +  \tfrac{\mu}{2} \big(S_{-,e}^i(r, \xi_+^{\mathcal{S}^i_e}(r,x,t)) -\tilde S_{-,e}^i(r, \xi_+^{\mathcal{\tilde S}^i_e}(r,x,t)) \big)\\
		&\qquad- (\sigma  (\mathcal{S}^i_e) (r, \xi_+^{\mathcal{S}^i_e}(r,x,t))
		- \sigma  (\mathcal{\tilde S}^i_e) (r, \xi_+^{\mathcal{\tilde S}^i_e}(r,x,t)) ) \Big)dr\\
		&\quad-\int_{t_+^{\mathcal{\tilde S}^i_e}(x,t)}^{t_+^{\mathcal{S}^i_e}(x,t)} \e^{-\tfrac{\mu}{2}(s-r)} 
		\left( \tfrac{\mu}{2} ( R_{+,e}-R_{-,e} +\tilde S_{-,e}^i)(r, \xi_+^{\mathcal{\tilde S}^i_e}(r,x,t)) 
		- \sigma  (\mathcal{\tilde S}^i_e) (r, \xi_+^{\mathcal{\tilde S}^i_e}(r,x,t))  \right)dr\\
		&\quad + \big(S_{+,e}^{i+1}\xipe{t_+^{\mathcal{S}^i_e}(x,t)} \e^{-\tfrac{\mu}{2}(s-t_+^{\mathcal{ S}^i_e}(x,t))}\\
		&\qquad -\tilde S_{+,e}^{i+1}\xipte{t_+^{\mathcal{\tilde S}^i_e}(x,t)}\e^{-\tfrac{\mu}{2}(s-t_+^{\mathcal{ \tilde S}^i_e}(x,t))} \big)\\
		& =: \int_{t_+^{\mathcal{S}^i_e}(x,t)}^s \e^{-\tfrac{\mu}{2}(s-r)}
		\Big( (F_a) + (F_b) + (F_c)  \Big) dr
		+ (F_d) + (F_e).
	\end{align*}
	Using the Lipschitz-constant of $R_\pm$ with respect to $x$, we can estimate $(F_a)$ by
	\begin{align*}
		|(F_a)| &=  \tfrac{\mu}{2} \big|( R_{+,e}-R_{-,e})(r, \xi_+^{\mathcal{S}^i_e}(r,x,t)) -( R_{+,e}-R_{-,e})(r, \xi_+^{\mathcal{\tilde S}^i_e}(r,x,t)) \big|\\
		&\le \tfrac{\mu}{2} \, 2 \,  L_R |\xi_+^{\mathcal{S}^i_e}(r,x,t) - \xi_+^{\mathcal{\tilde S}^i_e}(r,x,t)|.
	\end{align*}
	Analogously to Lemma 5.1 in \cite{GugatUlbrich2018} we can show that 
	\begin{align}\label{eq:estimate_xi_Stilde}
		|\xi_+^{\mathcal{S}^i_e}(r,x,t) - \xi_+^{\mathcal{\tilde S}^i_e}(r,x,t)|
		\le T L_\lambda \exp(2 L_R T  L_\lambda) \|\mathcal{S}^i_e-\mathcal{\tilde S}^i_e\|_M
		= L_\lambda  T  L_x \|\mathcal{S}^i_e-\mathcal{\tilde S}^i_e\|_M.
	\end{align}
	This implies
	\begin{align*}
		|(F_a)| \le \mu L_R L_\lambda  T  L_x \|\mathcal{S}^i_e-\mathcal{\tilde S}^i_e\|_M.
	\end{align*}
	For $(F_b)$, we can estimate
	\begin{align*}
		|(F_b)| 
		&\le \tfrac{\mu}{2} \big( |S_{-,e}^i(r, \xi_+^{\mathcal{S}^i_e}(r,x,t)) -\tilde S_{-,e}^i(r, \xi_+^{\mathcal{S}^i_e}(r,x,t))|
		 +|\tilde S_{-,e}^i(r, \xi_+^{\mathcal{S}^i_e}(r,x,t)) -\tilde S_{-,e}^i(r, \xi_+^{\mathcal{\tilde S}^i_e}(r,x,t))| \big) \\
		&\le \tfrac{\mu}{2} \|\mathcal{S}^i_e-\mathcal{\tilde S}^i_e\|_M
		+ \tfrac{\mu}{2}  L_R |\xi_+^{\mathcal{S}^i_e}(r,x,t) - \xi_+^{\mathcal{\tilde S}^i_e}(r,x,t) |\\
		&\le \tfrac{\mu}{2}\left( 1+ L_R  L_\lambda  T  L_x \right) \|\mathcal{S}^i_e-\mathcal{\tilde S}^i_e \|_M.
	\end{align*}
	The friction term $(F_c)$ can be bounded by
	\begin{align*}
		|(F_c)|&= |\sigma  (\mathcal{S}^i_e) (r, \xi_+^{\mathcal{S}^i_e}(r,x,t))
		- \sigma  (\mathcal{\tilde S}^i_e) (r, \xi_+^{\mathcal{\tilde S}^i_e}(r,x,t)) |\\
		&\le L_\sigma 2 L_R |\xi_+^{\mathcal{S}^i_e}(r,x,t)-\xi_+^{\mathcal{\tilde S}^i_e}(r,x,t)| +L_\sigma  \|\mathcal{S}^i_e-\mathcal{\tilde S}^i_e\|_M\\
		&\le L_\sigma \left(1+2 L_R     L_\lambda  T  L_x \right) \|\mathcal{S}^i_e-\mathcal{\tilde S}^i_e\|_M. 
	\end{align*}
	The integral $(F_d)$ can be estimated by
	\begin{align*}
		|(F_d)| 
		\le \big(\tfrac{3 \mu}{2} S_{\max} +\sigma_{\max} \big)  |t_+^{\mathcal{\tilde S}^i_e}(x,t) - t_+^{\mathcal{S}^i_e}(x,t)|
		\le  \big(\tfrac{3 \mu}{2} S_{\max} +\sigma_{\max} \big) \tfrac{L_\lambda}{\ubar\Lambda(S_{\max})} T L_x  \|\mathcal{S}^i_e-\mathcal{\tilde S}^i_e\|_M,
	\end{align*}
	where we have used in the last step that $|t_+^{\mathcal{\tilde S}^i_e}(x,t) - t_+^{\mathcal{S}^i_e}(x,t)|$ can be estimated as in Lemma 5.1 of \cite{GugatUlbrich2018}.
	Therefore we can summarize
	\begin{align*}
		&\Big|\int_{t_+^{\mathcal{S}^i_e}(x,t)}^s \e^{-\tfrac{\mu}{2}(s-r)}
		\Big( (F_a) + (F_b) + (F_c)  \Big) dr + (F_d) \Big |\\
		&\le \Big(  \tfrac{1}{12} \tfrac{1}{4+C(n)}\big( 1 
		+  3  L_\lambda L_R T L_x\big)  +T L_\sigma (1+2L_\lambda L_R T L_x)\\
		&\quad  +\big( \tfrac{3 \mu}{2} S_{\max} + \sigma_{\max}\big)  \tfrac{L_\lambda}{\ubar\Lambda(S_{\max})} T L_x \Big)
		 \|\mathcal{S}^i_e-\mathcal{\tilde S}^i_e\|_M\\
		&\le \Big( \tfrac{1}{12} 
		 	+ \big( \tfrac{1}{4}  L_R  + (\tfrac{3 \mu}{2} S_{\max} + \sigma_{\max})\tfrac{1}{\ubar\Lambda(S_{\max})} 
		 	\big)  L_\lambda T L_x +T L_\sigma (1+2L_\lambda L_R T L_x)
		 	\Big) 
		 	 \|\mathcal{S}^i_e-\mathcal{\tilde S}^i_e\|_M.
	\end{align*}
	Now, we choose $S_{\max}>0$ and $L_R>0$ sufficiently small such that
	\begin{align*}
		\tfrac{1}{12} 
		+ \big( \tfrac{1}{4}  L_R  + (\tfrac{3 \mu}{2} S_{\max} + \sigma_{\max})\tfrac{1}{\ubar\Lambda(S_{\max})}  
		\big)  L_\lambda T L_x +T L_\sigma (1+2L_\lambda L_R T L_x)
		\le \tfrac{2}{3},
	\end{align*}
	i.e., the contraction constant of the terms $(F_a)-(F_d)$ is bounded by $\tfrac{2}{3}$.
	
	For the estimation of term $(F_e)$ we distinguish three cases. 
	
	\textbf{Case 1:} $t_+^{\mathcal{ S}^i_e}(x,t)=0=t_+^{\mathcal{\tilde S}^i_e}(x,t)$.\newline
	In this case we have
	\begin{align*}
		|(F_e)| &= |S_{+,e}^{i+1}\xipe{0} \e^{-\tfrac{\mu}{2}s}-\tilde S_{+,e}^{i+1}\xipte{0}\e^{-\tfrac{\mu}{2}s}|\\
		&\le |y_+^e(\xipoe(0,x,t))  - y_+^e(\xiptoe (0,x,t) ) | \, |\e^{-\tfrac{\mu}{2}s}|\\
		&\le L_I \,  |\xipoe(0,x,t) - \xiptoe (0,x,t)| 
		\le L_I \, L_\lambda  T  L_x \|\mathcal{S}^i_e-\mathcal{\tilde S}^i_e\|_M.
	\end{align*}

	\textbf{Case 2:} $t_+^{\mathcal{ S}^i_e}(x,t)>0$ and $t_+^{\mathcal{\tilde S}^i_e}(x,t)>0$.\newline
	Then $\xipoe(t_+^{\mathcal{S}^i_e}(x,t),x,t)=\xiptoe(t_+^{\mathcal{\tilde S}^i_e}(x,t),x,t)$ is a boundary node or an inner node.
	
	\textbf{Case 2a:}
	If $\xipoe(t_+^{\mathcal{S}^i_e}(x,t),x,t)=\xiptoe(t_+^{\mathcal{\tilde S}^i_e}(x,t),x,t)$ is a boundary node and we use the boundary conditions of Case (i), 
	we have
	\begin{align*}
		|(F_e)| &= |S_{+,e}^{i+1}( t_+^{\mathcal{S}^i_e}(x,t), 0) \e^{-\tfrac{\mu}{2}(s-t_+^{\mathcal{ S}^i_e}(x,t))} -\tilde S_{+,e}^{i+1}(t_+^{\mathcal{\tilde S}^i_e}(x,t), 0) \e^{-\tfrac{\mu}{2}(s-t_+^{\mathcal{ \tilde S}^i_e}(x,t))} |\\
		&= |u_+^e(t_+^{\mathcal{S}^i_e}(x,t)) - u_+^e(t_+^{\mathcal{\tilde S}^i_e}(x,t))   |\, |\e^{-\tfrac{\mu}{2}(s-t_+^{\mathcal{ S}^i_e}(x,t))}|\\
		&\quad  +|u_+^e(t_+^{\mathcal{\tilde S}^i_e}(x,t))| \,  |\e^{-\tfrac{\mu}{2}(s-t_+^{\mathcal{ S}^i_e}(x,t))} -\e^{-\tfrac{\mu}{2}(s-t_+^{\mathcal{\tilde  S}^i_e}(x,t))}|  \\
		&\le L_I |t_+^{\mathcal{S}^i_e}(x,t) - t_+^{\mathcal{\tilde S}^i_e}(x,t)|
		+B_{\max} \tfrac{\mu}{2} |t_+^{\mathcal{S}^i_e}(x,t) - t_+^{\mathcal{\tilde S}^i_e}(x,t)|\\
		&\le ( L_I +B_{\max} \tfrac{\mu}{2}  ) \, \tfrac{L_\lambda}{\ubar\Lambda(S_{\max})} T  L_x \|\mathcal{S}^i_e-\mathcal{\tilde S}^i_e\|_M.
	\end{align*}
	
	\textbf{Case 2b:}
	If $\xipoe(t_+^{\mathcal{S}^i_e}(x,t),x,t)=\xiptoe(t_+^{\mathcal{\tilde S}^i_e}(x,t),x,t)$ is an inner node $\nu$ of the network, then
	\begin{align*}
		|(F_e)| 
		&= |S_{+,e}^{i+1}(t_+^{\mathcal{ S}^i_e}(x,t), \nu) \e^{-\tfrac{\mu}{2}(s-t_+^{\mathcal{ S}^i_e}(x,t))} -\tilde S_{+,e}^{i+1}(t_+^{\mathcal{\tilde S}^i_e}(x,t), \nu )  \e^{-\tfrac{\mu}{2}(s-t_+^{\mathcal{\tilde  S}^i_e}(x,t))}|\\
		&\le |S_{+,e}^{i+1}(t_+^{\mathcal{ S}^i_e}(x,t), \nu)|\, |\e^{-\tfrac{\mu}{2}(s-t_+^{\mathcal{ S}^i_e}(x,t))} -\e^{-\tfrac{\mu}{2}(s-t_+^{\mathcal{\tilde  S}^i_e}(x,t))}| \\
		& \quad +|S_{+,e}^{i+1}(t_+^{\mathcal{ S}^i_e}(x,t), \nu) -\tilde S_{+,e}^{i+1}(t_+^{\mathcal{\tilde S}^i_e}(x,t), \nu )| \, |\e^{-\tfrac{\mu}{2}(s-t_+^{\mathcal{ \tilde S}^i_e}(x,t))}|\\
		&\le  S_{\max} \tfrac{\mu}{2} \tfrac{L_\lambda}{\ubar\Lambda(S_{\max})} T  L_x  \|\mathcal{S}^i_e-\mathcal{\tilde S}^i_e\|_M\\
		&\quad+C(n) |S_{-}^{i+1}(t_+^{\mathcal{ S}^i_e}(x,t), \nu) -\tilde S_{-}^{i+1}(t_+^{\mathcal{\tilde S}^i_e}(x,t), \nu)|_\infty
	\end{align*}
	with
	\begin{align*}
		&|S_{-,f}^{i+1}(t_+^{\mathcal{ S}^i_e}(x,t), \nu) -\tilde S_{-,f}^{i+1}(t_+^{\mathcal{\tilde S}^i_e}(x,t), \nu )|\\
		&\le |S_{-,f}^{i+1}(0,\xi_-^{\mathcal{S}^i_f}(0,\nu,t_+^{\mathcal{S}^i_e}(x,t)))
		\e^{-\tfrac{\mu}{2}t_+^{\mathcal{ S}^i_e}(x,t)} 
		-\tilde S_{-,f}^{i+1}(0,\xi_-^{\mathcal{\tilde S}^i_f}(0,\nu,t_+^{\mathcal{\tilde S}^i_e}(x,t)))
		\e^{-\tfrac{\mu}{2}t_+^{\mathcal{ \tilde S}^i_e}(x,t)} |\\
		&\quad + \Big|\int_{0}^{t_+^{\mathcal{S}^i_e}(x,t)}
		\e^{-\tfrac{\mu}{2}(t_+^{\mathcal{S}^i_e}(x,t)-r)}
		\Big( -\tfrac{\mu}{2} ( R_{+,f}-R_{-,f} -S_{+,f}^i)(r, \xi_-^{\mathcal{S}^i_f}(r,\nu,t_+^{\mathcal{S}^i_e}(x,t))) \\
		&\hspace{4.5cm} +\sigma  (\mathcal{S}^i_f) (r, \xi_-^{\mathcal{S}^i_f}(r,\nu,t_+^{\mathcal{S}^i_e}(x,t)))  \Big)dr	\\
		&\qquad-\int_{0}^{t_+^{\mathcal{\tilde S}^i_e}(x,t)}
		\e^{-\tfrac{\mu}{2}(t_+^{\mathcal{\tilde S}^i_e}(x,t)-r)}
		\Big( -\tfrac{\mu}{2} ( R_{+,f}-R_{-,f} -\tilde S_{+,f}^i)(r, \xi_-^{\mathcal{\tilde S}^i_f}(r,\nu,t_+^{\mathcal{\tilde S}^i_e}(x,t))) \\
		&\hspace{4.5cm} +\sigma  (\mathcal{\tilde S}^i_f) (r, \xi_-^{\mathcal{\tilde S}^i_f}(r,\nu,t_+^{\mathcal{\tilde S}^i_e}(x,t)))  \Big)dr	\Big|\\
		&\le L_I |\xi_-^{\mathcal{S}^i_f}(0,\nu,t_+^{\mathcal{S}^i_e}(x,t)) -\xi_-^{\mathcal{\tilde S}^i_f}(0,\nu,t_+^{\mathcal{\tilde S}^i_e}(x,t))|
		+B_{\max} \tfrac{\mu}{2} \tfrac{L_\lambda}{\ubar\Lambda(S_{\max})} T  L_x  \|\mathcal{S}^i_e-\mathcal{\tilde S}^i_e\|_M\\
		&\quad + (\tfrac{1}{12}\tfrac{1}{4+C(n)} 3  L_R +2 T L_\sigma L_R) \max_{r\in[0, t_+^{\mathcal{\tilde S}^i_e}(x,t)]}|\xi_-^{\mathcal{S}^i_f}(r,\nu,t_+^{\mathcal{S}^i_e}(x,t)) -\xi_-^{\mathcal{\tilde S}^i_f}(r,\nu,t_+^{\mathcal{\tilde S}^i_e}(x,t))| \\ 
		&\quad+(\tfrac{1}{12}\tfrac{1}{4+C(n)}+2 T L_\sigma)\|\mathcal{S}^i_f-\mathcal{\tilde S}^i_f\|_M 
		+ (3\tfrac{\mu}{2} S_{\max}+\sigma_{\max})(1+\tfrac{\mu}{2}T) \tfrac{L_\lambda}{\ubar\Lambda(S_{\max})} T  L_x  \|\mathcal{S}^i_e-\mathcal{\tilde S}^i_e\|_M,
	\end{align*}
	where we used in the last step the estimation of $|t_+^{\mathcal{ S}^i_e}(x,t) -t_+^{\mathcal{\tilde S}^i_e}(x,t)|$ as in $(F_d)$ and estimated the integral terms similar to \eqref{eq:S-_self-mapping}.
	For the characteristic curves we can estimate for $r\in [0, t_+^{\mathcal{\tilde S}^i_e}(x,t)]$
	\begin{align*}
		&|\xi_-^{\mathcal{S}^i_f}(r,\nu,t_+^{\mathcal{S}^i_e}(x,t)) -\xi_-^{\mathcal{\tilde S}^i_f}(r,\nu,t_+^{\mathcal{\tilde S}^i_e}(x,t)| \\
		&\le |\xi_-^{\mathcal{ S}^i_f}(r,\nu,t_+^{\mathcal{S}^i_e}(x,t))
		-\xi_-^{\mathcal{\tilde S}^i_f}(r,\nu,t_+^{\mathcal{S}^i_e}(x,t))|
		 +|\xi_-^{\mathcal{\tilde S}^i_f}(r,\nu,t_+^{\mathcal{S}^i_e}(x,t))
		-\xi_-^{\mathcal{\tilde S}^i_f}(r,\nu,t_+^{\mathcal{\tilde S}^i_e}(x,t)| \\
		&\underset{\eqref{eq:estimate_xi_Stilde}}{\le} L_\lambda  T  L_x \|\mathcal{S}^i_f-\mathcal{\tilde S}^i_f\|_M
		+ |\xi_-^{\mathcal{\tilde S}^i_f}(r,\nu,t_+^{\mathcal{S}^i_e}(x,t))
		-\xi_-^{\mathcal{\tilde S}^i_f}(r,\xi_-^{\mathcal{\tilde S}^i_f}(t_+^{\mathcal{S}^i_e}(x,t),\nu,t_+^{\mathcal{\tilde S}^i_e}(x,t)),t_+^{\mathcal{ S}^i_e}(x,t))| \\
		&\le L_\lambda  T  L_x \|\mathcal{S}^i_f-\mathcal{\tilde S}^i_f\|_M + L_x L_s \tfrac{L_\lambda}{\ubar\Lambda(S_{\max})} T  L_x  \|\mathcal{S}^i_e-\mathcal{\tilde S}^i_e\|_M.
	\end{align*}
	In the second step we used $\xi_-^{\mathcal{\tilde S}^i_f}(r,\nu,t_+^{\mathcal{\tilde S}^i_e}(x,t))
	=\xi_-^{\mathcal{\tilde S}^i_f}(r,\xi_-^{\mathcal{\tilde S}^i_f}(t_+^{\mathcal{S}^i_e}(x,t),\nu,t_+^{\mathcal{\tilde S}^i_e}(x,t)),t_+^{\mathcal{ S}^i_e}(x,t))$.
	This yields 
	\begin{align*}
		|(F_e)|
		&\le \Big(S_{\max} \tfrac{\mu}{2} +C(n) L_I (\ubar\Lambda(S_{\max}) +L_s L_x)
		 +C(n)B_{\max} \tfrac{\mu}{2}\\
		&\qquad+(\tfrac{1}{4} L_R + 2 T C(n) L_\sigma L_R )(\ubar\Lambda(S_{\max})+ L_s L_x)
		+ C(n) (3\tfrac{\mu}{2} S_{\max}+\sigma_{\max})(1+\tfrac{\mu}{2}T)
		\Big)\\
		&\quad \cdot \tfrac{L_\lambda}{\ubar\Lambda(S_{\max})} T  L_x  \|\mathcal{S}^i-\mathcal{\tilde S}^i\|_{C([0,T]\times\E)^2}
		+(\tfrac{1}{12}+ 2 T C(n) L_\sigma) \|\mathcal{S}^i-\mathcal{\tilde S}^i\|_{C([0,T]\times\E)^2}.
	\end{align*}
	
	\textbf{Case 2c:} If $\xipoe(t_+^{\mathcal{S}^i_e}(x,t),x,t)=\xiptoe(t_+^{\mathcal{\tilde S}^i_e}(x,t),x,t)$ is a boundary node $\nu$ of the network
	and we use the boundary conditions of Case (ii), 
	then 
	\begin{align*}
		|(F_e)| 
		&= |S_{+,e}^{i+1}(t_+^{\mathcal{ S}^i_e}(x,t), \nu) \e^{-\tfrac{\mu}{2}(s-t_+^{\mathcal{ S}^i_e}(x,t))} -\tilde S_{+,e}^{i+1}(t_+^{\mathcal{\tilde S}^i_e}(x,t), \nu )  \e^{-\tfrac{\mu}{2}(s-t_+^{\mathcal{\tilde  S}^i_e}(x,t))}|\\
		&\le |S_{+,e}^{i+1}(t_+^{\mathcal{ S}^i_e}(x,t), \nu)|\, |\e^{-\tfrac{\mu}{2}(s-t_+^{\mathcal{ S}^i_e}(x,t))} -\e^{-\tfrac{\mu}{2}(s-t_+^{\mathcal{\tilde  S}^i_e}(x,t))}| \\
		& \quad +|S_{+,e}^{i+1}(t_+^{\mathcal{ S}^i_e}(x,t), \nu) -\tilde S_{+,e}^{i+1}(t_+^{\mathcal{\tilde S}^i_e}(x,t), \nu )| \, |\e^{-\tfrac{\mu}{2}(s-t_+^{\mathcal{ \tilde S}^i_e}(x,t))}|\\
		&\le  S_{\max} \tfrac{\mu}{2} \tfrac{L_\lambda}{\ubar\Lambda(S_{\max})} T  L_x  \|\mathcal{S}^i_e-\mathcal{\tilde S}^i_e\|_M
		+|S_{+,e}^{i+1}(t_+^{\mathcal{ S}^i_e}(x,t), \nu) -\tilde S_{+,e}^{i+1}(t_+^{\mathcal{\tilde S}^i_e}(x,t), \nu )| 
	\end{align*}
	with
	\begin{align*}
		&|S_{+,e}^{i+1}(t_+^{\mathcal{ S}^i_e}(x,t), \nu) -\tilde S_{+,e}^{i+1}(t_+^{\mathcal{\tilde S}^i_e}(x,t), \nu )| \\
		&\le 3 |S_{-,e}^{i+1}(t_+^{\mathcal{ S}^i_e}(x,t), \nu) -\tilde S_{-,e}^{i+1}(t_+^{\mathcal{\tilde S}^i_e}(x,t), \nu )| 
		+C_b L_I \tfrac{L_\lambda}{\ubar\Lambda(S_{\max})} T  L_x  \|\mathcal{S}^i_e-\mathcal{\tilde S}^i_e\|_M.
	\end{align*}
	As in Case 2b we estimate
	\begin{align*}
		&|S_{-,e}^{i+1}(t_+^{\mathcal{ S}^i_e}(x,t), \nu) -\tilde S_{-,e}^{i+1}(t_+^{\mathcal{\tilde S}^i_e}(x,t), \nu )| \\
		&\le L_I (\ubar\Lambda(S_{\max})+L_x L_s) \tfrac{L_\lambda}{\ubar\Lambda(S_{\max})} T  L_x  \|\mathcal{S}^i_e-\mathcal{\tilde S}^i_e\|_M
		+B_{\max} \tfrac{\mu}{2} \tfrac{L_\lambda}{\ubar\Lambda(S_{\max})} T  L_x  \|\mathcal{S}^i_e-\mathcal{\tilde S}^i_e\|_M\\
		&\quad + (\tfrac{1}{12}\tfrac{1}{4+C(n)} 3  L_R +2 T L_\sigma L_R) (\ubar\Lambda(S_{\max})+L_x L_s) \tfrac{L_\lambda}{\ubar\Lambda(S_{\max})} T  L_x  \|\mathcal{S}^i_e-\mathcal{\tilde S}^i_e\|_M \\ 
		&\quad+(\tfrac{1}{12}\tfrac{1}{4+C(n)}+2 T L_\sigma)\|\mathcal{S}^i_f-\mathcal{\tilde S}^i_f\|_M 
		+ (3\tfrac{\mu}{2} S_{\max}+\sigma_{\max})(1+\tfrac{\mu}{2}T) \tfrac{L_\lambda}{\ubar\Lambda(S_{\max})} T  L_x  \|\mathcal{S}^i_e-\mathcal{\tilde S}^i_e\|_M\\
		&\le \tfrac{L_\lambda}{\ubar\Lambda(S_{\max})} T  L_x  \|\mathcal{S}^i_e-\mathcal{\tilde S}^i_e\|_M
		\Big(  L_I (\ubar\Lambda(S_{\max})+L_x L_s) 
		+B_{\max} \tfrac{\mu}{2} \\
		&\qquad + (\tfrac{1}{4}\tfrac{1}{4+C(n)}  L_R +2 T L_\sigma L_R) (\ubar\Lambda(S_{\max})+L_x L_s)  
		+ (3\tfrac{\mu}{2} S_{\max}+\sigma_{\max})(1+\tfrac{\mu}{2}T) \Big)\\
		&\quad+(\tfrac{1}{12}\tfrac{1}{4+C(n)}+2 T L_\sigma)\|\mathcal{S}^i_f-\mathcal{\tilde S}^i_f\|_M 
	\end{align*}
	i.e.,
	\begin{align*}
		|(F_e)|
		&\le S_{\max} \tfrac{\mu}{2} \tfrac{L_\lambda}{\ubar\Lambda(S_{\max})} T  L_x  \|\mathcal{S}^i_e-\mathcal{\tilde S}^i_e\|_M
		+ (\tfrac{1}{4}\tfrac{1}{4+C(n)}+6 T L_\sigma)\|\mathcal{S}^i_f-\mathcal{\tilde S}^i_f\|_M \\
		&\quad +3 \tfrac{L_\lambda}{\ubar\Lambda(S_{\max})} T  L_x  \|\mathcal{S}^i_e-\mathcal{\tilde S}^i_e\|_M
		\Big(  L_I (\ubar\Lambda(S_{\max})+L_x L_s) 
		+B_{\max} \tfrac{\mu}{2} \\
		&\qquad + (\tfrac{1}{4}\tfrac{1}{4+C(n)}  L_R +2 T L_\sigma L_R) (\ubar\Lambda(S_{\max})+L_x L_s)  
		+ (3\tfrac{\mu}{2} S_{\max}+\sigma_{\max})(1+\tfrac{\mu}{2}T) \Big)\\
		&\quad + C_b L_I \tfrac{L_\lambda}{\ubar\Lambda(S_{\max})} T  L_x  \|\mathcal{S}^i_e-\mathcal{\tilde S}^i_e\|_M.
	\end{align*}

	\textbf{Case 3:} $t_+^{\mathcal{ S}^i_e}(x,t)=0$, $t_+^{\mathcal{\tilde S}^i_e}(x,t)>0$.\newline
	Then $\xiptoe(t_+^{\mathcal{\tilde S}^i_e}(x,t),x,t)=0$ is either a boundary node or an inner node. Again, this case can be treated similar to Case 2.

	In all Cases 1, 2 and 3 we can choose $S_{\max}$, $B_{\max}$, $L_I$ and $L_R$ sufficiently small such that
	$|(F_e)|\le \tfrac{1}{6} \|\mathcal{S}^i-\mathcal{\tilde S}^i\|_{C([0,T]\times\E)^2}$.
	Together with the estimates for terms $(F_a)-(F_d)$ above this shows that the
	fixed-point map $\Phi$ is a contraction with contraction constant bounded by $\tfrac{2}{3}+\tfrac{1}{6}<1$, if $S_{\max}$, $B_{\max}$, $L_I$ and $L_R$ are sufficiently small.
	Therefore the Banach fixed-point theorem shows the existence of a unique fixed-point, which solves the observer system \eqref{eq:obs_RI_v-measurement}
	together with the initial conditions \eqref{eq:obs_RI_IC}, the boundary conditions of Case (i) or (ii) 
	and the coupling conditions \eqref{eq:cc_RI_m}-\eqref{eq:cc_RI_h}.	
\end{proof}

Using Theorem \ref{thm:localExistence} we can now show the existence of semi-global solutions on any finite time interval $[0,T]$ (cf. Theorem 6.1 in \cite{GugatUlbrich2018}).

\begin{thm} \label{thm:semi-global_existence}
	Let some $T>0$ be given and assume that the conditions of Theorem \ref{thm:localExistence} (apart from the restriction on $T$) are satisfied.
	Then there exist a constant $C_T>0$ and constants $c_1(T),\, c_2(T),\, c_3(T)$, $c_4(T)>0$ such that, if
	\begin{align*}
	B_{\max}\le c_1(T),\quad L_I\le c_2(T),\quad
	S_{\max}\le c_3(T), \ \text{ and } L_R \le c_4(T),
	\end{align*}
	then the observer system \eqref{eq:obs_RI_v-measurement} for velocity measurements together with the
	initial conditions~\eqref{eq:obs_RI_IC}, the boundary conditions of Case (i) or (ii)
	and the coupling conditions \eqref{eq:cc_RI_m}-\eqref{eq:cc_RI_h} admits a solution on $[0,T]$ satisfying
	\begin{align*}
	\|S_\pm\|_{L^\infty([0,T]\times[0,\ell])} \le C_T.
	\end{align*}
\end{thm}

\begin{proof}
	Let $T_1>0$ such that $T_1<\min_{e\in\E}\tfrac{\ell^e}{\Lambda(S_{\max})}$ and $1-\e^{-\tfrac{\mu}{2} T_1}\le \tfrac{1}{12}\min\{1,\tfrac{1}{4+C(n)}\tfrac{4}{9}\tfrac{\ubar\Lambda(S_{\max})}{\Lambda(S_{\max})}\}$.
		The main idea of the proof of Theorem \ref{thm:semi-global_existence} is to apply Theorem \ref{thm:localExistence} $\lceil\tfrac{T}{T_1}\rceil$-times.
		
		First, let us note that in addition to the bounds on $B_{\max}$, $L_I$, $S_{\max}$ and $L_R$ in the proof of Theorem \ref{thm:localExistence} we have also bounds on these variables that depend on each other, which can be summarized as
		\begin{align}\label{eq:bounds_globalExistence}
			B_{\max}\le \min\{\tilde c_1 S_{\max}, \tilde c_2 L_R \}, \qquad
			L_I\le \tilde c_3 L_R , \qquad
			S_{\max} \le \min\{\tilde c_4 L_R , \tilde c_5 \sqrt{L_R}\}.
		\end{align}
		Now, consider the interval $[0,T_1]$ and let some initial and boundary data be given that are bounded by $B_{\max,1}$ and are Lipschitz-continuous with respect to $x$ with Lipschitz-constant bounded by $L_{I,1}$. Let
		\begin{align*}
			S_{\max,1}&:=\tfrac{1}{\tilde c_1}B_{\max,1}=:f_1(B_{\max,1}), \\
			L_{R,1}&:=\max\{\tfrac{1}{\tilde c_2}B_{\max,1}, \tfrac{1}{\tilde c_3} L_{I,1}, \tfrac{1}{\tilde c_4} \tfrac{1}{\tilde c_1}B_{\max,1},
			\tfrac{1}{\tilde c_5^2} \tfrac{1}{\tilde c_1^2}B_{\max,1}^2\}=:f_2(B_{\max,1}, L_{I,1}).
		\end{align*}
		Then, $B_{\max,1}$, $L_{I,1}$, $S_{\max,1}$ and $L_{R,1}$ satisfy the bounds \eqref{eq:bounds_globalExistence} and we can choose 
		$B_{\max,1}$ and $L_{I,1}$ sufficiently small such that
		$B_{\max,1}$, $L_{I,1}$, $S_{\max,1}$ and $L_{R,1}$ satisfy the bounds of Theorem  \ref{thm:localExistence}, i.e., Theorem  \ref{thm:localExistence} provides existence of solutions $S_\pm$ on the time interval $[0,T_1]$.
		
		Next, we consider the time interval $[T_1, 2 T_1]$ and use $S_\pm(T_1)$ as initial data, which is bounded by $B_{\max,2}:=S_{\max,1}$ and its Lipschitz-constant is bounded by $L_{I,2}\le L_{R,1}$.
		Again we define
		\begin{align*}
			S_{\max,2}:=f_1(B_{\max,2}), \quad
			L_{R,2}:=f_2(B_{\max,2}, L_{I,2}).
		\end{align*}
		In particular,  $S_{\max,2}$ and $L_{R,2}$ can be written as a monotonically increasing function of $B_{\max,1}$ and $L_{I,1}$.
		Thus, we can choose $B_{\max,1}$ and $L_{I,1}$ sufficiently small such that
		Theorem  \ref{thm:localExistence} provides the existence of solutions also on the time interval $[T_1, 2 T_1]$.
		Proceeding in the same way, we can show existence of a solution $\mathcal{ S}=(S_+, S_-)$ on the time interval $[0,T]$ satisfying
		\begin{align*}
			\|S_\pm\|_{L^\infty(0,T;L^\infty(0,\ell^e))}\le S_{\max,N}=: C_T
		\end{align*}
		with $N=\lceil\tfrac{T}{T_1}\rceil$,
		if $B_{\max}:=B_{\max,1}\le c_1(T)$, $L_I:=L_{I,1}\le c_2(T)$, $S_{\max}:=S_{\max,1}\le c_3(T)$ and $L_R:=L_{R,1}\le c_4(T)$ with $c_1(T),\, c_2(T),\, c_3(T)$, $c_4(T)$ sufficiently small.	
\end{proof}

\section{Exponential synchronization} \label{sec:exp_synchronization} 
In this section, we show convergence of solutions of the observer system \eqref{eq:observer_1}-\eqref{eq:observer_2} towards solutions of the original system \eqref{eq:system_1}-\eqref{eq:system_2} for the case of measurement of one of the fields velocity, density or mass flow, i.e., for observer terms given by \eqref{eq:observer-term_v}, \eqref{eq:observer-term_rho} or \eqref{eq:observer-term_m}.
In general, \eqref{eq:system_1}-\eqref{eq:system_2} and \eqref{eq:observer_1}-\eqref{eq:observer_2} have different initial data, since usually only an approximation of the initial data of the exact solution is known.
In the first part of this section, we consider the observer system on a single pipe with length $\ell$ and
complement the equations for the original system and the observer system by the boundary conditions
\begin{align}\label{eq:BC}
	m(t,0)=\hat m(t,0)=m_b(t), \qquad
	h(t,\ell)=\hat h(t,\ell)=h_b(t), \qquad 0\le t\le T,
\end{align}
for the mass flow $m$ at the left end of the pipe and the total specific enthalpy $h$ at the right end of the pipe.

For the proof of exponential synchronization
we use
the following assumptions:
\begin{itemize}
	\item[(A1)] \textbf{Bounded state solution of original system:}
	For the system \eqref{eq:system_1}-\eqref{eq:system_2} together with the boundary conditions \eqref{eq:BC} there exists a solution
	\begin{align*} 
		(\rho,v) \in\big(C^0([0,T];W^{1,\infty}(0,\ell))\cap W^{1,\infty}(0,T;L^{\infty}(0,\ell)) \big)^2
	\end{align*}
	that satisfies \eqref{eq:system_1}-\eqref{eq:system_2} in a pointwise sense a.e. 
	and there exist $\ubar\rho, \bar \rho, \tilde v>0$ so that
	\begin{align*} 
	0 < \underline{\rho}\le \rho (t,x)\le \bar \rho
	\qquad \text{and} \qquad 
	-\tilde v \le v(t,x) \le \tilde v
	\end{align*}
	for all $0 \le x \le \ell$ and $0 \le t \le T$.
	\item[(A2)] \textbf{Subsonic condition:} The velocity bound $\tilde v$ satisfies
	$$
	\rho P''(\rho)\ge 4|\tilde v|^2 \quad \forall \  \underline{\rho} \leq \rho \leq \overline{\rho}.
	$$	
	\item[(A3)] \textbf{Bounded state solution of observer system:}
	The observer system \eqref{eq:observer_1}-\eqref{eq:observer_2} with the boundary conditions \eqref{eq:BC} admits a solution
	\begin{align*} 
		(\hat \rho, \hat v) \in\big(C^0([0,T];W^{1,\infty}(0,\ell))\cap W^{1,\infty}(0,T;L^{\infty}(0,\ell)) \big)^2
	\end{align*}
	that satisfies 
	\begin{align*} 
	0 < \underline{\rho}\le  \hat \rho (t,x) \le \bar \rho
	\qquad \text{and} \qquad 
	-\tilde v \le  \hat v(t,x) \le \tilde v
	\end{align*}
	for all $0 \le x \le \ell$ and $0 \le t \le T$.
	\item[(A4)]
	\textbf{Smallness-condition:}
	Assume that both solutions satisfy the following a priori bounds
	\begin{align*}
		\| \dt \rho\|_{L^\infty(0,T;L^\infty(0,\ell^e))} + \| \dt v\|_{L^\infty(0,T;L^\infty(0,\ell^e))} &\le C_t, \\
		\|v\|_{L^\infty(0,T;L^\infty(0,\ell^e))}, \|\hat v\|_{L^\infty(0,T;L^\infty(0,\ell^e))} &\le \bar v,
	\end{align*}
	for $C_t, \bar v>0$ sufficiently small.
	This means that both solutions have small velocity and the exact solution has sufficiently small time derivatives.
\end{itemize}
These assumptions are reasonable in practice, since in the usual operational range of gas pipes the gas velocities are small and only slow changes in the gas flow occur.

\begin{remark}
	The existence result stated in Section \ref{sec:existence} shows the existence of solutions of the original system and, for measurements of velocity or density, of the observer system satisfying the assumptions (A1)-(A4),
	provided that the initial and boundary conditions satisfy the assumptions of Theorem~\ref{thm:localExistence} for sufficiently small constants $B_{\text{max}}, L_I, S_{\text{max}}, L_R>0$.
	This can be seen as follows. 
	Theorem \ref{thm:localExistence} states the existence of a solution $S_{\pm,e}\in C^0([0,T]\times[0,\ell^e])$ of the observer system that satisfies the bound $|S_\pm|\le S_{\text{max}}$ for a constant $S_{\text{max}}>0$ and is
	Lipschitz-continuous with respect to $x$, uniformly in time. 
	This implies $S_{\pm,e}\in C^0([0,T];W^{1,\infty}(0,\ell^e))$ 
	and therefore also $(\hat\rho,\hat v)\in C^0([0,T];W^{1,\infty}(0,\ell^e))$. 
	In particular, this means that $\hat\rho,\, \hat v$  are differentiable a.e. in $(0,\ell^e)$ with respect to $x$ (cf. Rademacher's Theorem). 
	In addition, the bound $|S_\pm|\le S_{\text{max}}$
	yields the bounds 
	\begin{align} \label{eq:bounds_rho_v}
	|\hat v|\le c S_{\max}=:\bar v,\quad
	0<\ubar\rho:=\tilde P^{-1}(-S_{\max})\le \hat\rho\le \tilde P^{-1}(S_{\max})=:\bar{\rho}
	\end{align}
	for $\hat \rho$ and $\hat v$ (and similar for $\rho$, $v$) if $S_{\text{max}}$ is sufficiently small, see Remark \ref{rem:bounds_rho-v}. 
	Using these bounds as well 
	the equations \eqref{eq:observer_1}-\eqref{eq:observer_2} describing the observer system, we can show $ \dt\hat{\rho}, \dt\hat v\in  L^\infty(0,T;L^\infty(0,\ell^e))$, i.e., 
	\begin{align*}
		(\hat\rho,\hat v)\in C^0([0,T];W^{1,\infty}(0,\ell^e))\cap W^{1,\infty}(0,T;L^{\infty}(0,\ell^e)).
	\end{align*} 
	Analogously we can show that the solution $(\rho, v)$ of the original system is in the same space.
	In addition, \eqref{eq:bounds_rho_v} implies that for $S_{\text{max}}$ sufficiently small Theorem \ref{thm:localExistence} yields the bounds on $\rho,\, v, \, \hat \rho,\, \hat v$ that are required in (A1) and (A3) for the synchronization.
	Since $|v|,\, |\hat v|\le c S_{\text{max}}$, the subsonic condition (A2) is also satisfied for $S_{\text{max}}$ sufficiently small. 
	Finally, using the equations \eqref{eq:system_1}-\eqref{eq:system_2} we can see that
	\begin{align*}
	\| \dt \rho\|_{L^\infty(0,T;L^\infty(0,\ell^e))} + \| \dt v\|_{L^\infty(0,T;L^\infty(0,\ell^e))} \le C(S_{\text{max}}) \, L_R + \tilde C \, S_{\max}^2
	\end{align*}
	with a constant $C(S_{\text{max}})>0$ depending on $S_{\text{max}}$ and a constant $\tilde C>0$ (depending on $\gamma$ and $c$), i.e., the smallness-condition (A4) is satisfied if $L_R$ and $S_{\text{max}}$ are sufficiently small. 
	
	To summarize, as in the proof of Theorem \ref{thm:semi-global_existence} we can show that for given $S_{\text{max}},\, L_R>0$, we can choose $B_{\text{max}},\, L_I>0$ sufficiently small such that,
	if the initial data are bounded by $B_{\text{max}}$, the boundary data satisfy the bounds \eqref{eq:BC_bound_mb}-\eqref{eq:BC_bound_hb} and the initial and boundary data are Lipschitz-continuous with Lipschitz-constant bounded by $L_I$,
	then  there exists a solution ${S_{\pm,e}\in C^0([0,T]\times[0,\ell^e])}$ of \eqref{eq:observer_1}-\eqref{eq:observer_2} (or \eqref{eq:system_1}-\eqref{eq:system_2}) satisfying the initial conditions \eqref{eq:obs_RI_IC} (that can be translated into conditions for $\rho_0, v_0$) and the boundary conditions \eqref{eq:BC}.
	The solution is bounded by $S_{\text{max}}$
	and is Lipschitz-continuous with respect to $x$ with Lipschitz-constant bounded by $L_R$.
	In other words, if the initial and boundary data satisfy suitable smallness and smoothness assumptions, then there exist solutions that satisfy the bounds in (A1)-(A4).
\end{remark}

Now we can state the main theorem of this section.
\begin{thm}\label{thm:ExpSynchronization}
	Consider a solution $u=(\rho, v)$  of \eqref{eq:system_1}-\eqref{eq:system_2} and a solution $\hat u=(\hat \rho, \hat v)$ of \eqref{eq:observer_1}-\eqref{eq:observer_2} with observer terms given by \eqref{eq:observer-term_v}, \eqref{eq:observer-term_rho} or \eqref{eq:observer-term_m} satisfying the boundary conditions \eqref{eq:BC}, the properties stated in (A1)-(A3) and the bounds in (A4) with $C_t, \bar v>0$ sufficiently small.
	
	Then there exist constants $C_1, C_2>0$ such that
	\begin{align*}
	\| u(t,\cdot)-\hat u(t,\cdot)\|_{L^2(0,\ell)} \le C_1 
	\| u_0-\hat u_0\|_{L^2(0,\ell)} \e^{-C_2t}
	\end{align*}
	for all $0\le t\le T$.
\end{thm}

\begin{remark}
	Note that we have the bound $\tilde v$ on the velocity, which asserts that we have subsonic flow, see assumption (A2), and in addition the bound $\bar v\le \tilde v$ in Theorem \ref{thm:ExpSynchronization} ensures that the velocity is sufficiently small in order to control certain terms in our analysis.
\end{remark}

In other words, Theorem \ref{thm:ExpSynchronization} asserts that the state of the observer system converges exponentially towards the exact solution, provided that the time derivative of the exact solution and the velocity in both systems is sufficiently small.
The rest of this section is devoted to the proof of this theorem and to the extension to star-shaped networks.

\subsection{General strategy of the proof of Theorem \ref{thm:ExpSynchronization}}
In order to measure the distance between the exact solution and the solution of the observer system, we will use the 
\emph{relative energy}, which is defined by
\begin{equation*}
	\H(\hat u|u):= \H(\hat u)-\H(u)-\langle \H'(u),\hat u- u\rangle,
\end{equation*}
cf. \cite{Dafermos1979, EggerGiesselmann}. 
Here, $u=(\rho, v)$, $\hat u=(\hat \rho, \hat v)$ denote the solution of the original system and of the observer system, respectively,
$\langle \cdot, \cdot \rangle$ denotes the $L^2$-scalar product on $[0,\ell]$ and
$\H'(u)$ is the
variational derivative of the energy $\H$, see \eqref{eq:H},
that is given by
\begin{align}
	\delta_\rho \mathcal{H}(\rho, v)=\frac{1}{2} v^2 +P'(\rho) = h,\qquad
	\delta_v \mathcal{H}(\rho, v)= m. \label{eq:varDer_H}
\end{align}
Using these derivatives we can write the system \eqref{eq:system_1}-\eqref{eq:system_2} in the Hamiltonian form
\begin{equation} \label{eq:Hamiltonian_structure}
	\begin{pmatrix}
	\dt \rho\\ \dt v
	\end{pmatrix}
	+
	\begin{pmatrix}
		0 & \dx \\ \dx & 0
		\end{pmatrix}
	\begin{pmatrix}
	\delta_\rho \H \\ \delta_v \H	
	\end{pmatrix}
	=
	\begin{pmatrix}
	0 \\ -\gamma |v| v  
	\end{pmatrix},
\end{equation}
which will be crucial in the estimation of the time derivative of the relative energy.

Due to the convexity of the pressure potential $P$, the relative energy is positive and defines a distance measure on subsonic states. The proof of the following result can be found in \cite[Lemma 9]{EggerGiesselmann}.

\begin{lem} \label{lem:normequ}
	Under the assumptions (A1)--(A3) there exist constants $c_0, C_0>0$ such that
	\begin{align*}
		c_0 \| u_1-u_2\|_{L^2(0,\ell)}^2 \le \H(u_1 | u_2) \le C_0\| u_1-u_2\|_{L^2(0,\ell)}^2
	\end{align*}
	for all $u_1=(\rho_1, v_1), u_2=(\rho_2, v_2)\in L^2(0,\ell)$ with
	\begin{align*} 
		0 < \underline{\rho}\le \rho_1 (x),  \rho_2 (x) \le \bar \rho, \qquad 
		-\tilde v \le v_1(x), v_2(x) \le \tilde v,\qquad \forall x\in(0,\ell). 
	\end{align*}
\end{lem}

The main idea of the proof of Theorem \ref{thm:ExpSynchronization} is to use an extension of the relative energy method, i.e., we estimate $\dt \H(\hat u|u)$ using a similar strategy as in \cite{EggerGiesselmann} and then, motivated by an extension of the energy functional used in \cite{EggerKugler2018}, we add a suitable functional to the relative energy in order to get decrease also with respect to the variables that are not measured.
As a first step, we estimate the time derivative of the relative energy.

\begin{lem} \label{lem:dtH}
	Under the assumptions of Theorem \ref{thm:ExpSynchronization} the time derivative of the relative energy can be bounded by
	\begin{equation} \label{eq:dtH}
	\begin{split}
		&\dt \H(\hat u|u)\le 
		\,2 \gamma  \left(\max\{\|v\|_{L^\infty([0,\ell])}, \|\hat v\|_{L^\infty([0,\ell])}  \}\right)^3 \tfrac{1}{\ubar \rho}  
		\|\rho-\hat \rho\|_{L^2([0,\ell]) }^2 \\
		&\quad +\tilde C \, (\| \dt \rho\|_{L^\infty([0,\ell])} + \| \dt v\|_{L^\infty([0,\ell])}) \, \|u-\hat u\|_{L^2([0,\ell]) }^2 \
		-\langle \mathcal{L}_\rho, h-\hat h\rangle  -\langle \mathcal{L}_v, m-\hat m\rangle 
	\end{split}
	\end{equation}
	for a.e. $t\in(0,T)$, where
	$\tilde C:= \max(\tfrac{1}{2},  C_{P'''} )$.
\end{lem}

\begin{proof}
	Using the definition of the relative energy, we can compute
	\begin{align*}
		\dt \H(\hat u|u)= &
		\langle \H'(\hat u), \dt \hat u\rangle - \langle \H'(u), \dt u\rangle
		-\langle \H''(u)\dt u, \hat u- u\rangle
		-\langle \H'(u), \dt \hat u -\dt u \rangle\\
		= & \langle\H'(\hat u)- \H'(u), \dt \hat u -\dt u \rangle
		 + \langle \H'(\hat u)-\H'(u)-\H''(u)(\hat u -u), \dt u\rangle.
	\end{align*}
	Now we can use that the variational derivatives of the energy $\H$ are given by \eqref{eq:varDer_H}.
	Together with 
	the equations \eqref{eq:system_1}--\eqref{eq:system_2} and \eqref{eq:observer_1}--\eqref{eq:observer_2}, this leads to
	\begin{align}
	\dt \H(\hat u|u)
		= & \langle h-\hat h, \dt \rho- \dt\hat \rho \rangle
			+ \langle m-\hat m, \dt v-\dt \hat v\rangle
			+ \langle \H'(\hat u)-\H'(u)-\H''(u)(\hat u -u), \dt u\rangle\notag\\
		= & - \langle m-\hat m, \gamma |v| v-\gamma |\hat v| \hat v\rangle 
			+ \langle \H'(\hat u)-\H'(u)-\H''(u)(\hat u -u), \dt u\rangle \label{eq:dtH2} \\
			&-\langle \mathcal{L}_\rho, h-\hat h\rangle  -\langle \mathcal{L}_v, m-\hat m\rangle ,  \notag
	\end{align}
	where we have used
	\begin{align*}
	    \langle h-\hat h, \dx \hat m - \dx m\rangle + \langle\dx  h- \dx \hat h,  \hat m - m\rangle
	    = (h-\hat h)(\hat m-m)\vert_{x=0}^\ell =0
	\end{align*}
	due to the boundary conditions for $m$ and $h$. 
	For the first term we note that $ |v| v-|\hat v| \hat v =2(v-\hat v) \int_0^1 |\hat v +s (v-\hat v)| ds  $ with 
	\begin{align*}
		\tfrac{1}{4} (|\hat v | + |v|) \le \int_0^1 |\hat v +s (v-\hat v)| ds \le \tfrac{1}{2} (|\hat v | + |v|),
	\end{align*}
	compare \cite{EggerGiesselmann}. This implies
	\begin{align*}
		&- \langle m-\hat m, \gamma |v| v-\gamma |\hat v| \hat v\rangle 
		\le \int_0^\ell \gamma  |v| |\rho-\hat \rho| |v-\hat v| (|v|+|\hat v|) dx
			- \int_0^\ell \tfrac{1}{2} \gamma  \hat \rho (v-\hat v)^2  (|v|+|\hat v|) dx\\
		&\quad\le 
		 2 \gamma \left(\max\{\|v\|_{L^\infty([0,\ell])}, \|\hat v\|_{L^\infty([0,\ell])}  \}\right)^3 \tfrac{1}{\ubar \rho}
			\|\rho-\hat \rho\|_{L^2([0,\ell])}^2  
			-\int_0^\ell \tfrac{1}{4} \gamma \hat \rho  (|v| +|\hat v|)(v-\hat v)^2 dx. 
	\end{align*}
	where the second term in the last line is nonpositive.
	Since 
	\begin{align*}
		\H'(\hat u)-\H'(u)-\H''(u)(\hat u -u) =\begin{pmatrix}
		P'(\hat \rho|\rho ) + \tfrac{1}{2} (\hat v- v^2 )\\
		(\hat \rho-\rho)(\hat v -v)
		\end{pmatrix},
	\end{align*}
	the second term in \eqref{eq:dtH2} can be estimated by
	\begin{align*}
		&\langle \H'(\hat u)-\H'(u)-\H''(u)(\hat u -u), \dt u\rangle\\
		&\le \max(\tfrac{1}{2}, C_{P'''} ) \, (\| \dt \rho\|_{L^\infty([0,\ell])} + \| \dt v\|_{L^\infty([0,\ell])}) \, \|u-\hat u\|_{L^2([0,\ell]) }^2. \qedhere
	\end{align*}
\end{proof}

We expect that the observer terms $-\langle \mathcal{L}_\rho, h-\hat h\rangle  -\langle \mathcal{L}_v, m-\hat m\rangle$ in the estimate of $\dt \H(\hat u|u)$ lead to decrease with respect to the measured variables. In order to get decrease also with respect to the variables that are not measured, we introduce an additional functional $\G$ (cf. the extension of the energy functional in \cite{EggerKugler2018}).

\begin{lem} \label{lem:dtH_2}
	Assume that the assumptions of Theorem \ref{thm:ExpSynchronization} are satisfied and assume that 
	there are observer terms $\mathcal{L}_\rho, \mathcal{L}_v$ and a functional $\mathcal{G}$ 
	so that
	\begin{align}\label{eq:normeq-HG}
	\tfrac{c_0}{2}\| u(t,\cdot)-\hat u(t,\cdot)\|_{L^2(0,\ell)}^2
	\le \H(\hat u|u)(t)+ \G(\hat u|u)(t)
	\le \tfrac{3}{2} C_0 \| u(t,\cdot)-\hat u(t,\cdot)\|_{L^2(0,\ell)}^2  
	\end{align}
	and there exists a constant $\bar c>0$ so that
	\begin{equation}\label{eq:estimatecH}
	\begin{split}
	&-\langle \mathcal{L}_\rho(t,\cdot), h(t,\cdot)-\hat h(t,\cdot)\rangle  -\langle \mathcal{L}_v(t,\cdot), m(t,\cdot)-\hat m(t,\cdot)\rangle   + \partial_t \G(\hat u|u)(t)  \\
	&\leq - \bar c \left(\H(\hat u|u)(t)+ \G(\hat u|u)(t)\right)
	\end{split}
	\end{equation}
	for a.e. $0\le t\le T$.

	Then there exist constants $C_1, C_2>0$ such that
	\begin{align*}
	\| u(t,\cdot )-\hat u(t,\cdot) \|_{L^2(0,\ell)} \le C_1 
	\| u_0-\hat u_0\|_{L^2(0,\ell)} \e^{-C_2t}.
	\end{align*}
\end{lem}

\begin{proof}
    Combining \eqref{eq:estimatecH} with Lemma \ref{lem:dtH} yields
    \begin{align*}
        \dt &\H(\hat u|u)(t)+\dt \G(\hat u|u)(t)\\
        \le &\,2 \gamma  \tfrac{\bar v^3}{\ubar\rho}\,
		\Lnorm{\rho(t,\cdot)-\hat \rho(t, \cdot)}
		+\tilde C C_t \|u(t,\cdot)-\hat u(t,\cdot)\|_{L^2 }^2 
		- \bar c \left(\H(\hat u|u)(t)+ \G(\hat u|u)(t)\right).
    \end{align*}
    By Lemma \ref{lem:normequ}, the first and second term can be absorbed into the last term for sufficiently small $C_t, \bar v>0$. This yields
    \begin{align*}
       \dt \H(\hat u|u)(t)+\dt \G(\hat u|u)(t)
       \le -C \left(\H(\hat u|u)(t)+ \G(\hat u|u)(t)\right)
    \end{align*}
    for some $C>0$.
    Then, using \eqref{eq:normeq-HG} and the Gronwall Lemma shows the assertion.
\end{proof}

\begin{remark}
		Note that $(\rho, v),\, (\hat \rho, \hat v) \in\big(C^0([0,T];W^{1,\infty}(0,\ell))\cap W^{1,\infty}(0,T;L^{\infty}(0,\ell)) \big)^2$, i.e., the derivatives $\dt \H(\hat u|u)(t)$ etc. in Lemma \ref{lem:dtH} and Lemma \ref{lem:dtH_2} are defined in an a.e.-sense.
		In particular, inequality \eqref{eq:dtH} holds for a.e. $t\in(0,T)$.
		The Gronwall Lemma can still be applied in this setting.
\end{remark}

\begin{remark}
		We expect that the exponential convergence stated in Lemma \ref{lem:dtH_2} can be generalized to other $2\times2$--systems of hyperbolic conservation laws
		that have a Hamiltonian structure similar to \eqref{eq:Hamiltonian_structure} with convex energy functional.
		In this case, one would have to find a suitable functional~$\G$ and check conditions similar to the conditions required in Lemma~\ref{lem:dtH} and Lemma~\ref{lem:dtH_2}. However, it is not clear to us how to find such a functional $\G$ in the general case.
		
		The analysis in this paper is based on the relative energy method.
		This is similar to the convergence analysis of numerical schemes as in \cite{Egger2023_RelativeEnergysiDiscrete}.	
		Thus, it is interesting to combine these techniques to show synchronization of discretized observers. 
		However, this is not straightforward and beyond the scope of the current paper.
		
\end{remark}

In the following, we state possible functionals $\G$ that allow to verify the conditions \eqref{eq:normeq-HG} and \eqref{eq:estimatecH} for the observer terms \eqref{eq:observer-term_v}, \eqref{eq:observer-term_rho} and \eqref{eq:observer-term_m} corresponding to 
measurements of $m$, $v$ or $\rho$, respectively. For velocity measurements we present the proof in detail, while for density and mass flow measurements we just present the main idea of the proof.
Once it is shown that the conditions \eqref{eq:normeq-HG} and \eqref{eq:estimatecH} are satisfied, Theorem~\ref{thm:ExpSynchronization} follows directly from Lemma~\ref{lem:dtH_2}.

\subsection{Proof for measurement of $v$} \label{sec:meas_v}

First, we show exponential convergence of the state of the observer system towards the original system state for measurement of the velocity.
In this case, we use the observer terms
\begin{align*}
	\mathcal{L}_\rho=0, \qquad \mathcal{L}_v=\mu  (v-\hat v)
\end{align*}
with $\mu>0$.
Under the assumption that $	|v|, |\hat v| \le \bar v$, we can estimate 
\begin{equation} \label{eq:estimate_Rr_Rv_v}
\begin{split}
	-\langle \mathcal{L}_\rho, h-\hat h\rangle  -\langle \mathcal{L}_v, m-\hat m\rangle  
	&\le -\mu  \ubar \rho  \Lnorm{v-\hat v} -\langle \mu (v-\hat v), \hat v  (\rho-\hat \rho)\rangle\\
	&\le -\tfrac{1}{2} \mu \ubar \rho  \Lnorm{v-\hat v}
	+\tfrac{1}{2} \mu \frac{\bar v^2}{\ubar \rho} \Lnorm{\rho-\hat\rho}.
\end{split}
\end{equation}
In order to show decrease with respect to $\Lnorm{\rho-\hat\rho}$, we use the following additional functional, where the construction of $\F_1$ is inspired by the construction of the additional functional used for a linear problem in \cite{EggerKugler2018}.

\begin{definition} \label{def:G}
	Let
	\begin{align*}
		\F_1(\hat u|u)(t):=\int_0^\ell (M(x,t)-\hat M(x,t)) (v(x,t)-\hat v(x,t)) dx
	\end{align*}
	with
	\begin{align*}
		M(x,t):=\int_0^t m(x, t') dt' - \int_0^x \rho_0(x') dx'.
	\end{align*}
\end{definition}

The strategy for the proof of Theorem \ref{thm:ExpSynchronization} in the case of measurement of $v$ is that we will estimate
${\dt \H(\hat u|u) +\delta \dt \F_1(\hat u|u)}$ with a constant $\delta>0$ and show that $\dt \F_1(\hat u|u)$ guarantees decrease with respect to
${\|\rho-\hat \rho\|_{L^2(0,\ell)}^2}$,
i.e., we will use Lemma \ref{lem:dtH_2} and verify the conditions \eqref{eq:normeq-HG} and \eqref{eq:estimatecH} for $\G:=\delta \F_1$. 
As a first step, let us note that for $\delta$ sufficiently small, the term $\H(\hat u|u) + \delta \F_1(\hat u|u)$ is equivalent to 
$\| u-\hat u\|_{L^2(0,\ell)}^2$.

\begin{lem} \label{lem:normequ_v}
	Suppose that the assumptions (A1)--(A3) hold and 
	\begin{align*}
		\delta \le \frac{c_0}{C_{Poin} \ell}
	\end{align*}
	with the constant $C_{Poin}$ from the Poincar\'{e} inequality. 
	Then
	\begin{align*}
		\tfrac{c_0}{2} \| \hat u(t, \cdot)- u(t, \cdot)\|_{L^2(0,\ell)}^2 \le \H( \hat u| u)(t) +\delta \F_1( \hat u| u)(t) \le \tfrac{3}{2} C_0\|\hat u(t, \cdot)- u(t, \cdot)\|_{L^2(0,\ell)}^2\qquad 
	\end{align*}
	for all $0\le t\le T$.
\end{lem}

\begin{proof}
	Since 
	\begin{align*}
	M(0,t)=\int_0^t m(0, t') dt' = \int_0^t m_b(t') dt' = \hat M(0,t),
	\end{align*}
	we can apply the Poincar\'{e} inequality to $\Lnorm{M-\hat M}$. Therefore we can compute
	\begin{align*}
	|\F_1(\hat u|u)|&\le \|M-\hat M\|_{L^2} \|v-\hat v\|_{L^2}
	\le C_{Poin} \ell \|\dx M-\dx \hat M\|_{L^2} \|v-\hat v\|_{L^2}\\
	&\le \tfrac{1}{2} C_{Poin}\ell \Lnorm{\rho-\hat \rho} + \tfrac{1}{2} C_{Poin}\ell \Lnorm{v-\hat v},
	\end{align*}
	where we have used the fact that $\dx M(x) = \int_0^t \dx m(x, t') dt' - \rho_0(x) =- \rho(x)$ and Young's inequality in the last step.
	If $\delta \le \tfrac{c_0}{C_{Poin} \ell}$, we have
	\begin{align*}
	\delta |\F_1(\hat u|u)| \le \frac{c_0}{2} \Lnorm{u-\hat u}.
	\end{align*}
	Applying Lemma \ref{lem:normequ} yields the assertion.
\end{proof}

\begin{remark}
	Lemma \ref{lem:normequ_v} can be generalized to functions 
	$u_1=(\rho_1, v_1)$, $u_2=(\rho_2, v_2)\in L^2(0,\ell)$ with
	\begin{align*} 
	0 < \underline{\rho}\le \rho_1 (x),  \rho_2 (x) \le \bar \rho, \qquad 
	-\tilde v \le v_1(x), v_2(x) \le \tilde v,\qquad \forall x\in(0,\ell),
	\end{align*}
	and $M_1(\ell,t)=M_2(\ell,t)$ instead of $M_1(0,t)=M_2(0,t)$. In this case the definition of $M$ has to be modified, i.e., the integral $\int_0^x \rho_0(x') dx'$ has to be replaced by $\int_\ell^x \rho_0(x') dx'$ and boundary conditions for $m$ have to be prescribed at the right end of the pipe.
\end{remark}

Thus, we have shown that condition \eqref{eq:normeq-HG} is satisfied. It remains to
verify condition \eqref{eq:estimatecH}.

\begin{lem}\label{lem:decrease_v}
	Consider a solution $u=(\rho, v)$  of \eqref{eq:system_1}-\eqref{eq:system_2} and a solution $\hat u=(\hat \rho, \hat v)$ of \eqref{eq:observer_1}-\eqref{eq:observer_2} with observer terms given by \eqref{eq:observer-term_v} satisfying the boundary conditions \eqref{eq:BC}, the properties stated in (A1)-(A3) and the bounds in (A4) with $C_t, \bar v>0$ sufficiently small.
	
	Then there exist constants $\delta,  \, \bar c>0 $ with $\delta \le \tfrac{c_0}{C_{Poin} \ell}$  such that
	\begin{equation*} 
	\begin{split}
	-\langle \mathcal{L}_\rho, h-\hat h\rangle  -\langle \mathcal{L}_v, m-\hat m\rangle   + \delta \partial_t \F_1(\hat u|u)   
	\leq - \bar c \left(\H(\hat u|u)+ \delta \F_1(\hat u|u)\right).
	\end{split}
	\end{equation*}
\end{lem}

\begin{proof}
	We start by computing the time derivative of $\F_1$, which is given by
	\begin{align*}
		\dt \F_1(\hat u|u) 
		=&\, \int_0^\ell (\dt M-\dt \hat  M) (v-\hat v) dx +\int_0^\ell (M-\hat M) (\dt v-\dt \hat v) dx\\
		=&\, \int_0^\ell (m-\hat m)(v-\hat v) dx + \int_0^\ell (M-\hat M)(\dx \hat h-\dx h) dx\\
		& + \int_0^\ell (M-\hat M)(-\mathcal{L}_v) dx + \int_0^\ell (M-\hat M)(-\gamma |v| v+\gamma |\hat v | \hat v) dx\\
		= &\, (E_1) + (E_2) + (E_3) +(E_4),
	\end{align*}
	where the first term can be estimated by
	\begin{align*}
		(E_1) &= \int_0^\ell (m-\hat m)(v-\hat v) dx 
		= \int_0^\ell\left( \rho (v-\hat v)^2 + \hat v (\rho-\hat \rho ) (v-\hat v) \right)dx\\
		&\le 
		\tfrac{1}{2} \tfrac{\bar v^2}{\bar \rho} \Lnorm{\rho-\hat \rho} + \tfrac{3}{2} \bar \rho \Lnorm{v-\hat v}.
	\end{align*}
	Using integration by parts and the boundary conditions \eqref{eq:BC} we can estimate
		\begin{align*}
		(E_2)
		& = \int_0^\ell  (\dx M- \dx \hat M)( h-\hat h) dx 
		 = \int_0^\ell  (\hat \rho-\rho)\left(\tfrac{1}{2} (v^2-\hat v^2)+P'(\rho)-P'(\hat \rho)\right) dx\\
		&\le - \ubar C_{P''} \Lnorm{\rho-\hat{\rho}}
			+\int_0^\ell \frac{1}{2} (v+\hat v)(v-\hat v) (\hat\rho-\rho) dx\\
		&\le - \left( \ubar C_{P''} - \tfrac{\bar v^2  }{2 \ubar \rho}  \right) \Lnorm{\rho-\hat{\rho}}
			+ \tfrac{1}{2} \ubar \rho \Lnorm{v-\hat v}.
	\end{align*}
	Applying the Poincar\'{e} inequality to $(E_3)$ yields
	\begin{align*}
		(E_3)
		\le C_{Poin} \ell \|\rho-\hat\rho\|_{L^2} \,\mu  \|v-\hat v\|_{L^2}
		\le \tfrac{1}{4} \ubar C_{P''} \Lnorm{\rho-\hat{\rho}} 
		+ C_{Poin}^2 \ell^2 \mu^2 \tfrac{1}{\ubar C_{P''}} \Lnorm{v-\hat v}.
	\end{align*}
	Finally, the last term can be bounded by
	\begin{align*}
		(E_4)&=\int_0^\ell (M-\hat M)(-\gamma |v| v+\gamma |\hat v | \hat v) dx
		\le C_{Poin} \ell \|\rho-\hat\rho\|_{L^2} \, 2 \gamma \bar v \|v-\hat v\|_{L^2}\\
		&\le \tfrac{1}{4} \ubar C_{P''} \Lnorm{\rho-\hat{\rho}} 
			+ 4 \gamma^2 C_{Poin}^2 \ell^2 \frac{\bar v^2}{\Cp} \Lnorm{v-\hat v}.
	\end{align*}
	Assume that $\bar v^2\le \tfrac{1}{4}\ubar \rho \Cp$. Then summing the contributions from $(E_1)$--$(E_4)$ yields
	\begin{align*}
		\dt\F_1(\hat u|u) \le -\frac{1}{4} \ubar C_{P''} \Lnorm{\rho-\hat\rho}
			+(2 \bar \rho +C_{Poin}^2 \ell^2 \mu^2 \frac{1}{\Cp} +C_{Poin}^2 \ell^2 \gamma^2  \ubar \rho) \Lnorm{v-\hat v}.
	\end{align*}
	Combining this with \eqref{eq:estimate_Rr_Rv_v} leads to
	\begin{equation*} 
	\begin{split}
	    &-\langle \mathcal{L}_\rho, h-\hat h\rangle  -\langle \mathcal{L}_v, m-\hat m\rangle   + \delta \partial_t \F_1(\hat u|u) \\
		& \le \left(-\frac{1}{4} \delta  \ubar C_{P''}
		+ \tfrac{1}{2} \mu \frac{\bar v^2}{\ubar\rho}\right) \Lnorm{\rho-\hat\rho}\\
		&\qquad +\left(-\tfrac{1}{2} \mu \ubar \rho
		+ \delta \left( 2 \bar \rho +C_{Poin}^2 \ell^2 \mu^2 \frac{1}{\Cp} +C_{Poin}^2 \ell^2  \gamma^2 \ubar \rho \right)\right)	\Lnorm{v-\hat v}.
	\end{split}	
	\end{equation*}
	Now, we choose $\delta$ sufficiently small such that
	\begin{align} \label{eq:bound_delta}
		\delta \big( 2 \bar \rho +C_{Poin}^2 \ell^2 \mu^2 \frac{1}{\Cp} +C_{Poin}^2 \ell^2  \gamma^2 \ubar \rho \big)\le \tfrac{1}{4} \mu \ubar \rho
	\end{align}
	and the bound $\delta \le \tfrac{c_0}{C_{Poin} \ell}$ from Lemma \ref{lem:normequ_v} is satisfied.
	Then, for this $\delta$ we choose the velocity bound $\bar v$ sufficiently small such that
	\begin{align*}
		\tfrac{1}{2} \mu \frac{\bar v^2}{\ubar\rho} \le \frac{1}{8} \delta  \ubar C_{P''}.
	\end{align*}
	This yields
	\begin{align*}
		-\langle \mathcal{L}_\rho, h-\hat h\rangle  -\langle \mathcal{L}_v, m-\hat m\rangle   + \delta \partial_t \F_1(\hat u|u)
		\le -\frac{1}{8} \delta  \ubar C_{P''} \Lnorm{\rho-\hat\rho}
		-\tfrac{1}{4} \mu \ubar \rho	\Lnorm{v-\hat v},	
	\end{align*}
	which together with Lemma \ref{lem:normequ_v} implies the statement of the lemma.
\end{proof}

Applying Lemma ~\ref{lem:dtH_2} finishes the proof of Theorem \ref{thm:ExpSynchronization} for the case of velocity measurement.

\begin{remark}
	Increasing the parameter $\mu$ in the observer term does not necessarily improve the speed of convergence of the state of the observer system towards the original system state:
	In the proof of Lemma \eqref{lem:decrease_v} we see that $-\langle \mathcal{L}_\rho, h-\hat h\rangle  -\langle \mathcal{L}_v, m-\hat m\rangle   + \delta \partial_t \F_1(\hat u|u)$ has the decrease rate $\min\{\frac{1}{8} \delta  \ubar C_{P''} , \tfrac{1}{4} \mu \ubar \rho  \}$ and due to the bound \eqref{eq:bound_delta} on $\delta$  a large $\mu$ enforces a small $\delta$.
\end{remark}

\begin{remark}
	In the proof of Lemma \ref{lem:decrease_v} we see that we need at one end of the pipe boundary conditions for $m$ and at the other end boundary conditions for $h$, i.e., \eqref{eq:BC} or
	\begin{align*}
	m(t,\ell)=\hat m(t,\ell)=\tilde m_b(t), \qquad
	h(t,0)=\hat h(t,0)=\tilde h_b(t), \qquad 0\le t\le T.
	\end{align*}
	The reason for this is that
	we want to use the Poincar\'{e} inequality for $M$ in the estimation of $(E_3)$, so that we need $M=\hat M$ at at least one point in the interval $(0,\ell)$. Due to the definition of $M$ in Definition \ref{def:G}, this is satisfied at $x=0$, if we choose boundary conditions for $m$ at the left end of the pipe. Then, if we want to have vanishing boundary contributions in the partial integration in the estimation of $(E_2)$, we need $(M-\hat M)(\hat h -h)(x=\ell) =0$. Since in general
	\begin{align*}
		M(\ell,t)=\int_0^t m(\ell, t') dt' - \int_0^\ell \rho_0(x') dx'
		\ne \hat M(\ell,t),
	\end{align*}
	we need boundary conditions for $h$ at the other end of the pipe.
\end{remark}

\subsection{Idea of proof for measurement of $\rho$ or $m$} \label{sec:meas_rho_m}
In the case of  measurements of $\rho$ the nudging terms are given by
\begin{align*}
	\mathcal{L}_\rho=\mu\frac{c}{\sqrt{p'(\hat{\rho})}}\hat\rho (\tilde P(\rho)-\tilde P(\hat\rho)), \qquad \mathcal{L}_v=0, \qquad \mu>0.
\end{align*}
The proof of Theorem \ref{thm:ExpSynchronization}  for measurements of $\rho$ is similar to the proof for measurements of $v$.
Again, we start by estimating the last two terms of \eqref{eq:dtH}, i.e., 
\begin{align*} 
	&-\langle \mathcal{L}_\rho, h-\hat h\rangle -\langle \mathcal{L}_v, m-\hat m\rangle\notag\\
	&\le -\frac{1}{2}\mu \frac{\ubar{\rho}}{\bar \rho}\Cp  \frac{\sqrt{\ubar C_{p'}}}{\sqrt{\bar C_{p'}}} \Lnorm{\rho-\hat \rho} 
	+\frac{1}{2} \mu \left(\frac{\bar{\rho}}{\ubar{\rho}} \right)^3\frac{\bar v^2}{\Cp} \left( \frac{\bar C_{p'}}{\ubar C_{p'}}\right)^{3/2} \Lnorm{v-\hat v}, 
\end{align*}
where we have used the bounds in assumptions (A1)--(A3) as well as the assumption $|v|, |\hat v| \le \bar v$ for some constant $\bar v>0$.
In order to ensure decrease also with respect to $\|v-\hat v\|_{L^2}^2$, we define the functional
\begin{align*}
\F_2(\hat u|u):=\int_0^\ell (N-\hat N) (\rho- \hat \rho) dx
\end{align*}
with
\begin{align*}
	N(x,t):= & \int_0^t h(x, t') dt' -\int_\ell^x v_0(x') dx' 
	+ \int_\ell^x \int_0^t \gamma |v(x', t')| v(x', t') dt' dx'.
\end{align*}	
Using that 	$(N-\hat N)(\ell,t)=0$ and the Poincar\'{e} inequality, we can show that
\begin{align} \label{eq:equivalence_rho}
	\tfrac{c_0}{2} \| u(t,\cdot)-\hat u(t,\cdot)\|_{L^2(0,\ell)}^2 \le \H(\hat u| u)(t) +\delta \F_2(\hat u|u)(t) \le \tfrac{3}{2} C_0\| u(t,\cdot)-\hat u(t,\cdot)\|_{L^2(0,\ell)}^2
\end{align}
for all $0\le t\le T$ for $\delta$ sufficiently small.
Finally, it can be shown that, if $|v|, |\hat v| \le \bar v$ with $\bar v$ sufficiently small, then 
there exist constants $\delta,  \,\bar c>0 $ such that \eqref{eq:equivalence_rho} is satisfied and
\begin{align*}
	-\langle \mathcal{L}_\rho, h-\hat h\rangle -\langle \mathcal{L}_v, m-\hat m\rangle +\delta \dt \F_2(\hat u|u)
	\le -\bar c (\H(\hat u|u)+\delta \F_2(\hat u|u)),
\end{align*}
i.e., condition \eqref{eq:estimatecH} is satisfied for $\G:=\delta \F_2$. This finishes the proof for measurement of $v$.

For the case of measurement of the mass flow $m$ we use the nudging terms
\begin{align*}
\mathcal{L}_\rho=0, \qquad \mathcal{L}_v=\mu (m-\hat m)
\end{align*}
with $\mu>0$. This implies 
\begin{align*}
	-\langle \mathcal{L}_\rho, h-\hat h\rangle  -\langle \mathcal{L}_v, m-\hat m\rangle  
	\le -\tfrac{1}{2}\mu\ubar\rho^2 \Lnorm{v-\hat v} + 2\mu \bar v^2 \tfrac{\bar\rho^2}{\ubar \rho^2} \Lnorm{\rho -\hat\rho}.
\end{align*}
Again, the observer terms give decrease only with respect to one state variable.
We use the same functional as for the velocity measurement, i.e. the functional $\F_1$ defined in Definition \ref{def:G}, in order to achieve decrease also with respect to the other state variable.
Since the equivalence to the $L^2$-norm was already shown in Lemma \ref{lem:normequ_v},
it remains to show that the decrease condition \eqref{eq:estimatecH} is satisfied.
This can be shown similar to the proof of Lemma \ref{lem:decrease_v}.

\subsection{Extension to star-shaped networks} \label{sec:extensionNetworks}
Now we extend the convergence results that are shown in Theorem \ref{thm:ExpSynchronization} for a single pipe to star-shaped networks. 
Such networks have only one inner node, which we denote by $\nu_0$. We assume that the edges are oriented such that they start in the central node $\nu_0$ and end at a boundary node.
For the case of measurements of $v$ or $m$, we prescribe at every boundary node the boundary conditions
\begin{align} \label{eq:bc_net_mv}
	m^e(t,\nu)= \hat m^e(t,\nu)=m_b^e(t,\nu)\qquad \forall \nu\in \V_\partial,
\end{align}
while for density measurements we prescribe the boundary conditions
\begin{align} \label{eq:bc_net_rho_vp}
	h^e(t, \nu_\partial)=\hat h^e(t, \nu_\partial)=h_b^e(t, \nu_\partial)
\end{align}
for one arbitrary, but fixed boundary node $\nu_\partial\in\V_\partial$ and
\begin{align} \label{eq:bc_net_rho}
	m^e(t,\nu)=\hat m^e(t,\nu)=m_b^e(t,\nu)\qquad \forall \nu\in \V_\partial\setminus\{\nu_\partial\}.
\end{align}
Together with the energy conserving coupling conditions \eqref{eq:couplingc} 
this allows us to extend the results obtained above for a single pipe to star-shaped networks, up to some limitations detailed below.
The main idea in the proof is summing the estimates over all edges,
where we have to pay attention to the application of the Poincar\'{e} inequality in the proof of Lemma~\ref{lem:normequ_v} and of equation \eqref{eq:equivalence_rho}
as well as to the partial integration used in the verification of condition \eqref{eq:estimatecH}. For measurements of $v$ we present the proof in detail, while for measurements of $m$ or $\rho$ we just present the idea of the proof.

We start with the result for measurements of $v$ or $m$. In this case, we get exponential synchronization up to some error that is proportional to the difference between the total mass in the original system and the total mass in the observer system at the initial time $t=0$.

\begin{thm}\label{thm:ExpSynchronization_net_m-v}
	Let $u^e=(\rho^e, v^e)$ and $\hat u^e=(\hat\rho^e, \hat v^e)$, $e\in\E$, be Lipschitz continuous solutions of \eqref{eq:system_1}-\eqref{eq:system_2} and
    \eqref{eq:observer_1}-\eqref{eq:observer_2}, respectively, on the network $(\V, \E)$ with boundary conditions \eqref{eq:bc_net_mv}
    and source terms $\mathcal{L}_\rho$, $\mathcal{L}_v$ given by \eqref{eq:observer-term_v} or \eqref{eq:observer-term_m},
    that satisfy the coupling conditions~\eqref{eq:couplingc},
    the properties stated in the
    assumptions (A1)--(A3) and the bounds in (A4) with $C_t, \bar v>0$ sufficiently small on all pipes $e\in\E$.

	Then there exist constants $C_1, C_2, C_3>0$ such that
	\begin{align*}
    	\| u(t,\cdot)-\hat u(t,\cdot)\|_{L^2(\E)}^2 \le& \ C_1 
    	\| u_0-\hat u_0\|_{L^2(\E)}^2 \exp(-C_2t) + C_3 |\Delta M_0|
	\end{align*}
	for all $0\le t\le T$ with mass difference at time $t=0$
	\begin{align*}
	    \Delta M_0:=  \sum_{e\in\E} \int_0^{\ell^e} \rho_0^e(x) dx - \sum_{e\in\E} \int_0^{\ell^e} \hat \rho_0^e(x) dx  .
	\end{align*}
\end{thm}

\begin{remark}
	In contrast to the result on a single pipe, in Theorem \ref{thm:ExpSynchronization_net_m-v} there is the additional term $C_3 |\Delta M_0|$ depending on the initial total mass difference.
	Due to the choice of the boundary conditions for $m$ at each boundary node and the choice  $\mathcal{L}_\rho=0$ for measurement of $v$ or $m$, there is no chance that this contribution vanishes: For the mass difference at time $t$
	\begin{align*}
		\Delta M(t):=  \sum_{e\in\E} \int_0^{\ell^e} \rho^e(x,t) dx - \sum_{e\in\E} \int_0^{\ell^e} \hat \rho^e(x,t) dx 
	\end{align*}
	we have
	\begin{align*}
		\dt (\Delta M(t))&=  \sum_{e\in\E} \int_0^{\ell^e} (-\dx m^e(x,t)+\dx \hat m^e(x,t)) dx
		= 	\sum_{e\in\E}( m^e(\nu_0) -  \hat m^e(\nu_0) )= 0,
	\end{align*}
	where we have used the boundary conditions \eqref{eq:bc_net_mv}
	and the coupling conditions \eqref{eq:couplingc}. 
	This implies that $\Delta M(t) = \Delta M_0$ for all $t\in[0,T]$, i.e. the total mass difference at time $t$ is the same as the total mass difference at the initial time $t=0$.
\end{remark}

\begin{proof}[Proof of Theorem \ref{thm:ExpSynchronization_net_m-v}]
	By summing over all edges $e\in\E$, the proof in Section 
	\ref{sec:meas_v} for measurement of $v$ transfers almost verbatim to star-shaped networks except for the application of the Poincar\'{e} inequality in the proof of Lemma~\ref{lem:normequ_v} and the partial integration used in the proof of Lemma~\ref{lem:decrease_v}, which we will detail in the following.
    In order to transfer Lemma~\ref{lem:normequ_v} to networks, we define analogously to Definition~\ref{def:G}
    \begin{align*}
		\F_1(\hat u|u)(t):=\sum_{e\in\E}\int_0^{\ell^e} (M^e(x,t)-\hat M^e(x,t)) (v^e(x,t)-\hat v^e(x,t)) dx
	\end{align*}
	with
	\begin{align*}
		M^e(x,t):=\int_0^t m^e(x, t') dt' - \int_{\ell^e}^x \rho^e_0(x') dx', \quad e\in\E.
	\end{align*}
	Due to the boundary conditions \eqref{eq:bc_net_mv} we have
	\begin{align*}
	    M^e(\nu) = \int_0^t m^e(\nu,t') dt' =  \hat M^e(\nu) \qquad \forall \nu\in\V_\partial, e\in\E(\nu).
	\end{align*}
	Since the network is star-shaped, every edge $e\in\E$ is incident to some boundary node. Therefore, we can apply the Poincar\'{e} inequality to each edge separately and show similarly to Lemma~\ref{lem:normequ_v} that
	\begin{align*}
		&\tfrac{c_0}{2} \| \hat u(t, \cdot)- u(t, \cdot)\|_{L^2(\E)}^2 \le \sum_{e\in\E}\H( \hat u^e| u^e)(t) +\delta \F_1( \hat u| u)(t) 
		\le \tfrac{3}{2} C_0\|\hat u(t, \cdot)- u(t, \cdot)\|_{L^2(\E)}^2\qquad 
	\end{align*}
	for all $0\le t\le T$, if
	$\delta \le \frac{c_0}{ C_{Poin} \ell_{\max}}$
	with $\ell_{\max}:=\max_{e\in\E}\ell^e$.

	Next, we consider the partial integration that is used in the estimation of term $(E_2)$ in the proof of Lemma~\ref{lem:decrease_v}. For star-shaped networks, we get
	\begin{align*}
		(E_2)& = \sum_{e\in\E}\int_0^{\ell^e} (M^e-\hat M^e)(\dx \hat h^e-\dx h^e) dx\\
		&= \sum_{e\in\E} \int_0^{\ell^e}  (\dx M^e- \dx \hat M^e)( h^e-\hat h^e) dx
		    + \sum_{e\in\E} (M^e-\hat M^e) (\hat h^e- h^e)\vert_0^{\ell^e}\\
		&=  \sum_{e\in\E} \int_0^{\ell^e}  (\dx M^e- \dx \hat M^e)( h^e-\hat h^e) dx
		   + \sum_{\nu\in\V_\partial} \sum_{e\in\E(\nu)} (M^e(\nu)-\hat M^e(\nu)) (\hat h^e(\nu)- h^e(\nu)) \\
		&\quad     -(\hat h(\nu_0)- h(\nu_0)) \sum_{e\in\E(\nu_0)}  (M^e(\nu_0)-\hat M^e(\nu_0)),
	\end{align*}
	where we have used the coupling conditions \eqref{eq:couplingc} in the last step.
	The sum over the boundary nodes vanishes due to the boundary condition \eqref{eq:bc_net_mv}
	and for the last term we can compute
	\begin{align*}
		&\sum_{e\in\E(\nu_0)}  (M^e(\nu_0)-\hat M^e(\nu_0))\\
		&= \int_0^t \left( \sum_{e\in\E(\nu_0)} m^e(\nu_0,t') - \sum_{e\in\E(\nu_0)} \hat m^e(\nu_0,t')  \right) dt'
		+ \sum_{e\in\E} \int_0^{\ell^e} \rho_0^e(x') dx'  - \sum_{e\in\E} \int_0^{\ell^e} \hat\rho_0^e(x') dx',
	\end{align*}
	where the first term vanishes due to the coupling conditions \eqref{eq:couplingc}.
	Using the estimation of $(E_2)$ in the proof of Lemma \ref{lem:decrease_v} and the fact that by Assumption (A2) the term $\hat h(\nu_0)- h(\nu_0)$
	is bounded by a constant $C_h$ depending only on the bounds in the assumptions,
	this implies
	\begin{align*}
		(E_2)\le - \left( \ubar C_{P''} - \tfrac{\bar v^2  }{2 \ubar \rho}  \right) \|\rho-\hat{\rho}\|_{L^2(\E)}^2
		+ \tfrac{1}{2} \ubar \rho \|v-\hat v\|_{L^2(\E)}^2 +C_h |\Delta M_0|.
	\end{align*}
	For measurements of $v$, we can proceed as in the proof of Lemma \ref{lem:decrease_v}. This yields
    \begin{align*}
		&-\langle \mathcal{L}_\rho, h-\hat h\rangle  -\langle \mathcal{L}_v, m-\hat m\rangle   + \delta \partial_t \F_1(\hat u|u)\\
		&\le -\frac{1}{8} \delta  \ubar C_{P''} \|\rho-\hat\rho\|_{L^2(\E)}^2 
		-\tfrac{1}{4} \mu \ubar \rho	\|v-\hat v\|_{L^2(\E)}^2  +\delta C_h |\Delta M_0|.	
	\end{align*}
	As in the proof of Lemma \ref{lem:dtH_2}, we can then apply a Gronwall Lemma, which yields the assertion.	The statement for measurement of $m$ follows similarly.
\end{proof}

Now, we want to extend the exponential synchronization for measurement of $\rho$ to star-shaped networks. 
Here, we do not get an additional term depending on the initial mass difference, but we will use a slightly adapted definition of the functional $\F_2$, which makes use of the continuity of the specific enthalpy $h$ included into the coupling conditions.

\begin{thm}\label{thm:ExpSynchronization_net_rho}
	Let $u^e=(\rho^e, v^e)$ and $\hat u^e=(\hat\rho^e, \hat v^e)$, $e\in\E$, be Lipschitz continuous solutions of \eqref{eq:system_1}-\eqref{eq:system_2} and
    \eqref{eq:observer_1}-\eqref{eq:observer_2}, respectively, on the network $(\V, \E)$ with boundary conditions \eqref{eq:bc_net_rho_vp}-\eqref{eq:bc_net_rho}
    and source terms $\mathcal{L}_\rho$, $\mathcal{L}_v$ given by \eqref{eq:observer-term_rho},
   	that satisfy the coupling conditions~\eqref{eq:couplingc},
   	the properties stated in the
   	assumptions (A1)--(A3) and the bounds in (A4) with $C_t, \bar v>0$ sufficiently small on all pipes $e\in\E$.

	Then there exist constants $C_1, C_2 >0$ such that
	\begin{align} \label{eq:decrease_rho_network}
    	\| u^e(t,\cdot)-\hat u^e(t,\cdot)\|_{L^2(\E)}^2 \le& \ C_1 
    	\| u^e_0-\hat u^e_0\|_{L^2(\E)}^2 \exp(-C_2t) 
	\end{align}
	for all $0\le t\le T$.
\end{thm}

\begin{proof}[Idea of proof]
	For measurements of $\rho$, we define similar to the definition in Section \ref{sec:meas_rho_m}
	\begin{align*}
	\F_2(\hat u|u):=\sum_{e\in\E}\int_0^{\ell^e} (N^e-\hat N^e) (\rho^e- \hat \rho^e) dx
	\end{align*}
	with
	\begin{align*}
	N^e(x,t):= & \int_0^t h^e(x, t') dt' -\int_{\vp}^x v_0^e(x') dx' 
	+ \int_{\vp}^x \int_0^t \gamma |v^e(x', t')| v^e(x', t') dt' dx',
	\end{align*}
	where we integrate along the unique path from $\vp$ to $x\in e$.
	It can be shown that 
	\begin{align*}
		\tfrac{c_0}{2} \| u(t,\cdot)-\hat u(t,\cdot)\|_{L^2(\E)}^2 \le \sum_{e\in\E} \H(\hat u^e| u^e)(t) +\delta \F_2(\hat u|u)(t) \le \tfrac{3}{2} C_0\| u(t,\cdot)-\hat u(t,\cdot)\|_{L^2(\E)}^2
	\end{align*}
	for all  $0\le t\le T$ for $\delta$ sufficiently small.
	Finally, using the coupling conditions \eqref{eq:couplingc} we can show that all boundary terms at inner nodes of the network vanish and \eqref{eq:estimatecH} is satisfied for $\G=\delta \F_2$ for suitable constants $\delta,\, \bar v,\, \bar c>0$.
\end{proof}

\begin{remark}
	For measurements of $v$ or $m$ we use the boundary conditions \eqref{eq:bc_net_mv}, i.e., we prescribe the value of the mass flow at every boundary node.
	Due to this choice of the boundary conditions it is possible to apply the Poincar\'{e} inequality in order to show equivalence of $\H+\delta \F_1$ to the $L^2$-norm (see the proof of Theorem \ref{thm:ExpSynchronization_net_m-v}),
	but this boundary conditions also lead  to the mass difference $\Delta M_0$ in Theorem \eqref{thm:ExpSynchronization_net_m-v}.
	In contrast to this, for measurements of the density we use the boundary conditions \eqref{eq:bc_net_rho_vp}-\eqref{eq:bc_net_rho}, i.e. at one fixed boundary node we use boundary conditions for $h$, while at all other boundary nodes we prescribe the value of $m$.
	This can be done since the continuity of $h$ is included in the coupling conditions such that the Poincar\'{e} inequality, which is needed in order to show the norm equivalence, can be applied.
	In combination with the coupling conditions this choice of boundary conditions has the advantage that we do not have a term of the form $\Delta M_0$ at the right hand side of \eqref{eq:decrease_rho_network}.
\end{remark}

\section*{Acknowledgement}
The authors are grateful for financial support by the German Science Foundation (DFG) via grant TRR~154 (project number 239904186), \emph{Mathematical modelling, simulation and optimization using the example of gas networks}, project~C05.


\begin{thebibliography}{10}
	
	\bibitem{Aamo06}
	O.~M. Aamo, J.~Salvesen, and B.~A. Foss.
	\newblock Observer design using boundary injections for pipeline monitoring and
	leak detection.
	\newblock {\em IFAC Proceedings Volumes}, 39(2):53--58, 2006.
	\newblock 6th IFAC Symposium on Advanced Control of Chemical Processes.
	
	\bibitem{AschBocquetNodet_DataAssimilation_2016}
	M.~Asch, M.~Bocquet, and M.~Nodet.
	\newblock {\em Data Assimilation: Methods, Algorithms, and Applications}.
	\newblock Society for Industrial and Applied Mathematics, Philadelphia, PA,
	2016.
	
	\bibitem{BENABDELHADI_Krstic_2021}
	A.~Benabdelhadi, F.~Giri, T.~Ahmed-Ali, M.~Krstic, H.~{El Fadil}, and F.-Z.
	Chaoui.
	\newblock Adaptive observer design for wave pdes with nonlinear dynamics and
	parameter uncertainty.
	\newblock {\em Automatica}, 123:109295, 2021.
	
	\bibitem{BoulangerMoireau2015}
	A.-C. Boulanger, P.~Moireau, B.~Perthame, and J.~Sainte-Marie.
	\newblock Data assimilation for hyperbolic conservation laws: a {L}uenberger
	observer approach based on a kinetic description.
	\newblock {\em Commun. Math. Sci.}, 13(3):587--622, 2015.
	
	\bibitem{bressan2005hyperbolic}
	A.~Bressan.
	\newblock {\em Hyperbolic systems of conservation laws: the one dimensional
		Cauchy problem}.
	\newblock Oxford Univ. Press, New York, 2005.
	
	\bibitem{BressanHerty2014}
	A.~Bressan, S.~\v{C}ani\'{c}, M.~Garavello, M.~Herty, and B.~Piccoli.
	\newblock Flows on networks: recent results and perspectives.
	\newblock {\em EMS Surv. Math. Sci.}, 1(1):47--111, 2014.
	
	\bibitem{BrouwerGasserHerty2011}
	J.~Brouwer, I.~Gasser, and M.~Herty.
	\newblock Gas pipeline models revisited: model hierarchies, nonisothermal
	models, and simulations of networks.
	\newblock {\em Multiscale Model. Simul.}, 9(2):601--623, 2011.
	
	\bibitem{ChapelleCindea2012}
	D.~Chapelle, N.~C\^{\i}ndea, M.~De~Buhan, and P.~Moireau.
	\newblock Exponential convergence of an observer based on partial field
	measurements for the wave equation.
	\newblock {\em Math. Probl. Eng.}, 2012.
	
	\bibitem{ChapelleCindeaMoireau2012}
	D.~Chapelle, N.~C\^{\i}ndea, and P.~Moireau.
	\newblock Improving convergence in numerical analysis using observers --- the
	wave-like equation case.
	\newblock {\em Math. Models Methods Appl. Sci.}, 22(12):1250040, 35, 2012.
	
	\bibitem{CindeaImperialeMoireau2015}
	N.~C\^{\i}ndea, A.~Imperiale, and P.~Moireau.
	\newblock Data assimilation of time under-sampled measurements using observers,
	the wave-like equation example.
	\newblock {\em ESAIM Control Optim. Calc. Var.}, 21(3):635--669, 2015.


	\bibitem{Coron2007}
	J.-M. Coron, B.~d'Andr\'{e}a{-}Novel, and G.~Bastin.
	\newblock A strict {L}yapunov function for boundary control of hyperbolic
	systems of conservation laws.
	\newblock {\em IEEE Trans. Automat. Control}, 52(1):2--11, 2007.
	
	\bibitem{Dafermos1979}
	C.~M. Dafermos.
	\newblock The second law of thermodynamics and stability.
	\newblock {\em Arch. Rational Mech. Anal.}, 70(2):167--179, 1979.
	
	\bibitem{dafermos2016hyperbolic}
	C.~M. Dafermos.
	\newblock {\em Hyperbolic Conservation Laws in Continuum Physics}.
	\newblock Springer, Berlin, Heidelberg, 2016.
	
	\bibitem{EggerGiesselmann}
	H.~Egger and J.~Giesselmann.
	\newblock Stability and asymptotic analysis for instationary gas transport via
	relative energy estimates.
	\newblock arXiv:2012.14135, 2020.
	
	\bibitem{Egger2023_RelativeEnergysiDiscrete}
	H.~Egger, J.~Giesselmann, T.~Kunkel, and N.~Philippi.
	\newblock An asymptotic-preserving discretization scheme for gas transport in
	pipe networks.
	\newblock {\em IMA J. Numer. Anal.}, 43(4):2137--2168, 2023.
	
	\bibitem{EggerKugler2018}
	H.~Egger and T.~Kugler.
	\newblock Damped wave systems on networks: exponential stability and uniform
	approximations.
	\newblock {\em Numer. Math.}, 138(4):839--867, 2018.
	
	\bibitem{FarhatJohnstonJollyTiti2018}
	A.~Farhat, H.~Johnston, M.~Jolly, and E.~S. Titi.
	\newblock Assimilation of nearly turbulent {R}ayleigh-{B}\'{e}nard flow through
	vorticity or local circulation measurements: a computational study.
	\newblock {\em J. Sci. Comput.}, 77(3):1519--1533, 2018.
	
	\bibitem{Titi2020}
	B.~Garc\'{\i}a-Archilla, J.~Novo, and E.~S. Titi.
	\newblock Uniform in time error estimates for a finite element method applied
	to a downscaling data assimilation algorithm for the {N}avier-{S}tokes
	equations.
	\newblock {\em SIAM J. Numer. Anal.}, 58(1):410--429, 2020.
	
	\bibitem{GiesselmannLattanzioTzavaras2017}
	J.~Giesselmann, C.~Lattanzio, and A.~E. Tzavaras.
	\newblock Relative energy for the {K}orteweg theory and related {H}amiltonian
	flows in gas dynamics.
	\newblock {\em Arch. Ration. Mech. Anal.}, 223(3):1427--1484, 2017.
	
	\bibitem{GoatinLaurent-Brouty2019}
	P.~Goatin, and N.~Laurent-Brouty.
	\newblock The zero relaxation limit for the {A}w-{R}ascle-{Z}hang
	traffic flow model.
	\newblock {\em Z.~Angew. Math. Phys.}, 70(1), 2019.
		
	\bibitem{Gugat2021}
	M.~Gugat, J.~Giesselmann, and T.~Kunkel.
	\newblock Exponential synchronization of a nodal observer for a semilinear
	model for the flow in gas networks.
	\newblock {\em IMA J. Math. Control Inform.}, 38(4):1109--1147, 2021.
	
	\bibitem{GugatUlbrich2018}
	M.~Gugat and S.~Ulbrich.
	\newblock Lipschitz solutions of initial boundary value problems for balance
	laws.
	\newblock {\em Math. Models Methods Appl. Sci.}, 28(5):921--951, 2018.
	
	\bibitem{GugatHertyYu2018}
	M.~Gugat, M.~Herty, and H.~Yu.
	\newblock On the relaxation approximation for {$2\times2$} hyperbolic
	balance laws.
	\newblock In {\em Theory, numerics and applications of hyperbolic problems.
		{I}}, {\em Springer Proc. Math. Stat.}, 236:651--663, 2018.

	\bibitem{Hayat2021}
	A.~Hayat.
	\newblock Boundary stabilization of 1{D} hyperbolic systems.
	\newblock {\em Annu. Rev. Control}, 52:222--242, 2021.

	\bibitem{JinXin1995}
	S.~Jin, and Z.~Xin.
	\newblock The relaxation schemes for systems of conservation laws in arbitrary
	space dimensions.
	\newblock {\em Comm. Pure Appl. Math.}, 48(3):235--276, 1995.

	
	\bibitem{JollyTiti2019}
	M.~S. Jolly, V.~R. Martinez, E.~J. Olson, and E.~S. Titi.
	\newblock Continuous data assimilation with blurred-in-time measurements of the
	surface quasi-geostrophic equation.
	\newblock {\em Chinese Ann. Math. Ser. B}, 40(5):721--764, 2019.
	
	\bibitem{LawStuartZygalakis_Dataassimilation2015}
	K.~Law, A.~Stuart, and K.~Zygalakis.
	\newblock {\em Data Assimilation: A Mathematical Introduction}.
	\newblock Springer International Publishing, Cham, 2015.
	
	\bibitem{LiLu2022}
	T.~Li and X.~Lu.
	\newblock Exact boundary synchronization for a kind of first order hyperbolic
	system.
	\newblock {\em ESAIM: COCV}, 28:34, 2022.
	
	\bibitem{LiRao2003}
	T.-T. Li and B.-P. Rao.
	\newblock {Exact Boundary Controllability for Quasi-Linear Hyperbolic Systems}.
	\newblock {\em SIAM J. Control Optim.}, 41(6):1748--1755, 2003.
	
	\bibitem{Luenberger1971}
	D.~Luenberger.
	\newblock An introduction to observers.
	\newblock {\em IEEE Transactions on Automatic Control}, 16(6):596--602, 1971.
	
	\bibitem{Reigstad2015}
	G.~A. Reigstad.
	\newblock Existence and uniqueness of solutions to the generalized {R}iemann
	problem for isentropic flow.
	\newblock {\em SIAM J. Appl. Math.}, 75(2):679--702, 2015.
	
	\bibitem{RochouxCollin2018}
	M.~C. Rochoux, A.~Collin, C.~Zhang, A.~Trouv\'{e}, D.~Lucor, and P.~Moireau.
	\newblock Front shape similarity measure for shape-oriented sensitivity
	analysis and data assimilation for eikonal equation.
	\newblock In {\em C{EMRACS} 2016---numerical challenges in parallel scientific
		computing}, volume~63 of {\em ESAIM Proc. Surveys}, pages 258--279. EDP Sci.,
	Les Ulis, 2018.
	
	\bibitem{YuBayenKrstic2019}
	H.~Yu, A.~M. Bayen, and M.~Krstic.
	\newblock Boundary observer for congested freeway traffic state estimation via
	{A}w-{R}ascle-{Z}hang model.
	\newblock {\em IFAC-PapersOnLine}, 52(2):183--188, 2019.
	
	\bibitem{Zuazua1988}
	E.~Zuazua.
	\newblock Stability and decay for a class of nonlinear hyperbolic problems.
	\newblock {\em Asymptotic Anal.}, 1(2):161--185, 1988.
	
\end{thebibliography}
\end{document}